\documentclass[11pt]{ip-journal}

\usepackage{rotating,diagrams,mathabx,epsfig,amsthm,amssymb,color}
 \usepackage{tikz}\usetikzlibrary{matrix,arrows}
 \usepackage{amsrefs}
\usepackage{amsfonts,amsmath}
\newtheorem{theorem}{Theorem}[section]
\newtheorem{prop}[theorem]{Proposition}
\newtheorem{lemma}[theorem]{Lemma}
\newtheorem{definition}[theorem]{Definition}
\newtheorem{corollary}[theorem]{Corollary}
\newtheorem{remark}[theorem]{Remark}

\newtheorem*{prop*}{Proposition}

\newcommand{\Ta}{\mathbb{T}_\alpha}
\newcommand{\Tb}{\mathbb{T}_\beta}
\newcommand{\Tg}{\mathbb{T}_\gamma}
\newcommand{\Td}{\mathbb{T}_\delta}
\newcommand{\s}{\mathfrak{s}}

\renewcommand{\tt}{\mathfrak{t}}

\newcommand{\lk}{\mbox{lk}}
\newcommand{\ds}{\displaystyle}

\newcommand{\btau}{{\bar{\tau}}}

\newcommand{\spinc}{{\mbox{spin$^c$} }}

\newcommand{\zee}{\mathbb{Z}}

\newcommand{\cue}{\mathbb{Q}}
\newcommand{\cee}{\mathbb{C}}

\newcommand{\K}{{\mathcal K}}

\renewcommand{\O}{{\mbox{\normalfont\gothfamily O}}}

\renewcommand{\L}{{\mathcal L}}
\newcommand{\M}{{\mathcal M}}

\newcommand{\F}{{\mathbb F}}

\newcommand{\XX}{{\mathbb X}}

\newcommand{\hfhat}{\widehat{HF}}
\newcommand{\cfhat}{\widehat{CF}}
\newcommand{\ahat}{\widehat{A}}
\newcommand{\Xhat}{\widehat{X}}
\newcommand{\sdgs}{\s_{\rm\sc DGS}}
\newcommand{\jdgs}{J_{\rm\sc DGS}}
\newcommand{\xidgs}{\xi_{\rm\sc DGS}}
\newcommand{\jhat}{\widehat{J}}
\newcommand{\Akq}{A_{\lfloor\frac{k}{q}\rfloor}}
\newcommand{\dvert}{\partial^{\mathrm{vert}}}
\newcommand{\dhorz}{\partial^{\mathrm{horz}}}

\newcommand{\ts}{\textstyle}
\newcommand{\ul}{\underline}

\newcommand{\tSigma}{\widetilde{\Sigma}}

\newcommand{\tW}{\widetilde{W}}
\newcommand{\tlambda}{{\tilde{\lambda}}}

\newcommand{\cptwobar}{\overline{\cee P}^2}

\DeclareMathOperator{\x}{\rm\bf x}
\DeclareMathOperator{\tx}{\rm\bf\tilde{x}}
\DeclareMathOperator{\y}{\rm\bf y}
\DeclareMathOperator{\U}{\rm\bf u}
\DeclareMathOperator{\V}{\rm\bf v}
\DeclareMathOperator{\WW}{\rm\bf w}
\DeclareMathOperator{\tb}{tb}
\DeclareMathOperator{\rot}{rot}
\DeclareMathOperator{\Sym}{Sym}
\DeclareMathOperator{\Spinc}{{\rm Spin}^c}
\DeclareMathOperator{\srobs}{\s_{\rm\sc ROBS}}

\newcommand{\tc}{\tilde{c}}

\newcommand{\MM}{\mbox{\bf M}}
\renewcommand{\AA}{\mathbb A}
\newcommand{\BB}{\mathbb B}

\renewcommand{\P}{{\mathcal P}}

\newcommand{\aalpha}{\mbox{\boldmath $\alpha$}}
\newcommand{\bbeta}{\mbox{\boldmath $\beta$}}
\newcommand{\ggamma}{\mbox{\boldmath $\gamma$}}
\newcommand{\taalpha}{\mbox{\boldmath $\tilde\alpha$}}
\newcommand{\tbbeta}{\mbox{\boldmath $\tilde\beta$}}
\newcommand{\tggamma}{\mbox{\boldmath $\tilde\gamma$}}
\newcommand{\ddelta}{\mbox{\boldmath $\delta$}}

\usepackage{graphicx,amsfonts}

\newarrow{Dashto}{}{dash}{}{dash}>

\begin{document}

\title[Naturality under positive contact surgery]{Naturality of Heegaard Floer invariants under positive rational contact surgery}
\author{Thomas E. Mark}
\address{Department of Mathematics, University of Virginia}
\email{tmark@virginia.edu	}
\author{B\"ulent Tosun}
\address{Department of Mathematics, University of Alabama}
\email{btosun@ua.edu	}

\begin{abstract} For a nullhomologous Legendrian knot in a closed contact 3-manifold $Y$ we consider a contact structure obtained by positive rational contact surgery. We prove that in this situation the Heegaard Floer contact invariant of $Y$ is mapped by a surgery cobordism to the contact invariant of the result of contact surgery, and we characterize the \spinc structure on the cobordism that induces the relevant map. As a consequence we determine necessary and sufficient conditions for the nonvanishing of the contact invariant after rational surgery on a Legendrian knot in the standard 3-sphere, generalizing previous results of Lisca-Stipsicz and Golla. In fact our methods allow direct calculation of the contact invariant in terms of the rational surgery mapping cone of Ozsv\'ath and Szab\'o. The proof involves a construction called reducible open book surgery, which reduces in special cases to the capping-off construction studied by Baldwin.
\end{abstract}

\maketitle

\section{Introduction}

One of the fundamental outstanding problems in 3-dimensional contact topology is the determination of which (closed, oriented) 3-manifolds admit tight contact structures. The question has been resolved in many cases, e.g., by work of Eliashberg-Thurston \cite{ET} and Gabai \cite{gabai}, any 3-manifold with nontrivial second homology admits a tight contact structure. Lisca and Stipsicz \cite{LStightseifert} determined exactly which Seifert 3-manifolds have tight contact structures. However, the situation for hyperbolic rational homology spheres is still largely open. 

Any closed oriented 3-manifold admits contact structures, and any contact structure $\xi$ can be described by contact surgery on a Legendrian link in $S^3$ \cite{dinggeiges2004}. If such a description for $\xi$ can be obtained for which all contact surgery coefficients are negative, then $\xi$ is Stein fillable and hence tight. It is therefore natural to consider the extent to which contact surgery with positive coefficients results in tight contact structures; it is this question that motivates this paper.

As a proxy for tightness of a contact structure we consider the nonvanishing of the Heegaard Floer contact invariant $c(\xi)$, which is a strictly stronger condition. Our main result is a naturality property for this invariant under positive rational contact surgery, generalizing a well-known result in the case of contact $+1$ surgery \cite{LS2004}. To state it, we recall a bit of the formal structure of Heegaard Floer invariants. 

For a closed oriented 3-manifold $Y$, the Heegaard Floer homology of $Y$ is an abelian group that is the homology of a finitely-generated chain complex $\widehat{CF}(Y)$ of free abelian groups. Forming the tensor product with the field $\F = \zee/2\zee$ we get a chain complex of $\F$-vector spaces whose homology is the Heegaard Floer homology $\hfhat(Y) = \hfhat(Y;\F)$ with coefficients in $\F$, which will be our exclusive concern. There is always a decomposition of $\hfhat(Y)$ as a direct sum of Heegaard Floer groups $\hfhat(Y,\tt)$ associated to \spinc structures $\tt$ on $Y$. Moreover, for a connected sum of 3-manifolds, we have a K\"unneth decomposition $\hfhat(Y'\# Y'', \tt'\#\tt'') \cong \hfhat(Y',\tt')\otimes \hfhat(Y'',\tt'')$. To any cobordism $W: Y_1\to Y_2$ between closed oriented 3-manifolds, i.e., a compact oriented 4-manifold with boundary $\partial W = -Y_1 \cup Y_2$, equipped with a \spinc structure $\s$, is associated a homomorphism $F_{W,\s}: \hfhat(Y_1,\tt_1) \to \hfhat(Y_2,\tt_2)$ where $\tt_i = \s|_{Y_i}$.

If $\xi$ is a contact structure on $Y$, the Heegaard Floer {\it contact invariant} of $\xi$ is an element $c(\xi) \in \hfhat(-Y, \tt_\xi)$, where $-Y$ is $Y$ equipped with the other orientation and $\tt_\xi$ is the \spinc structure associated to the contact structure $\xi$. It is a basic result of Ozsv\'ath and Szab\'o, to whom the theory of Heegaard Floer homology and the construction of $c(\xi)$ is due, that $c(\xi) = 0$ if $\xi$ is overtwisted \cite{OScontact}.

Now suppose $K\subset Y$ is an oriented nullhomologous knot and fix a Seifert surface $S$ for $K$. For an integer $n$, there is a cobordism from $Y$ to the result $Y_n(K)$ of $n$-framed surgery along $K$, consisting of a single 2-handle attached with framing $n$. More generally let $p/q$ be a rational number with $p,q$ relatively prime and $q>0$, and write $p = mq - r$ for integers $m$ and $r$ with $0\leq r < q$. Then there is a ``rational surgery cobordism'' $W: Y\# {-L(q,r)}\to Y_{p/q}(K)$, where we use the convention that $L(q,r)$ is the lens space obtained by $-q/r$ surgery on the unknot in $S^3$. Indeed, we form $Y\# {-L(q,r)}$ by performing $q/r$ surgery on a meridian of $K\subset Y$, and then $W$ is given by attaching a 2-handle along the image of $K$ after this surgery with an appropriate choice of framing (in terms of a Kirby picture, one would call it framing $m$). Observe that $H_2(W,Y\# {-L(q,r)};\zee) \cong \zee$, generated by the homology class of the core of the 2-handle which we write as $[F]$. We orient the core $F$ such that the induced orientation on $K$ is opposite to that induced by the Seifert surface $S$. We also have 
\[
H_2(W;\zee) / H_2(Y\# {-L(q,r)};\zee) \cong \zee,
\]
generated by a class $[\widetilde{S}]\in H_2(W;\zee)$ whose sign is fixed by the requirement that $[\widetilde{S}]\mapsto q[F]$ under inclusion. There is some ambiguity in the characterization of $[\widetilde{S}]$ by these conditions if $H_2(Y;\zee)\neq 0$; one can specify the class $[\widetilde{S}]$ uniquely by using the Seifert surface $S$ to construct a surface closing off---in $Y\#{-L(q,r)}$---$q$ parallel copies of the core $F$, but for the statement and applications of our theorem we will not need this.

Equip $Y$ with a contact structure $\xi$, and assume $\K$ is a Legendrian representative of the smooth knot $K$. We have two ``classical invariants'' of $\K$, the Thurston-Bennequin number $\tb(\K)$ and the rotation number $\rot(\K)$, where the latter depends on the orientation of $K$ and in general on the choice of $S$. The dependence on $S$ is eliminated if we suppose that the Euler class of $\xi$, or equivalently the first Chern class of $\tt_\xi$, is torsion. For a rational number $x/y \neq 0$, there is a notion of contact $x/y$ surgery along $\K$, which provides a contact structure $\xi_{x/y}$ on the 3-manifold obtained by surgery on $K$ with coefficient $p/q = \tb(\K) + x/y$. Note that $q = y$ and $p = x + y\,\tb(\K)$. When $x \neq \pm 1$ one can obtain several different contact structures by rational contact surgery, depending on certain choices; our results concern one particular such contact structure, written $\xi^-_{x/y}$ and corresponding to choosing ``all negative stabilizations'' in the construction. For more details see Section \ref{surgerysec}, in particular Theorem \ref{DGSalgorithm} and following text.

Recall that the contact invariant $c(\xi)\in\hfhat(-Y)$ satisfies a certain naturality property under contact $+1$ surgeries, in the sense that if $W: Y\to Y_{{\rm tb}(\K)+1}(K)$ is the surgery cobordism then there exists a \spinc structure $\s$ on $-W$ such that the induced homomorphism $\hfhat(-Y)\to\hfhat(-Y_{{\rm tb}(\K)+1}(K))$ carries $c(\xi)$ to $c(\xi_{+1})$. Our first main result generalizes this property to any positive rational contact surgery.

\begin{theorem}\label{naturalitythmintro} Let $\K$ be an oriented nullhomologous Legendrian knot in a contact 3-manifold $(Y,\xi)$, and let $0< \frac{x}{y}\in \cue$ be a contact surgery coefficient corresponding to smooth surgery coefficient $\frac{p}{q} = \tb(\K) + \frac{x}{y}$. Let $W: Y\#{-L(q,r)}\to Y_{p/q}(K)$ be the corresponding rational surgery cobordism, where $p = mq -r$ as above, and consider the contact structure $\xi_{x/y}^-$ on $Y_{p/q}(K)$.
\begin{enumerate}
\item There exists a \spinc structure $\s$ on $W$ and a generator $\tilde{c}\in \hfhat(L(q,r))$ such that the homomorphism 
\[
F_{-W,\s}: \hfhat(-Y\# L(q,r))\to \hfhat(-Y_{p/q}(K))
\]
 induced by $W$ with its orientation reversed satisfies
\[
F_{-W,\s}(c(\xi)\otimes \tilde{c}) = c(\xi^-_{x/y}).
\]
\item Assume also that $\xi$ and $\xi_{x/y}^-$ have torsion first Chern class. Then $\s$ has the property that
\[
\pm\langle c_1(\s), [\widetilde{S}]\rangle = p + (\rot(\K) - \tb(\K))q - 1.
\]
\end{enumerate}
\end{theorem}

Note that part (2) of the theorem characterizes the \spinc structure $\s$ uniquely up to conjugation under the given hypotheses. A version of this theorem for positive integer contact surgeries is implicit in \cite{LS2011}.

A couple of remarks are in order about the class $\tilde{c}\in \hfhat(L(q,r))$. First, this class depends only on $q$ and $r$, not on $Y$, $K$, or $\xi$. Second, though $\tilde{c}$ is homogeneous with respect to the decomposition of $\hfhat(L(q,r))$ along \spinc structures, it is {\it not}, as might be supposed, always the contact invariant of a contact structure on $-L(q,r)$. It may more naturally be interpreted as an invariant associated to a certain transverse knot in the latter manifold (the knot $O_{q/r}$ of Section \ref{ratsurgsec}), though we neither need nor pursue this interpretation. For our purposes here, it suffices to observe that for each \spinc structure $\tt$ on $L(q,r)$, we have $\hfhat(L(q,r),\tt) = \F$ and therefore the generator $\tilde{c}$ is specified uniquely by its associated \spinc structure.

Theorem \ref{naturalitythmintro} is proved by use of a construction we call ``reducible open book surgery'', which can be seen as a generalization of the operation on open book decompositions known as ``capping off'' a boundary component. From this point of view the theorem above follows from a generalization of a theorem of Baldwin on the behavior of the contact invariant under capping off \cite{baldwincap}. Our generalization is given as Theorem \ref{naturalitythm} below.

We now specialize to Legendrian knots in the standard contact 3-sphere. Theorem \ref{naturalitythmintro} allows us to determine exactly when a positive contact surgery yields a contact structure with nonvanishing Heegaard Floer invariant. The answer depends on certain other Heegaard-Floer-theoretic invariants for the knot $K\subset S^3$, {\it viz.}:
\begin{itemize}
\item The integer $\tau(K)$ discovered by Ozsv\'ath and Szab\'o \cite{FourBall} and by Rasmussen \cite{RasThesis}. If $-K$ denotes the mirror of $K$, it is known that $\tau(-K) = -\tau(K)$; note that $\tau(K)$ is independent of the orientation of $K$. It was proved by Plamenevskaya \cite{Olga2004} that in the standard contact structure, if $\K$ is an oriented Legendrian knot isotopic to $K$ then
\begin{equation}
\tb(\K) + |\rot(\K)| \leq 2\tau(K) - 1.
\label{plamresult}
\end{equation}
\item The concordance invariant $\epsilon(K)\in \{-1,0,1\}$ introduced by Hom \cite{Hom1}. It was shown in \cite{Hom1} that if $K$ is slice then $\epsilon(K) = 0$, and if $\epsilon(K) = 0$ then $\tau(K) = 0$.
\end{itemize}

\begin{theorem}\label{cinvariantthm} Let $\K$ be an oriented Legendrian knot in $S^3$ and $0< \frac{x}{y}\in \cue$. Write $(Y_{x/y},\xi^-_{x/y})$ for the contact manifold obtained by contact $\frac{x}{y}$ surgery on $\K$, in which all stabilizations are chosen to be negative. Let $c(\xi_{x/y}^-) \in \hfhat(-Y_{x/y})$ be the Ozsv\'ath-Szab\'o contact invariant of $\xi_{x/y}^-$, and write $K$ for the smooth knot type underlying $\K$. Finally, let $\frac{p}{q} = \frac{x}{y} + \tb(\K)$ be the corresponding smooth surgery coefficient.

\begin{enumerate}
\item If $\tb(\K) - \rot(\K) < 2\tau(K) - 1$, then $c(\xi^-_{x/y}) = 0$.
\item Suppose $\tb(\K) - \rot(\K) = 2\tau(K) - 1$. 
\begin{enumerate}
\item If $\epsilon(K) = 1$, then $c(\xi_{x/y}^-) \neq 0$ if and only if $\frac{p}{q} > 2\tau(K) - 1$.
\item If $\epsilon(K) = 0$, then $c(\xi_{x/y}^-) \neq 0$ if and only if $\frac{p}{q} \geq 2\tau(K)$.
\item If $\epsilon(K) = -1$, then $c(\xi_{x/y}^-) = 0$.
\end{enumerate}
\end{enumerate}
\end{theorem}

Clearly, if we are to have $c(\xi^-_{x/y})\neq 0$, we must orient $K$ such that $\rot(\K)\leq 0$. Alternatively, one can consider $\xi^+_{x/y}$, obtained by all positive stabilizations; indeed, it is easy to deduce versions of Theorems \ref{naturalitythmintro} and \ref{cinvariantthm} for $\xi^+$ using the fact that $\xi_{x/y}^-(\K) = \xi_{x/y}^+(\overline\K)$, where $\overline\K$ is $\K$ with the opposite strand orientation (cf. \cite[Lemma 2.2]{LS2011}).

A result analogous to Theorem \ref{cinvariantthm} for the case of integer surgeries was obtained by Golla \cite{golla}; note that in the integer case, the distinction between cases 2(a) and 2(b) of the theorem does not arise. 
Golla also obtains some partial results for rational surgeries as \cite[Proposition 6.18]{golla}, which shows that in case 2(a) of Theorem \ref{cinvariantthm}, the condition $\frac{p}{q} > 2\tau(K) -1$ suffices to give the existence of {\it some} tight structure on $Y_{x/y}$. Our methods treat integer and rational surgeries simultaneously, and give information specific to the contact structure $\xi_{x/y}^-$.

As an application of Theorem \ref{cinvariantthm}, we prove the following generalization of a result of Lisca and Stipsicz \cite[Theorem 1.1]{LStight1}. Again, the existence portion of the result can also be deduced from the work of Golla \cite{golla}.

\begin{theorem} Let $K\subset S^3$ be a knot with slice genus $g_s(K)>0$ that admits a Legendrian representative $\K$ with $\tb(\K)+|\rot(\K)| = 2g_s(K) - 1$. Then the manifold $S^3_{p/q}(K)$ obtained by smooth $\frac{p}{q}$ surgery along $K$ admits a tight contact structure, for every $\frac{p}{q}\not\in [2g_s(K) -1-|\rot(\K)|, 2g_s(K) -1]$. 

In particular if $\K$ is oriented so that $\rot(\K)\leq 0$, then the contact structure $\xi^-_{x/y}(\K)$ is tight for $\frac{x}{y}$ the contact surgery coefficient corresponding to smooth $\frac{p}{q}$ surgery as above.
\end{theorem}


\begin{proof} Recall that $|\tau(K)|\leq g_s(K)$, so by \eqref{plamresult} we must have $\tau(K) = g_s(K)$. We consider contact $\frac{x}{y}$ surgery along the representative $\K$.  

If $\frac{p}{q}< 2g_s(K)-1-|\rot(\K)|$ then the contact surgery coefficient is $\frac{x}{y} = \frac{p}{q} - \tb(\K) < 0$. Any negative contact surgery can be realized by a sequence of contact $-1$ surgeries \cite{DGS}, and such surgeries result in a Stein fillable, hence tight, contact structure.

If $\frac{p}{q} > 2g_s(K)-1$ then we orient $\K$ such that $\rot(\K)\leq 0$ and consider the contact structure $\xi^-_{x/y}$ given by contact surgery with coefficient $\frac{x}{y} = \frac{p}{q}-\tb(\K) = \frac{p}{q} -(2g_s(K)-1-|\rot(\K)|)>0$. By a result of Hom \cite[p. 288]{Hom1}, since $\tau(K) = g_s(K)$ we have $\epsilon(K) = \mbox{sign}(\tau(K)) = 1$. The contact structure $\xi_{x/y}^-$ is then tight by Theorem \ref{cinvariantthm}.
\end{proof}

A similar argument shows that if $K$ is a slice knot with a Legendrian representative $\K$ satisfying $\tb(\K) + |\rot(\K)| = 2g_s(K) -1 = -1$, then $S^3_{p/q}(K)$ admits a tight contact structure for all $\frac{p}{q}\not\in [-1-|\rot(\K)|, 0)$. As an example, Figure \ref{legendrian} shows a family of slice Legendrian knots $\K_n$, $n\geq 1$, each satisfying $\tb(\K_n) = -2$ and $|\rot(\K_n)| = 1$. For $n = 1$ the smooth type of $\K_n$ is the knot $8_{20}$, a hyperbolic knot, from which it is easy to see that $\K_n$ is hyperbolic for all but at most finitely many $n$. It then follows that $S_{p/q}(\K_n)$ admits a tight contact structure for all $\frac{p}{q}\not\in [-2,0)$, in particular by taking $p = 1$ we obtain (for each $n$) an infinite family of hyperbolic integer homology spheres with tight contact structures.

\begin{figure}[t]
\def\svgwidth{2.5in}
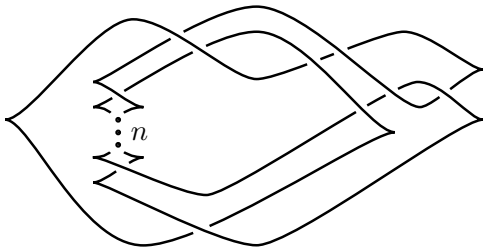
\caption{\label{legendrian}The Legendrian $\K_n$, with $n$ crossings in the indicated twist region. A ribbon move across the twisted band shows $\K_n$ is slice. Note that for $n>1$, $\K_n$ generally does not have maximal Thurston-Bennequin invariant in its smooth isotopy class.}
\end{figure}

Our techniques also allow specification of the class $c(\xi_{x/y}^-)$ more precisely than in Theorem \ref{cinvariantthm}, thanks to the second part of Theorem \ref{naturalitythmintro}. To understand this, we briefly recall a method due to Ozsv\'ath and Szab\'o for calculating the Heegaard Floer homology of the result of rational surgery along $K\subset S^3$ (see \cite{OSratsurg}). 

Given $K\subset S^3$, Ozsv\'ath and Szab\'o define a filtration of the chain complex $B := \widehat{CF}(S^3)$ and, with some additional machinery, produce a chain complex denoted $CFK^\infty(K)$. From this complex one obtains a sequence of subquotient complexes $A_s = A_s(K)$, $s \in \zee$, with certain properties:
\begin{itemize}
\item For $|s|\gg 0$, we have $A_s \simeq B$, where $\simeq$ denotes chain homotopy equivalence.
\item For any $n\geq 2g(K) -1$ and $|s|\leq \frac{n}{2}$, there is an isomorphism
\[
A_s \cong \widehat{CF}(S^3_n(K), \tt_s),
\]
where $\tt_s$ is the \spinc structure on $S^3_n(K)$ obtained as follows. Let $W_n: S^3\to S^3_n(K)$ be the surgery cobordism and $\s_s\in\Spinc(W_n)$ the \spinc structure characterized by
\[
\langle c_1(\s_s), [\widetilde{S}]\rangle + n = 2s
\]
where $[\widetilde{S}]$ is obtained from a Seifert surface capped off in $W_n$ as before. Then $\tt_s = {\s_s}|_{S^3_n(K)}$.
\item For each $s$ there are chain maps $v_s, h_s: A_s\to B$. If $v_{s*}$ and $h_{s*}$ are the corresponding maps in homology, we have that for $s\gg 0$, $v_{s*}$ is an isomorphism while $h_{s*}$ is trivial, and likewise $v_{-s*}$ is trivial while $h_{-s*}$ is an isomorphism.
\end{itemize}

Now suppose $\frac{p}{q}\in\cue$ is a rational number, with $q>0$ as before. Define a chain complex $\XX_{p/q}(K)$ as follows. First let 
\[
\AA_{p/q} = \bigoplus_{k\in\zee} (k, A_{\lfloor\frac{k}{q}\rfloor}) \quad\mbox{and}\quad \BB_{p/q} = \bigoplus_{k\in\zee} (k, B).
\]
Here the entry ``$k$'' in $(k, A_{\lfloor\frac{k}{q}\rfloor})$ indexes the direct sum, as in \cite{OSratsurg}. Define a chain map $D_{p/q}: \AA_{p/q}\to \BB_{p/q}$ by
\[ 
D_{p/q}(k, x) = (k, v(x)) + (k + p, h(x)).
\]
Here and to follow, we omit the subscript on the maps $v$ and $h$ whenever their domain is clear from context.

Finally, let $\XX_{p/q}(K)$ be the mapping cone of $D_{p/q}$. This mapping cone gives the Floer homology of the result of $p/q$ surgery along $K$, and also determines the maps induced by the surgery cobordism, according to the following.

\begin{theorem}\label{ratsurgformula} For any knot $K\subset S^3$ and any rational number $\frac{p}{q}\in \cue$, we have:
\begin{enumerate}
\item[1.] (Ozsv\'ath-Szab\'o \cite{OSratsurg}) There is a chain homotopy equivalence $\Phi: \XX_{p/q}(K) \to \cfhat(S^3_{p/q}(K))$, in particular the homology of $\XX_{p/q}(K)$ is isomorphic to $\hfhat(S^3_{p/q}(K))$.

\item[2.] Let $W_{p/q}: S^3 \#{-L(q,r)}\to S^3_{p/q}(K)$ be the rational surgery cobordism, where $p = mq-r$, and let $\s\in \Spinc(W_{p/q})$. Write $[\widetilde{S}]$ for the generator of $H_2(W_{p/q})/H_2(-L(q,r))$ as above. Then the map in Floer homology induced by $\s$ corresponds via $\Phi$ to the inclusion of $(k, B)$ in $\XX_{p/q}(K)$, where $k$ is determined by 
\begin{equation}\label{c1determination}
\langle c_1(\s), [\widetilde{S}]\rangle + p + q -1 = 2k.
\end{equation}
\end{enumerate}
\end{theorem}

For the case of integer surgeries, the analog of the second claim is spelled out in \cite{OSintsurg}. The version for rational surgeries is more complicated but essentially similar, however it  (particularly formula \eqref{c1determination}) does not seem to appear in the literature. We give a proof of the second part of the theorem in Section \ref{section5} (Corollary \ref{inclusioncor}).
 
Comparing the second parts of Theorems \ref{naturalitythmintro} and \ref{ratsurgformula}, bearing in mind that the necessary orientation reversal is equivalent to replacing $K$ by $-K$ and $p$ by $-p$ in Theorem \ref{ratsurgformula}, gives the following.

\begin{corollary}\label{cor1} Let $K\subset S^3$ be a knot with Legendrian representative $\K$, fix $0< \frac{x}{y}\in \cue$, and let $\frac{p}{q} = \tb(\K) + \frac{x}{y}$. Then the contact invariant $c(\xi_{x/y}^-)\in \hfhat(-S^3_{p/q}(K))$ is equal (up to conjugation) to the generator of the image in homology of the map given by the inclusion
\[
(k,B) \hookrightarrow \XX_{-p/q}(-K),
\]
where $k$ satisfies
\begin{equation}
2k = (\rot(\K) - \tb(\K) + 1)q -2.
\label{kchar}
\end{equation}
\end{corollary}

This corollary is the essential step in the proof of Theorem \ref{cinvariantthm}. Note also that it gives a direct description of the contact invariant $c(\xi_{x/y}^-)$ in terms of the mapping cone formula for Floer homology.

We now give a concrete example of the application of Corollary \ref{cor1} to calculation of a contact invariant. The example illustrates the ``typical'' situation in which one obtains a nonvanishing invariant.

Let $K$ be the $(1,2)$ cable of the right-handed trefoil knot (here the $(p,q)$ cable of
a knot type $K'$ is the knot type obtained by taking the curve that traverses the meridional direction $p$ times and the longitudinal direction $q$ times on the boundary of a tubular neighborhood of a representative of $K'$). This knot has $\tau(K) = g(K) = 2$, and admits a Legendrian representative $\K$ with $\tb(\K) = 2$ and $\rot(\K) = -1$ (see \cite{ELT2012}). In particular we have $\tb(\K) - \rot(\K) = 2\tau(K) -1$, and hence we expect to find nonvanishing contact invariant $c(\xi_{x/y}^-)$ for all $\frac{x}{y}$ with corresponding smooth surgery coefficient satisfying $\frac{x}{y} + \tb(\K) > 2\tau(K) - 1 = 3$. Let us choose $\frac{x}{y} = \frac{3}{2}$, corresponding to smooth surgery with coefficient $\frac{7}{2}$. According to Corollary \ref{cor1} the contact invariant $c(\xi_{3/2}^-)$ is given by the image in homology of the inclusion of $(k,B)$ in the mapping cone, where for our data $k = -3$. 

The knot Floer complex for $K$ was determined by Hedden \cite[Proposition 3.2.2]{heddenthesis}. For the experts, the results can be summarized diagrammatically and without explanation as in Figure \ref{fig:small}.

\begin{figure}[h!]
\begin{center}
  \includegraphics[width=4.5in]{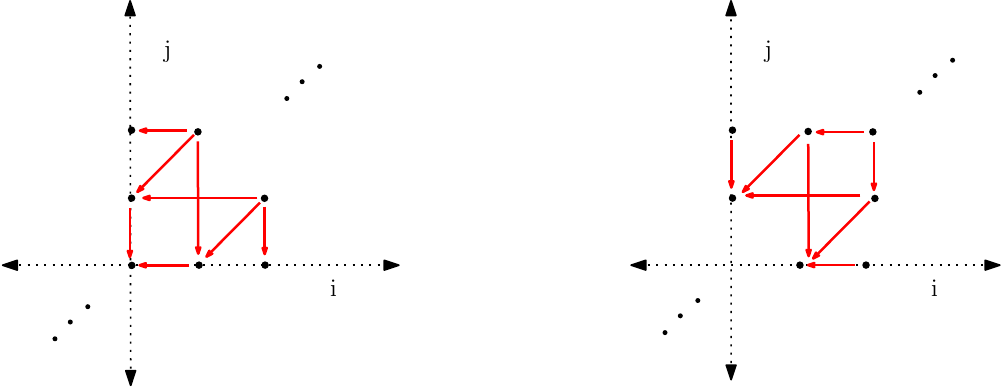}
 \caption{$CFK^{\infty}(S^3,K)$ (left) and $CFK^{\infty}(S^3, - K)$ (right)}
  \label{fig:small}
\end{center}
\end{figure}

For our purposes, it more than suffices to know the following, where $A_s$ refers to the subquotient complex obtained from the complex for the mirror knot $-K$:
\[
H_*(A_s) \cong \left\{\begin{array}{ll} \F & \mbox{$|s|\geq 2$}\\ \F^3 & |s| = 1 \\ \F^5 & s = 0 \end{array}\right. \qquad \mbox{$v_{s*}$ and $h_{-s*}$ are} \left\{\begin{array}{ll} 0 & \mbox{ if $s \leq -2$} \\ \mbox{onto} & \mbox{ if $s \geq -1$}\end{array}\right.
\]

A portion of the cone $\XX_{-7/2}(-K)$ can then be described as follows:
\[
\begin{diagram}[height=.5cm]
\cdots \hspace*{1em}& (-4, A_{-2})\hspace*{1em} & (-3, A_{-2})\hspace*{1em} & (-2, A_{-1}) & \cdots & (3, A_1) \hspace*{1em}& (4, A_2)\hspace*{1em} & (5, A_2) & \cdots\\
          & \ldTo           &                          &           &           &   \ldTo(4,4)^{h_* = 0}   &                   & & \\
          &                   &  \dTo^{v_* = 0} &           &            &                                & \dTo_\sim & & \\
          &                   &                          &           &            &                                 &                  &\ldTo & \\
\cdots \hspace*{1em} & B          & (-3, B)        & B          & \cdots &B     & B  & B     & \cdots 
\end{diagram}
\]
Since the homology of $(-3,B)$ does not interact with the map induced on homology by $D_{-7/2}$ it clearly survives to the homology of $\XX_{-7/2}(-K)$, which proves the nonvanishing of $c(\xi_{3/2}^-)$. A couple of other remarks:
\begin{itemize}
\item By computing the homology of $\XX_{-7/2}(-K)$, we can see that the class $c(\xi_{3/2}^-)$ is the generator of $\hfhat(-S^3_{7/2}(K), \tt_\xi) \cong \F$. In general the technique allows explicit description of the contact invariant as an element of its corresponding Floer group, via the mapping cone as above (at least, modulo automorphisms). 
\item Varying the numerator of the surgery parameter $x/y$, or equivalently $p/q$, has the effect of adjusting the source of the arrow labeled as $h_*$ in the diagram above. In particular the reader can check that the homology generator of $(-3,B)$ vanishes in $H_*(\XX_{-p/2}(-K))$ for any $p< 7$, and survives whenever $p\geq 7$, the transition corresponding to the (non) vanishing of the relevant map $h_*$ (of course, Corollary \ref{cor1} applies only for surgery coefficients $p/q >2$). Similarly, if a Legendrian representative for $K$ is chosen with a smaller value of $\tb - \rot$, we are led to consider the inclusion of $(k, B)$ for $k>-3$. The homology generator of this group is in the image of a $v_*$ and hence vanishes.
\item For this choice of $K$, the surgery manifold $S^3_{7/2}(K)$ is a Seifert fibered space. In particular, it was known previously by work of Lisca-Stipsicz to admit a tight contact structure (likewise, Golla's results apply to show such a structure exists). However by our method we obtain an explicit, relatively simple surgery description for such a structure.  
\end{itemize}

In the next section we describe the key geometric construction that leads to our results, called reducible open book surgery, and prove a naturality property for the contact invariant under this operation. Section \ref{surgerysec} shows how to apply reducible open book surgery to deduce Theorem \ref{naturalitythmintro}. The proof of Theorem \ref{cinvariantthm} is given in Section \ref{lastsection}, which may be read independently of the preceding sections. In the last section we prove the second part of Theorem \ref{ratsurgformula}.

\subsection*{Acknowledgements} We are grateful to \c{C}a\u{g}r\i\ Karakurt and Katherine Raoux for useful discussions and comments, and to John Etnyre for his interest and support of this project. Thanks are also due to the referees for their careful reading and many corrections. The first author was supported in part by NSF grant DMS-1309212 and a grant from the Simons Foundation (523795, TM). The second author was supported in part by an AMS-Simons travel grant, and also thanks the Max Planck Institute for Mathematics for their hospitality during the summer of 2014.
 
\section{Reducible Open Book Surgery}\label{robssection}

Let $Y$ be a closed oriented 3-manifold equipped with an open book decomposition $(S, \phi)$. Recall that this means $S$ is a compact oriented surface with boundary, and the monodromy $\phi$ is an orientation preserving diffeomorphism fixing a neighborhood of $\partial S$. Moreover, we are given a diffeomorphism $Y\cong (S\times [0,1])/ \sim$, where the equivalence relation identifies $(x,1)$ with $(\phi(x), 0)$ for all $x\in S$, and also $(x, t)$ with $(x, t')$ for all $t,t'$ and all $x\in\partial S$. 

A surface diffeomorphism is {\it reducible} if a power preserves an essential (multi-) curve on the surface. Here we will be interested particularly in the case that $\gamma \subset S$ is a simple closed curve, and $\phi$ fixes $\gamma$. Moreover, we assume $\gamma$ separates $S$ into two subsurfaces, each containing at least one component of $\partial S$. In this situation, {\it reducible open book surgery} along $\gamma$ is defined to be the surgery on $\gamma$ (thought of as a knot in $Y$) with framing equal to that induced by the page on which $\gamma$ lies.

As a basic example one could consider $\gamma$ to be parallel to a boundary component of $S$, assuming $\partial S$ has at least two components. Then ``reducible'' open book surgery along $\gamma$ is equivalent to page-framed surgery along the corresponding boundary (binding) component, an operation usually called ``capping off'' the open book. 

We write $Y_\gamma$ for the result of page-framed surgery along a reducing curve $\gamma$ as above.

\begin{lemma} Let $S'\cup S''$ be the (disconnected) result of surgery along $\gamma\subset S$, thought of as an abstract surface. Write $\phi'$ and $\phi''$ for the diffeomorphisms of $S'$ and $S''$ obtained by restricting $\phi$ and then extending by the identity across the surgery disks. Then there is a diffeomorphism
\[
Y_\gamma \cong Y'\#Y'',
\]
where $Y'$ and $Y''$ are described by the open books $(S', \phi')$ and $(S'',\phi'')$. 
\end{lemma}

The proof is straightforward; we point out two items. First, since it is preserved by the monodromy the curve $\gamma$ sweeps out a torus in $Y$ (and the framing induced by the torus is the same as that induced by $S$). After page-framed surgery, the torus becomes the separating 2-sphere in the connected sum $Y'\# Y''$. Second, if $W_\gamma: Y\to Y_\gamma$ is the 2-handle cobordism corresponding to the surgery, then the cocore of the 2-handle intersects $Y_\gamma$ in a knot $K'\# K''$, where $K'\subset Y'$ is the 1-braid in the open book $(S', \phi')$ traced by the center of the capping disk in $S'$ (and similar for $K''$). We can then think of $Y$ as obtained from $Y'\cup Y''$ by a version of ``contact normal sum'' along the (transverse) knots $K'$ and $K''$, with framings induced by the corresponding open books. 

Now, by fundamental work of Thurston-Winkelnkemper \cite{TW} and Giroux \cite{giroux}, there is a correspondence between open book decompositions and contact structures on 3-manifolds. In particular, the 3-manifolds $Y$, $Y'$ and $Y''$ each carry contact structures $\xi$, $\xi'$ and $\xi''$, and the cobordism $W_\gamma$ connects the contact manifolds $(Y,\xi)$ and $(Y'\#Y'', \xi'\# \xi'')$. 

If $\gamma$ is parallel to a component of $\partial S$, then $S''$, say, is just a disk and the monodromy $\phi''$ is isotopic to the identity. Thus $(Y'', \xi'') \cong (S^3,\xi_{std})$, and $W_\gamma$ is the ``capping-off cobordism'' studied by Baldwin in \cite{baldwincap}. The main result of \cite{baldwincap} states that the map in Heegaard Floer homology induced by the reversed cap-off cobordism, $W_\gamma: -Y'\to -Y$, equipped with a particular \spinc structure, carries the contact invariant of $(Y',\xi')$ to that of $(Y,\xi)$. For the more general reducible open book surgery, recall that under the K\"unneth decomposition $\hfhat(-(Y'\# Y'')) = \hfhat(-Y')\otimes \hfhat(-Y'')$, we can write $c(\xi'\#\xi'') = c(\xi')\otimes c(\xi'')$. However, the obvious generalization of Baldwin's theorem is false, in general: there is usually {\it not} a \spinc structure on $W_\gamma$ that carries $c(\xi')\otimes c(\xi'')$ to $c(\xi)$. Nevertheless, an adaptation of Baldwin's techniques can prove a statement that suffices for our purposes. 

Note that Baldwin also considers reducible open books in the context of capping off: he observes that if $Y_1 = (S_1,\phi_1)$ and $Y_2 = (S_2, \phi_2)$ are open books each with at least two boundary components then the open book $Y = (S_1\cup S_2, \phi_1\cup\phi_2)$ obtained by gluing two boundary components of $S_1$ and $S_2$ can be realized as the result of capping off $Y_1\# Y_2$. Thus the cobordism $-Y\to -(Y_1\# Y_2)$ respects the contact class. Reducible open book surgery results in a cobordism in the other direction, $-(Y'\# Y'')\to -Y$, and behaves somewhat differently. To understand this, we recall some of the basics of Heegaard Floer theory, and the construction of $c(\xi)$ via open books following Honda, Kazez and Mati\'c \cite{HKM}. 

The Heegaard Floer chain complex for a closed oriented 3-manifold $Y$ requires for its construction a choice of Heegaard diagram: this is a triple $(\Sigma, \aalpha,\bbeta)$ where $\Sigma$ is a closed oriented surface of genus $g\geq 1$ and $\aalpha = \{\alpha_1,\ldots, \alpha_g\}$ and $\bbeta = \{\beta_1,\ldots, \beta_g\}$ are $g$-tuples of simple closed curves disjointly embedded in $\Sigma$, such that the members of each $g$-tuple are linearly independent in $H_1(\Sigma;\zee)$. The $g$-tuples determine a pair of 3-dimensional handlebodies $H_\alpha$ and $H_\beta$ uniquely specified by the requirement that each $\alpha_i$ bound an embedded disk in $H_\alpha$ and correspondingly for the $\beta_i$ in $H_\beta$. The orientations of $H_\alpha$ and $H_\beta$ are determined by the requirement $\partial H_\alpha = \Sigma  = - \partial H_\beta$, and the triple $(\Sigma, \aalpha,\bbeta)$ is a Heegaard diagram for $Y$ if there is an orientation-preserving diffeomorphism between $H_\alpha\cup_\Sigma H_\beta$ and $Y$. By isotopy of the curves in $\aalpha$ and $\bbeta$, we arrange that all the intersections between the $\alpha_i$ and $\beta_j$ are transverse double points; for the purposes of Heegaard Floer theory we must also choose a basepoint $w\in \Sigma$ in the complement of $\aalpha$ and $\bbeta$. The collection $(\Sigma, \aalpha,\bbeta,w)$ is called a {\it pointed} Heegaard diagram; we require certain admissibility conditions on this diagram that can be achieved by isotopy as well (c.f. \cite[Section 5]{OS1}).

Given a (pointed) Heegaard diagram, the Heegaard Floer chain complex has generators obtained as follows. We form the symmetric power $\Sym^g\Sigma$, being the space of unordered $g$-tuples of (not necessarily distinct) points on $\Sigma$, topologized in the natural way as a quotient of $\Sigma^g$. This space is naturally an orbifold, but it is well-known that it can be provided with the structure of a smooth $2g$-dimensional manifold by, for example, choosing a complex structure on $\Sigma$. In $\Sym^g\Sigma$ lie two $g$-dimensional tori $\Ta$ and $\Tb$, being the images of $\alpha_1\times\cdots\times \alpha_g$ and $\beta_1\times\cdots\times\beta_g$; since the $\alpha$ curves are disjoint, we see $\Ta$ is smoothly embedded (and similarly for $\Tb$). Moreover, under the admissibility conditions, the two tori intersect transversely at isolated points $\x\in \Ta\cap \Tb$ that can be described concretely as $g$-tuples $\x = \{x_1,\ldots,x_g\}$ such that each $x_j$ lies at an intersection of  $\alpha_j$ and $\beta_{\sigma(j)}$, where $\sigma$ is some permutation of $\{1,\ldots, n\}$. 

The chain complex $\cfhat(Y)$ is freely generated over $\F = \zee/2\zee$ by the intersection points $\x$. The construction of the differential is much more delicate, involving a count of holomorphic disks with boundary on the tori $\Ta$ and $\Tb$ and ``connecting'' intersection points $\x$ and $\y$ in an appropriate sense. The most we need to say for the moment is that the generators $\x$ fall into equivalence classes respected by the differential; in fact, granted the choice of basepoint $w$ one can associate a \spinc structure $\s_w(\x)$ to each generator, and there is a corresponding decomposition of chain complexes $\cfhat(Y) = \bigoplus_\s \cfhat(Y,\s)$.

Now suppose we are given an open book decomposition $(S,\phi)$ for $Y$. Following Honda, Kazez, and Mati\'c \cite{HKM}, we can then obtain a pointed Heegaard diagram for $Y$ as follows:
\begin{itemize}
\item Let $a_1,\ldots, a_n$ be a collection of properly embedded arcs in $S$ that cut $S$ into a disk, and let $b_1,\ldots,b_n$ be a set of arcs obtained by a small translation of the $a_i$  that moves the boundaries in the positively oriented direction of $\partial S$, and such that $b_i$ intersects $a_i$ transversely in a single point of $\mbox{int}(S)$.
\item The Heegaard surface $\Sigma$ is given by $S_{1/2}\cup(-S_0)$, where $S_t$ refers to the image of $S\times\{t\} \subset S\times [0,1]$ in $Y$.
\item For $i = 1,\ldots n$, the attaching circle $\alpha_i$ is equal to $a_i \times \{\frac{1}{2}\} \cup a_i\times\{0\}$, while $\beta_i$ is given by $b_i\times \{\frac{1}{2}\}\cup\phi(b_i)\times \{0\}$. We assume that all intersections between $\alpha$ and $\beta$ curves are transverse.
\item The basepoint $w$ for the diagram is placed in $S_{1/2}$, away from the regions between the arcs $a_i$ and $b_i$.
\end{itemize}
See Figure \ref{afigure}(a) below for an example, where the $\alpha$- and $\beta$- curves appear in red and blue, respectively.

We refer to a pointed Heegaard diagram constructed in this way from an open book decomposition as an {\it HKM diagram}. Reversing the roles of $\alpha$ and $\beta$ curves, we can think of $(\Sigma, \bbeta,\aalpha, w)$ as giving a Heegaard diagram for $-Y$ (which we also call an HKM diagram). In this diagram the generator $\x$ for $\cfhat(-Y)$ corresponding to the $n$ intersection points between the $a_i$ and $b_i$ on $\Sigma_{1/2}$ is a cycle---a fact which relies on the location of the basepoint---and by \cite{HKM} it represents the contact invariant $c(\xi)\in \hfhat(-Y)$.

\begin{definition}\label{strongdef} Let $Y$ be a rational homology 3-sphere. An open book decomposition $(S,\phi)$ supporting a contact structure $\xi$ on $Y$ is {\em HKM strong} if there is an HKM diagram corresponding to $(S, \phi)$ with the property that the canonical generator $\x$ is the only intersection point in its \spinc structure.
\end{definition}

Since in an HKM diagram the canonical intersection point $\x$ lies in the \spinc structure $\s_\xi$ associated to the contact structure, a necessary condition for a contact rational homology sphere $(Y,\xi)$ to admit an HKM strong open book decomposition is that the Floer homology $\hfhat(-Y, \s_\xi)$ is isomorphic to $\F$ (and $\x$ represents the generator of this module). In fact, by moving the basepoint in the HKM diagram (c.f. \cite[Lemma 2.19]{OS1}), one sees that for every \spinc structure $\tt$, the group $\hfhat(-Y,\tt)$ is isomorphic to $\F$: thus if $Y$ is a rational homology sphere admitting an HKM strong open book decomposition then necessarily $Y$ is an $L$-space. We will see below that the standard contact structure on a lens space admits an HKM strong open book decomposition.

\begin{figure}[b]
\includegraphics[width=3in]{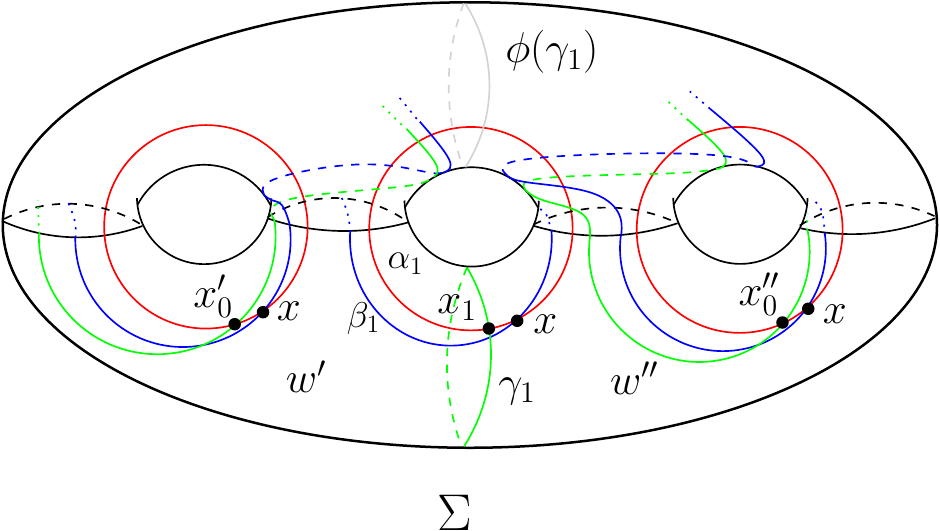} \\
(a) \vspace{2em}\\
\parbox{2.25in}{ \begin{centering}\includegraphics[width=2.25in]{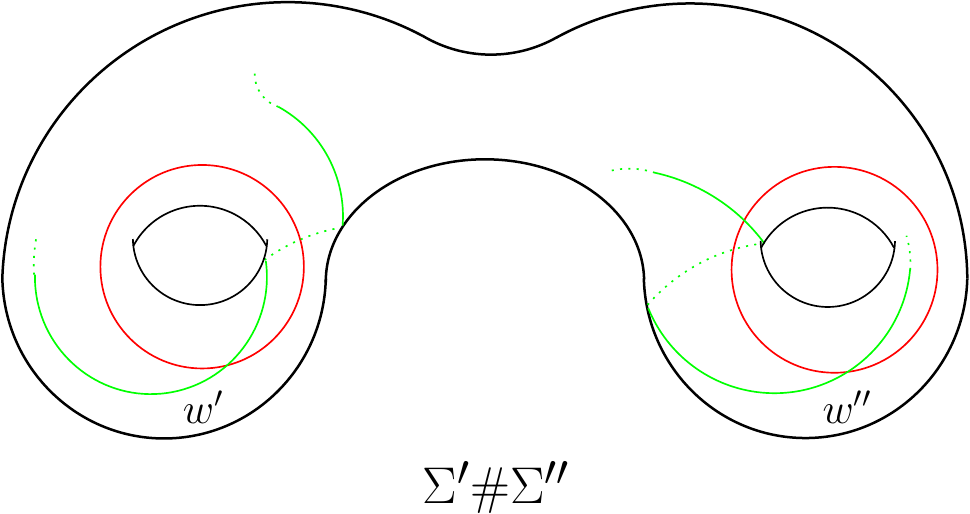} \\ (b)\\ \end{centering}} \quad\quad \parbox{2.25in}{\begin{centering}\includegraphics[width=2.25in]{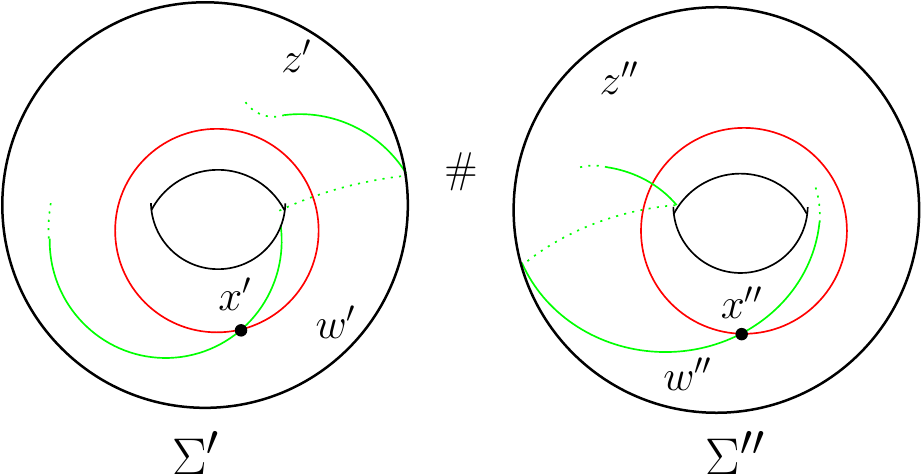}\\ (c) \\ \end{centering}}
\caption{\label{afigure} (a) shows an HKM diagram ($\alpha$ curves in red, $\beta$ in blue) associated to an open book with monodromy fixing a separating curve $\gamma_1$ (indicated, with other $\gamma$ curves, in green, parallel to $\beta$'s). Note that no $\alpha$ or $\beta$ curves except $\alpha_1$ and $\beta_1$ cross the grey curve $\phi(\gamma_1)$. The Heegaard diagram in (b) is obtained from the $\aalpha$ and $\ggamma$ curves in (a) after destabilizing by cancelling $\gamma_1$ with the $\alpha$ curve it hits, after possibly sliding some $\gamma$ curves over $\gamma_1$. This is the connected sum of the diagrams in (c) at the basepoints $z'$ and $z''$, which are not the standard basepoints in the HKM diagrams for the decomposed monodromies $\phi'$ and $\phi''$.}
\end{figure}

Now suppose $(S, \phi)$ is a reducible open book decomposition as previously, with $\gamma\subset S$ a separating curve fixed by the monodromy. Since $S$ has boundary on each side of $\gamma$, we can choose the arcs $a_1,\ldots, a_n$ in such a way that $a_1$ connects two different boundary components of $S$ and intersects $\gamma$ transversely in a single point, while the other $a_i$ are disjoint from $\gamma$. We obtain a new set of attaching circles $\gamma_1,\ldots,\gamma_n$ on $\Sigma = S_{1/2}\cup -S_0$ by taking $\gamma_1 = \gamma$ while the other $\gamma_i$ are (small Hamiltonian perturbations of) the corresponding $\beta_i$. Since our surgery is framed using the page, it is easy to see that $(\Sigma, \aalpha,\ggamma)$ is a Heegaard diagram for the open book surgery $Y_\gamma = Y'\#Y''$. In fact, $(\Sigma, \aalpha,\ggamma)$ is obtained by a single stabilization from a connected sum of HKM  diagrams corresponding to $(S', \phi')$ and $(S'', \phi'')$. Note, however, that some handleslides may be necessary in the destabilization, and in particular $(\Sigma, \aalpha, \ggamma)$ is not necessarily obtained by a connected sum of {\it pointed} HKM diagrams (c.f. Figure \ref{afigure}). We place basepoints $w'$ and $w''$ on either side of $\gamma_1$, so that $w'$ lies on the side of the diagram corresponding to $S'$.

The following generalizes Baldwin's theorem on capping off to the case of reducible open book surgery.

\begin{theorem}\label{naturalitythm} Assume that the reducible open book surgery corresponding to $\gamma\subset S\subset Y$ gives rise to open books $(S',\phi')$ and $(S'',\phi'')$ as above, and assume that $(S'',\phi'')$ is HKM strong. Then there exists a generator $\tilde{c}\in \hfhat(-Y'')$, and a \spinc structure $\s_0$ on the surgery cobordism $W_\gamma: -(Y'\#Y'')\to -Y$ such that 
\[
F_{W_\gamma, \s_0}(c(\xi')\otimes \tilde{c}) = c(\xi).
\]
\end{theorem}

In fact, the class $\tc$ is represented by the canonical generator $\x''$ on the HKM diagram for $(Y'', \xi'')$, with basepoint $z''$ as in Figure \ref{afigure}.

The main idea in the proof is to consider holomorphic triangles in the triple diagram $(\Sigma, \ggamma, \bbeta, \aalpha)$, much like \cite{baldwincap}. Indeed, this triple diagram describes the natural 2-handle cobordism between $Y$ and $Y'\#Y''$, thinking of the cobordism as connecting $-(Y'\#Y'')$ to $-Y$. (As in \cite{baldwincap}, the diagram is ``left-subordinate'' to the cobordism; see \cite[Section 5.2]{HolDiskFour}.) By construction of the diagram, there is a small triangle $\psi_0\in \pi_2(\x_0,\theta,\x)$ admitting a unique holomorphic representative, where $\x\in \Tb\cap \Ta$ is the canonical (HKM) representative of $c(\xi)$, $\theta\in \Tg\cap \Tb$ is the standard highest-degree intersection point, and $\x_0 \in \Tg\cap \Ta$ is the intersection point given by the standard intersections between the $\gamma_i$ and $\alpha_i$ for $i\neq 1$, together with the unique intersection point in $\gamma_1\cap \alpha_1$. 

In the following we continue to assume $Y''$ is a rational homology sphere. 

\begin{lemma} After possibly adjusting the monodromy $\phi$ by an isotopy, the diagram $(\Sigma, \ggamma,  \bbeta, \aalpha, w')$ is weakly admissible in the sense that every triply-periodic domain with $n_{w'}=0$ has both positive and negative coefficients.
\end{lemma}

\begin{proof} Suppose $\P$ is a nonnegative triply periodic domain in $(\Sigma, \ggamma,\bbeta,\aalpha, w')$, i.e., a nonnegative integer linear combination of regions between the attaching circles, excluding the region containing $w'$, whose boundary (as a chain) is a linear combination of $\alpha$, $\beta$, and $\gamma$ circles. We write the circles as $\aalpha = \aalpha' \cup \alpha_1\cup \aalpha''$, where the primes refer to the side of the diagram containing the curves, and similarly for $\bbeta$ and $\ggamma$. Since $n_{w'}(\P) = 0$, the combinatorial arguments from \cite[proof of Lemma 2.2]{baldwincap} show that none of the circles $\aalpha'$, $\bbeta'$, $\ggamma'$ appear in $\partial\P$. 

Consider the surface with boundary $H$ obtained from $\Sigma$ by cutting along $\gamma_1$ and $\phi(\gamma_1)\subset -S_{0}$, so $H$ is a disjoint union $H = H' \cup H''$ corresponding to the two sides of our diagram. Write $m$ for the coefficient of $\P$ in the ``small'' region between $\alpha_1$ and $\beta_1$ in $H'$ near $\gamma_1 \subset \partial H'$. 

\begin{figure}[b]
\includegraphics[width=3in]{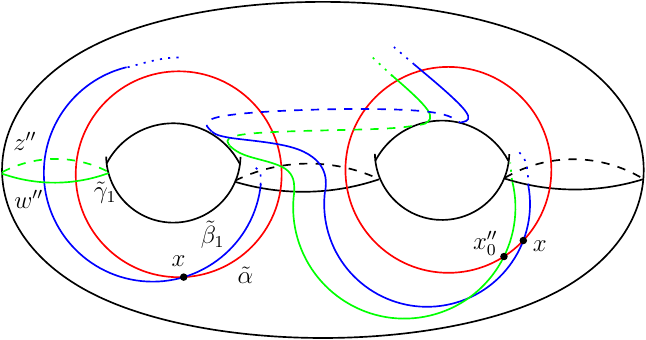}
\caption{\label{admissfig}The Heegaard triple obtained by cutting the diagram of Figure \ref{afigure} along $\gamma_1$ and $\phi(\gamma_1)$ and identifying boundary components. Shown is the portion corresponding to the right side of the diagram, $H''$.}
\end{figure}

By identifying the two boundary components of $H''$, we obtain a Heegaard diagram $(\widetilde{\Sigma}, \tilde\ggamma, \tilde\bbeta,\tilde\aalpha)$ as in Figure \ref{admissfig}, and $\P$ gives rise to a domain in this diagram. Observe that the new diagram is a Heegaard triple diagram describing surgery along a knot $K''$ in $Y''$ whose meridian corresponds to $\tilde\gamma_1$ and longitude is $\tilde\beta_1$. The set of triply-periodic domains in such a diagram is spanned by a domain for which the longitude $\beta_1$ appears in the boundary with coefficient equal to the order in first homology of $K''$, if that order is finite (c.f. the discussion in Section \ref{cobordismssec}). In our diagram the coefficient of $\beta_1$ in the boundary of $\P$ is just $m$, up to sign. Since we are assuming $b_1(Y'') = 0$, it follows that for any nontrivial, nonnegative periodic domain in the original diagram we must have $m \neq 0$.

Consider the portion of the diagram near $\phi(\gamma_1)\subset \partial H'$. Here it still must be the case that the coefficients of $\P$ differ by $m$ across each of $\alpha_1$ and $\beta_1$. Therefore we can ensure that a nonnegative triply periodic domain does not exist by introducing winding, in the sense of \cite[Section 5]{OS1}, of $\beta_1$ around the curve $\phi(\gamma_1)$, in both directions; this can be achieved by isotopy of the monodromy $\phi$ and thus still corresponds to an HKM diagram.
\end{proof}

The following refers to the triple diagram $(\Sigma, \ggamma,\bbeta,\aalpha, w', w'')$ corresponding to reducible open book surgery (Figure \ref{afigure}(a)).

\begin{prop}\label{triangleprop} For $\y \in \Tb\cap\Ta$, let $\psi\in\pi_2(\x_0,\theta,\y)$ be a homotopy class of Whitney triangles such that
\begin{enumerate}
\item $n_{w'}(\psi) = 0$
\item $\s_{w'}(\psi) = \s_{w'}(\psi_0)$
\item $\psi$  admits a holomorphic representative.
\end{enumerate}
Then $\y = \x$, and $\psi = \psi_0$.
\end{prop} 

In fact, the argument applies to triangles satisfying (1) and (2) and whose corresponding domain in $\Sigma$ has only nonnegative coefficients, which is true if $\psi$ admits a holomorphic representative.

\begin{proof} First we claim that any such $\psi$ also has $n_{w''}(\psi) = 0$. For this recall that the quantity $f(\x_0) = \langle c_1(\s_{w'}(\psi)), [F_\lambda]\rangle + [F_\lambda]^2 - 2(n_{w'}(\psi) - n_{w''}(\psi))$ depends only on the intersection point $\x_0$ and not, particularly, on the triangle $\psi\in \pi_2(\x_0,\theta,\y)$.  (See Lemma \ref{flemma}; here  $[F_\lambda]$ is the generator of $H_2(W, -(Y'\#Y''))$ represented by the core of the 2-handle.) If $\s_{w'}(\psi) = \s_{w'}(\psi_0)$ then the only term of $f(\x_0)$ that depends on $\psi$ is $n_{w'}(\psi) - n_{w''}(\psi)$. In particular since $n_{w'}(\psi_0) = n_{w''}(\psi_0) = 0$, the assumption $n_{w'}(\psi) =0$ forces $n_{w''}(\psi) =0$ as well. 

Hence the triangle $\psi$ must have vanishing coefficients in both ``large'' regions of the diagram $(\Sigma, \ggamma, \bbeta, \aalpha)$. The combinatorial arguments from \cite[Proposition 2.3]{baldwincap} now apply directly to give the conclusion.
\end{proof}

The proposition implies that an appropriately-defined chain map between $(\Sigma, \ggamma,\aalpha,w')$ and $(\Sigma, \bbeta,\aalpha, w')$ sends the generator $\x_0$ to the canonical representative $\x$ of the contact invariant $c(\xi)$, since there is just one homotopy class of triangle to consider and that homotopy class admits a unique holomorphic representative (we define the relevant chain map below). Moreover, if we write $\x_0 = \x_0' \times x_1\times \x_0''$, where $x_1\in \gamma_1\times \alpha_1$ is the unique intersection as before, then after destabilizing the diagram the intersections $\x_0'$ and $\x_0''$ are the canonical intersection points in HKM diagrams for $-Y'$ and $-Y''$.

\begin{lemma}\label{kunnethlemma} Assume that the diagram $(\Sigma'', \ggamma'', \aalpha'', w'')$ is HKM strong for the contact structure $\xi''$ supported by $(S'', \phi'')$. Then the intersection point $\x_0$ is a cycle in the chain complex $\cfhat(\Sigma,\ggamma,\aalpha, w')$. Moreover, under the K\"unneth isomorphism
\[
\hfhat(Y'\#Y'') \cong \hfhat(Y')\otimes \hfhat(Y''),
\]
the class $[\x_0]$ corresponds to $c(\xi)\otimes \tilde{c}$, where $\tilde{c}$ is the class represented by the intersection point $\x_0''$ in the diagram $(\Sigma'', \ggamma'', \aalpha'', z'')$.
\end{lemma}

Observe that since the diagram is HKM strong, the generator $\x_0''$ is a cycle in $\cfhat$ regardless of the position of the basepoint. It is, however, dealing with the basepoints that is the main difficulty in the proof of the lemma. To do so we make use of some technology from knot Floer theory. 

Recall that if $(\Sigma, \aalpha,\bbeta, w, z)$ is a doubly-pointed Heegaard diagram describing a knot $K\subset Y$, the knot Floer chain complex $CFK^-(Y, K) = CFK^-(\Sigma,\aalpha,\bbeta,w,z)$ is generated over $\F[U]$ by intersection points in the usual way, with differential counting holomorphic disks that miss the basepoint $z$ and keeping track of intersections with $w$ via the power of $U$. Explicitly:
\[
\partial^-_K\x = \sum_{n_z(\phi) = 0} \#\M(\phi) U^{n_w(\phi)}\y.
\]
Here the sum is over intersection points $\y$ and Whitney disks $\phi\in \pi_2(\x,\y)$ having Maslov index 1. There is a chain map $CFK^-(\Sigma,\aalpha,\bbeta, w,z)\to \cfhat(\Sigma, \aalpha,\bbeta, z)$ given by setting $U = 1$ (i.e., forgetting the basepoint $w$ and declaring $z$ to be the new basepoint). In the following, we will need to treat the basepoints we have been labeling as $w'$, $w''$, $z'$, $z''$ in different ways in the knot complex. To avoid confusion, we will use notation such as $CFK^-(\Sigma', \ggamma',\aalpha', n_{w'}= 0, U^{n_{z'}})$ to indicate the knot Floer complex in which $w'$ plays the role of $z$ above and $z'$ corresponds to $w$.

\begin{proof}[Proof of Lemma \ref{kunnethlemma}] Consider the two Heegaard diagrams $(\Sigma',\ggamma', \aalpha', w',z')$ and $(\Sigma'', \ggamma'', \aalpha'', w'', z'')$ of Figure \ref{afigure}(c), and the corresponding tensor product of knot Floer complexes
\[
CFK^-(\Sigma', \ggamma', \aalpha', n_{w'} = 0, U^{n_{z'}}) \otimes_{\zee[U]} 
CFK^-(\Sigma'', \ggamma'', \aalpha'', n_{z''} = 0, U^{n_{w''}}).
\]
According to \cite[Section 7]{OSknot} (see also \cite[Section 11]{OSlink}), there is an isomorphism of this complex with the complex $CFK^-(\Sigma'\#\Sigma'', \ggamma'\cup\ggamma'', \aalpha'\cup\aalpha'', n_{w'} = n_{z''} = 0, U^{n_{z'}+n_{w''}})$, obtained as follows. First form the triple diagram $(\Sigma'\#\Sigma'', \ggamma'\cup\ggamma'', \tilde{\aalpha}' \cup\tilde{\ggamma}'', \aalpha'\cup\aalpha'', w', w'')$, where the tildes indicate small Hamiltonian perturbation, as in Figure \ref{kunnethfigure}. Denoting by $\theta'$ (resp. $\theta''$) the canonical intersections between $\aalpha'$ and $\tilde{\aalpha}'$ (resp. $\ggamma''$ and $\tilde{\ggamma}''$), the image of $\x'\otimes \x''$ from the tensor product complex is obtained by counting holomorphic triangles in the triple diagram, having vanishing multiplicity at $w'$ and with corners at $\tilde{\x}'\times \theta''$ and $\theta'\times\tilde{\x}''$, where $\tilde{\x}'\in \ggamma'\cap \tilde{\aalpha}'$ and $\tilde{\x}'' \in \aalpha''\cap\tilde{\ggamma}''$ are the obvious intersection points corresponding to $\x'$ and $\x''$. We observe:
\begin{itemize}
\item There is a ``small triangle'' $\psi_0\in \pi_2(\tilde{\x}_0'\times \theta'', \theta'\times \tilde{\x}_0'', \x_0'\times\x_0'')$ admitting a unique holomorphic representative.
\item Any other triangle $\psi\in \pi_2(\tilde{\x}_0'\times \theta'', \theta'\times \tilde{\x}_0'', \U'\times\U'')$ with $n_{w'}(\psi) = 0$ and with nonnegative coefficients is actually equal to $\psi_0$.
\end{itemize}
Indeed, the first of these is clear from the diagram. For the second we adapt the argument from \cite[Section 7]{LOSS}: first note that any triangle $\psi$ as in the claim differs from $\psi_0$ by splicing a disk $\phi\in\pi_2(\x_0'\times \x_0'', \U'\times\U'')$ with nonnegative coefficients, boundary on $\Ta$ and $\Tg$, and $n_{w'}(\phi) = 0$. Clearly such a disk must have vanishing coefficients on the $\Sigma'$ side of the diagram, and in particular we must have $\U' = \x_0'$. But then $\phi$ gives rise to a disk starting from $\x_0''$ supported on $\Sigma''$, which is necessarily trivial by the HKM strong condition. 

\begin{figure}
\includegraphics[width=2.5in]{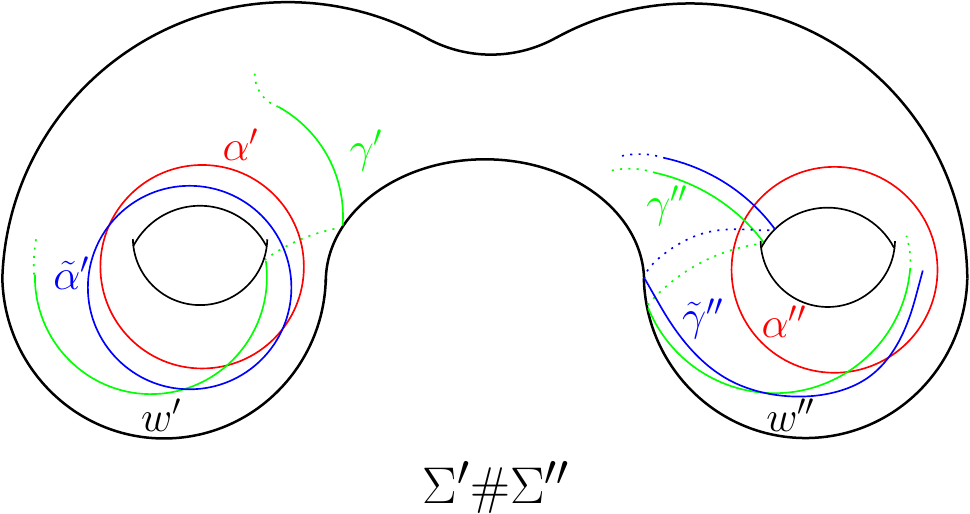}
\caption{\label{kunnethfigure}Heegaard triple for the K\"unneth argument.}
\end{figure}

To pass from the connected sum diagram to the HKM diagram $(\Sigma, \ggamma,\aalpha, w', w'')$, we must stabilize and then slide some $\gamma$ circles as necessary (reversing the transition from Figure \ref{afigure}(a) to (b)). Stabilization (at $w'$) certainly maps $\x_0'\times \x_0'' \mapsto \x_0'\times x_1\times \x''_0$; the maps induced by handleslides involve counts of holomorphic triangles that also have the desired behavior by an argument entirely similar to the one above. We leave the reader to fill in the details.

We have now seen that there is a chain isomorphism 
\begin{eqnarray*}
&CFK^-(\Sigma', \ggamma', \aalpha', n_{w'} = 0, U^{n_{z'}}) \otimes_{\zee[U]} 
CFK^-(\Sigma'', \ggamma'', \aalpha'', n_{z''} = 0, U^{n_{w''}})&\\ &\hspace*{3em}\to CFK^-(\Sigma, \ggamma,\aalpha,n_{w'} = 0, U^{n_{w''}})&
\end{eqnarray*}
such that the chain $\x_0'\otimes \x_0''$ is sent to $\x_0'\times x_1\times \x_0''$. Moreover, it is clear that both $\x_0'$ and $\x_0''$ are cycles in the respective factors on the left side ($\x_0'$ by virtue of the placement of the basepoint $w'$, and $\x_0''$ in light of the HKM strong condition). Applying the natural transformation $CFK^-\to \cfhat$ we get a homology isomorphism
\[
\hfhat(\Sigma', \ggamma',\aalpha', w')\otimes \hfhat(\Sigma'',\ggamma'',\aalpha'', z'')\to \hfhat(\Sigma,\ggamma,\aalpha, w')
\]
mapping $[\x_0']\otimes[\x_0'']$ to $[\x_0'\times x_1\times \x_0'']$, where the latter class is well-defined since  $CFK^-\to \cfhat$ is a chain map.
\end{proof}

\begin{proof}[Proof of Theorem \ref{naturalitythm}] The triple diagram $(\Sigma, \ggamma, \bbeta,\aalpha, w')$ is left-subordinate to the surgery cobordism $W_\gamma$. Writing $\s_0$ for the \spinc structure induced by the small triangle $\psi_0$, we see from Proposition \ref{triangleprop} that the corresponding chain map given by
\[
F_{W_\gamma,\s_0}(\y) = \sum_{\psi\in\pi_2(\y,\theta,\x)} \#\M(\psi) \x
\]
(where the sum is over homotopy classes of triangles with Maslov index 0, $\s_{w'}(\psi) = \s_{w'}(\psi_0)$, and $n_{w'}(\psi) = 0$) carries $\x'\times x_1\times\x''$ to the HKM generator $\x$ in $(\Sigma, \bbeta,\aalpha,w')$, which is a cycle representing $c(\xi)$. The theorem then follows from Lemma \ref{kunnethlemma}.
\end{proof}

\section{Positive Contact Surgery}\label{surgerysec}

We consider positive contact surgery along a nullhomologous Legendrian knot $K\subset (Y,\xi)$ (we abuse notation slightly here, using the same notation for a knot type and a particular Legendrian representative of it). We adhere to conventions from the introduction, so the smooth surgery coefficient will be written as $\frac{p}{q}$. The corresponding contact surgery coefficient will be written $\frac{x}{y}$, so that $y = q$ while $x = p - q\tb(K)$. We write $p = mq - r$ as previously, where $0\leq r < q$, so that the rational surgery cobordism is $W = W_{p/q}: Y\#(-L(q,r))\to Y_{p/q}(K)$. 

Our strategy in the proof of Theorem \ref{naturalitythmintro} relies on reducible open book surgery. Our construction applies in particular in the case of a contact surgery coefficient $\frac{x}{y}\geq 1$, so for most of this section we will make that assumption. In Section \ref{gencoeffsubsec} we show how to deduce the general case from this one.

\subsection{Naturality} Our main goal for this section is the following, which is a portion of Theorem \ref{naturalitythmintro}. 

\begin{theorem} \label{naturalitythmspec} Let $K\subset (Y,\xi)$ be an oriented nullhomologous Legendrian knot, and $\frac{x}{y}\in\cue$ a contact framing with $\frac{x}{y}\geq 1$. Let $(Y_{x/y}(K), \xi_{x/y}^-)$ be the result of contact $\frac{x}{y}$ surgery on $K$, with the contact structure $\xi_{x/y}^-$ described below. If $W: Y\# (-L(q,r))\to Y_{p/q}(K)$ is the corresponding rational surgery cobordism, then there exists a \spinc structure $\s\in \Spinc(W)$ and a generator $\tilde{c}\in\hfhat(L(q,r))$ with the property that
\[
F_{-W,\s}(c(\xi)\otimes \tilde{c}) = c(\xi_{x/y}^-),
\]
where $-W: -Y\# L(q,r)\to -Y_{p/q}(K)$ is the oppositely-oriented cobordism.
\end{theorem}

This theorem is an application of Theorem \ref{naturalitythm}, where we realize $-W$ above as a reducible open book surgery cobordism between $Y_{p/q}(K)$ and $Y\# (-L(q,r))$. To do this, we review an algorithm due to Ding, Geiges and Stipsicz \cite{DGS} for describing a rational contact surgery in terms of $\pm 1$ surgeries, and interpret that algorithm in the context of open books. 

\begin{theorem}[DGS algorithm] \label{DGSalgorithm} Given $Y,\xi,K$ as in Theorem \ref{naturalitythmspec}, let $0<\frac{x}{y}\in\cue$ be a contact surgery coefficient. Let $r\in\zee$ be the minimal positive integer such that $\frac{x}{y-rx} <0$, and form the continued fraction 
\[
\frac{x}{y-rx} = [a_1,a_2,\ldots, a_n] = a_1 - \frac{1}{a_2 - \frac{1}{\cdots - \frac{1}{a_n}}},
\]
where each $a_i \leq -2$. Then any contact $\frac{x}{y}$ surgery on $K$ can be described as contact surgery along a link $(K_0^{(1)}\cup\cdots\cup K_0^{(r)})\cup K_1\cup\cdots \cup K_n$,  where
\begin{itemize}
\item $K_0^{(1)},\ldots, K_0^{(r)}$ are parallel Legendrian pushoffs of the original Legendrian $K$,
\item $K_1$ is obtained from a Legendrian pushoff of $K_0^{(r)}$ by stabilizing $|a_1 + 1|$ times,
\item for $j\geq 2$, $K_j$ is obtained from a Legendrian pushoff of $K_{j-1}$ by stabilizing $|a_j + 2|$ times,
\item the contact surgery coefficient on each $K_0^{(i)}$ is $+1$, while the coefficient is $-1$ for the other $K_j$.
\end{itemize}
\end{theorem}

The ambiguity in the resulting contact structure $\xi_{x/y}$ arises from the choice of stabilizations used for each $K_j$; the contact structure $\xi_{x/y}^-$ is the one given by choosing all these stabilizations to be {\it negative} (with respect to the chosen orientation of $K$).

Note that if $\frac{x}{y}\geq 1$, then we have $r = 1$ in Theorem \ref{DGSalgorithm}, i.e., there is only one $+1$ contact surgery required in the algorithm. We will assume $\frac{x}{y}\geq 1$ from now until Subsection \ref{gencoeffsubsec}, and write $K_0$ for $K_0^{(1)}$ in the algorithm above.

Recall that for any Legendrian $K\subset (Y,\xi)$ one can find an open book decomposition supporting $\xi$ such that $K$ lies on a page of the open book and such that the contact framing on $K$ agrees with the framing induced by the open book (see \cite[Corollary 4.23]{etnyreclay}, for example). Moreover, it was observed in \cite{BEV} that one can arrange for stabilizations of $K$ to appear on pages of the stabilized open book in the following way. Having fixed an orientation for $K$, choose an embedded path $c$ on the page, which starts on a boundary component and approaches $K$ ``from the right.'' Stabilize the open book using the arc on the page that is the non-closed component of the boundary of a regular neighborhood of $K\cup c$; then the negative stabilization $K^-$ is Legendrian isotopic to a curve on the page of the stabilized open book that is parallel to the boundary component meeting $c$, and the page framing of $K^-$ agrees with the contact framing (see Figure \ref{Figure D, E, F}). The stabilization involves composing the monodromy of the open book with a Dehn twist along the closed curve $C$ that is the union of the indicated arc with the core of the new 1-handle. 
\begin{figure}
\includegraphics[width=4.5in]{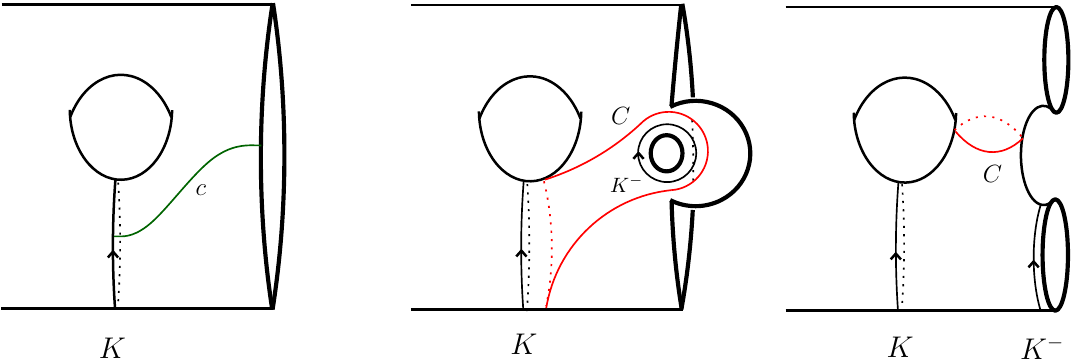}
\caption{\label{Figure D, E, F}Stabilizing a Legendrian on the page of an open book. The left picture is before stabilization, the center and right-hand pictures are equivalent pictures after stabilization, with the stabilized Legendrian $K^-$ indicated, and where the monodromy has been composed with a right twist about $C$.}
\end{figure}

Iterating this procedure, we can stabilize the open book repeatedly to obtain a decomposition supporting $\xi$ in which an arbitrary number of stabilizations of $K$ appear as curves on the page. In particular, given $\frac{x}{y}\geq 1$, we can find an open book decomposition for which all the Legendrians $K_0,\ldots, K_n$ from Theorem \ref{DGSalgorithm} appear as closed curves on the page, with contact framing equal to the page framing. Moreover, since performing $-1$ (resp. $+1$) surgery along a closed curve in the page, where the framing is measured with respect to the page framing, is equivalent to composing the monodromy of the open book with a right- (resp. left-) handed Dehn twist, we get an open book decomposition for $\xi^-_{x/y}$ by adding a left twist to $\phi$ along $K = K_0$, and right twists along each $K_1,\ldots, K_n$ (the Dehn twists corresponding to the various pushoffs commute, since the curves are disjoint in the page). The procedure is illustrated in Figure \ref{Figure G}. Note that each $K_j$ is a parallel copy of some stabilized Legendrian $K^{-m}$, but not all stabilizations are necessarily used and some may be repeated. We denote the open book decomposition for $(Y_{x/y}(K),\xi^-_{x/y})$ obtained this way by $(S,\phi)$. 

\begin{figure}
\includegraphics[width=4in]{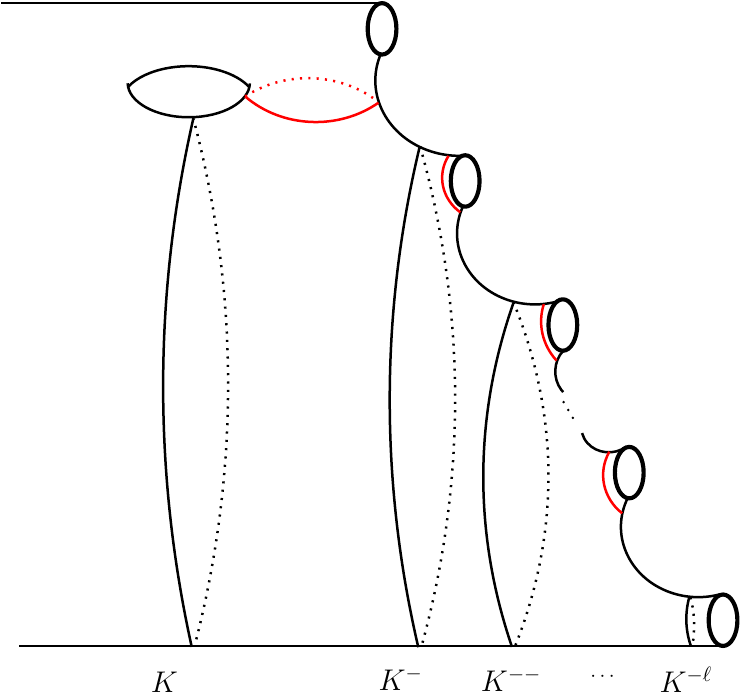}
\caption{\label{Figure G}Iterating the stabilization procedure. Note that a right Dehn twist is introduced to the monodromy around each new boundary component except for the last.}
\end{figure}

Observe that in $(S,\phi)$ the monodromy preserves the once-stabilized Legendrian $K^-$, and that $K^-$ separates $S$ with boundary components on each side. Applying reducible open book surgery along $K^-$, we obtain a 3-manifold that is the connected sum of manifolds with open book decompositions $(S',\phi')$ and $(S'',\phi'')$. Note in particular that after the reducing surgery, the original Legendrian $K\subset S'$ is isotopic to the first stabilization curve $C$, and therefore the right twist along $C$ cancels with the left twist we introduced along $K = K_0$. Hence, the open book $(S',\phi')$ is indeed the original open book for $(Y,\xi)$ we began with.

On the other hand, the open book $(S'',\phi'')$ is clearly planar and the monodromy $\phi''$ is given by the composition of right twists along all but one boundary component together with various copies of the disjoint circles $K^{-j}$ as indicated in Figure \ref{Figure H}. The proof of Theorem \ref{naturalitythmspec} rests on the following result.

\begin{figure}
\includegraphics[width=4.5in]{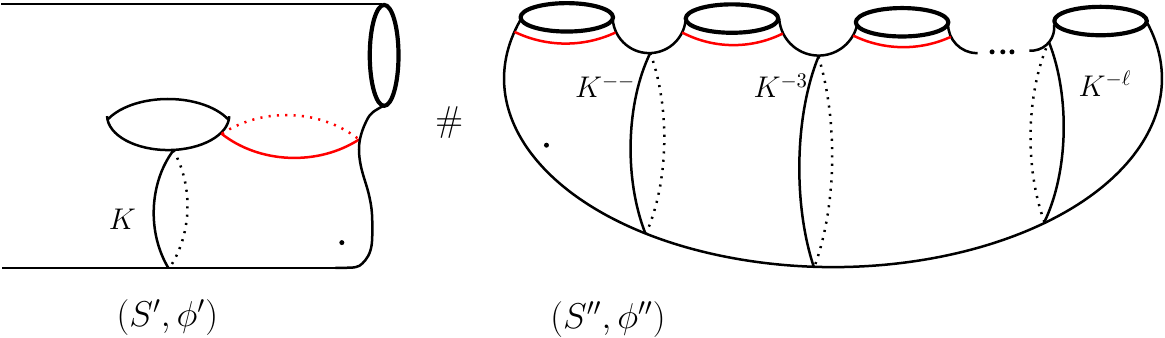}
\caption{\label{Figure H}The result of reducible open book surgery along $K^-$ in the open book of Figure \ref{Figure G}. The connect-sum points are indicated.}
\end{figure}

\begin{lemma} The 3-manifold $Y''$ described by $(S'',\phi'')$ is diffeomorphic to the lens space $-L(q,r)$. Moreover, the 2-handle cobordism $Y_{p/q}\to Y\#(-L(q,r))$ corresponding to the reducible open book surgery is diffeomorphic to $-W_{p/q}$, the rational surgery cobordism with its orientation reversed. Finally, the open book decomposition $(S'',\phi'')$ is HKM strong in the sense of Definition \ref{strongdef}.
\label{surgcoblemma}
\end{lemma}

\begin{figure}
\includegraphics[width=2.5in]{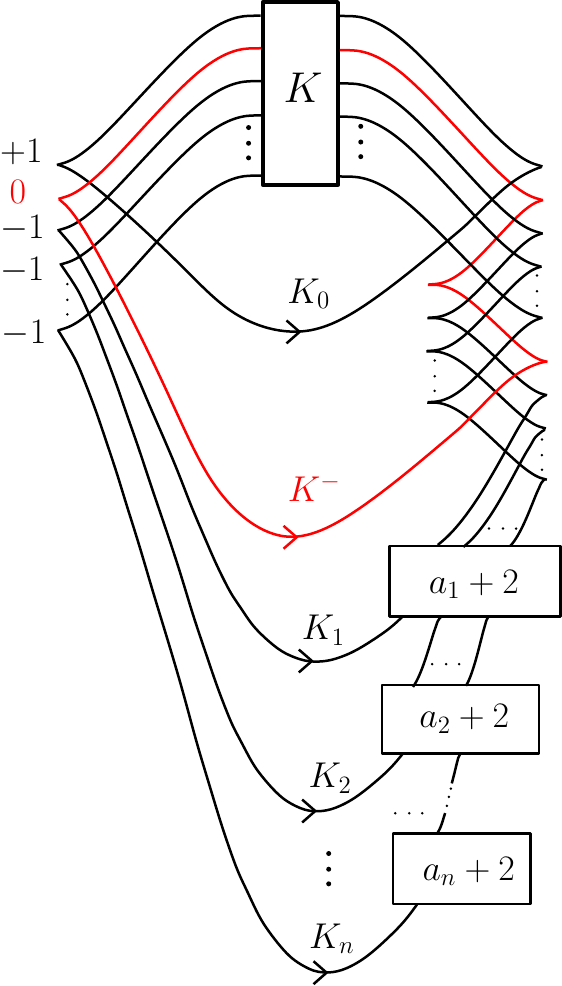}
\caption{\label{Figure I}The Ding-Geiges-Stipsicz picture for rational contact surgery. The boxes indicate the number of stabilizations; note each $a_j+2\leq 0$, i.e., all stabilizations are negative.}
\end{figure}

\begin{proof}
The first sentence follows from the second. The latter can be proved in a manner analogous to \cite[Proposition 4.1]{LS2011}, in fact our constructions in this section are generalizations of those in \cite{LS2011}. Working more directly,  Figure \ref{Figure I} shows the contact surgery diagram resulting from the Ding-Geiges-Stipsicz procedure, along with a copy of the once-stabilized knot $K^-$. Note that $K$ is a Legendrian in a 3-manifold $(Y,\xi)$, which itself can be described by a contact surgery diagram; this background diagram is not indicated in the figure. The reducible open book surgery cobordism is given by a 2-handle attached along $K^-$ with framing equal to the Thurston-Bennequin invariant of $K^-$; we turn the cobordism over and reverse its orientation by bracketing all surgery coefficients and introducing a 0-framed meridian for $K^-$. To simplify the picture, note that the stabilized pushoffs $K_2,\ldots, K_n$ can be successively slid, each over the preceding, to give a chain of unknots (in fact, since we have performed $-1$ contact surgery on each of these knots, \cite{DGhandle} implies each $K_j$  is Legendrian isotopic to a stabilized standard Legendrian meridian of $K_{j-1}$ for $j = 2,\ldots, n$). The corresponding smooth surgery picture is shown in Figure \ref{Figure K}(a), and a little more manipulation of the diagram (sliding $K^-$ and then $K_1$ over $K_0$) gives Figure \ref{Figure K}(d). 
\begin{figure}
\includegraphics[width=4.5in]{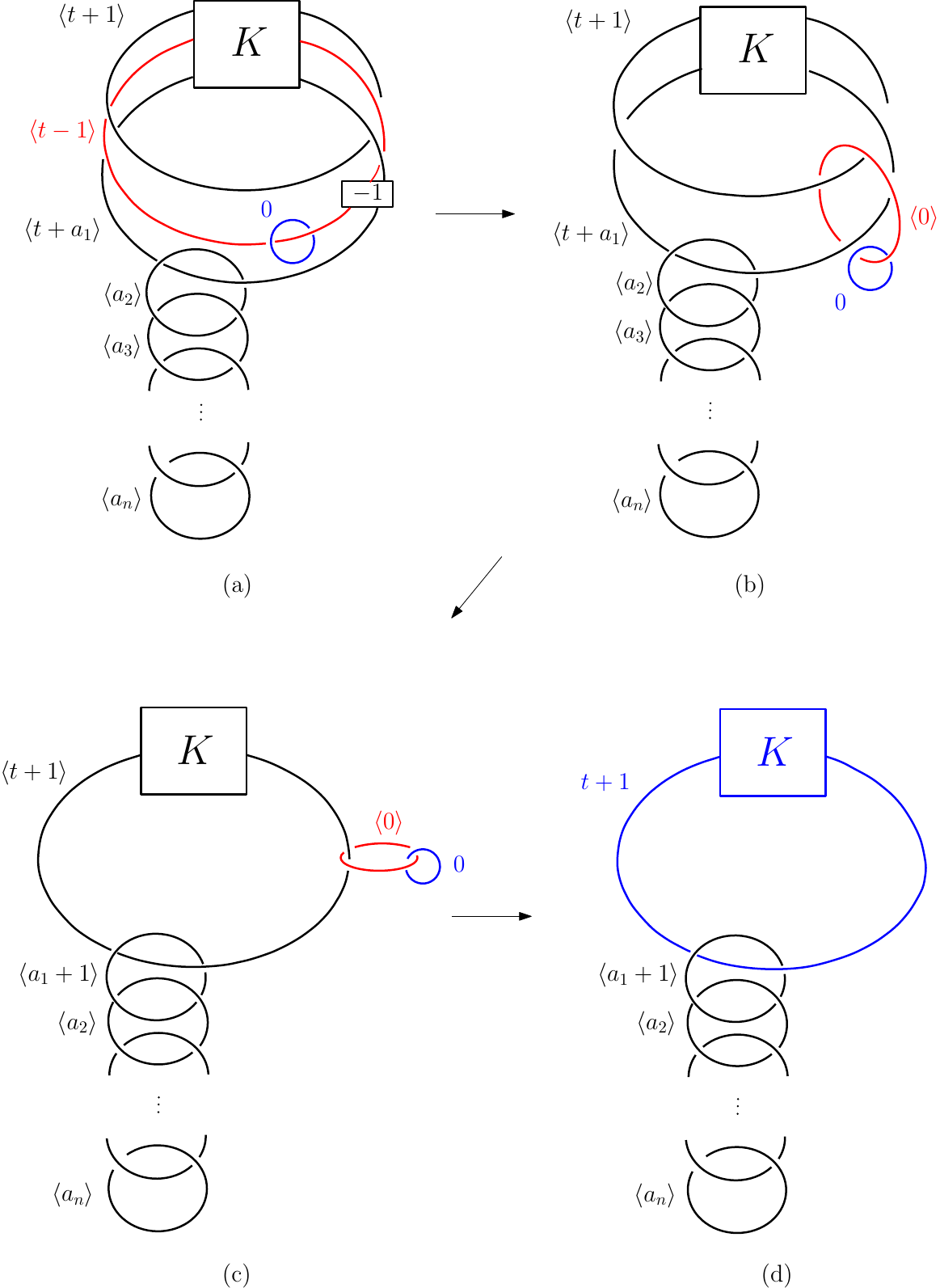}
\caption{\label{Figure K}The reducible open book surgery cobordism, upside-down and with reversed orientation. Here we write $t = \tb(K)$ for the Thurston-Bennequin number of the original knot $K$.}
\end{figure}
Recalling that $[a_1,\ldots, a_n] = \frac{x}{y-x}$, that $x = p-q\tb(K)$, $y = q$, and writing $p = mq-r$ it is easy to obtain Figure \ref{Figure P}, which is nothing but the surgery cobordism $W_{p/q}$. Since our picture has the wrong orientation, the proof is complete except for the claim that the diagram for $-L(q,r)$ is HKM strong. We postpone this to Section \ref{strongsec}.
\end{proof}

\begin{figure}
\includegraphics[width=4in]{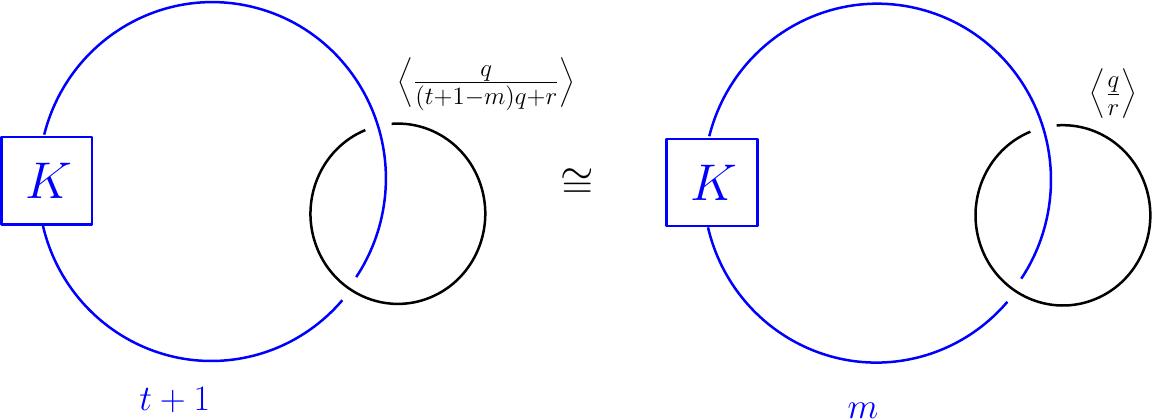}
\caption{\label{Figure P}Simplification of Figure \ref{Figure K}, still writing $t = \tb(K)$.}
\end{figure}

\begin{proof}[Proof of Theorem \ref{naturalitythmspec}] The previous lemma shows that the oppositely-oriented surgery cobordism $-W_{p/q}: -Y\#L(q,r)\to -Y_{x/y}(K)$ is diffeomorphic to the reducible open book surgery cobordism, after turning around. The naturality property of such reducible surgery cobordisms (Theorem \ref{naturalitythm}) proves the result.
\end{proof}

\subsection{$Spin^c$ structures} We now wish to identify the \spinc structure on the rational surgery cobordism whose existence is guaranteed by Theorem \ref{naturalitythmspec}. Since it is obtained from the reducible open book surgery construction, let us write $\srobs\in \Spinc(-W_{p/q})$ for that \spinc structure. 

To begin with, observe that there is another natural \spinc structure on $W_{p/q}$ arising from the Ding-Geiges-Stipsicz algorithm. Indeed, we can think of the diagram of Figure \ref{Figure I} (without the handle corresponding to $K^-$) as describing a 4-manifold $X_{p,q}$, which is clearly diffeomorphic to that given in Figure \ref{Figure K}(d) after removing all brackets. Thus $X_{p,q}$ contains the rational surgery cobordism $W_{p/q}: Y\# {-L(q,r)}\to Y_{p/q}(K)$. Moreover, since there is a single $+1$ contact surgery coefficient in the diagram of Figure \ref{Figure I}, the manifold $\Xhat_{p,q} = X_{p,q}\# \cee P^2$ admits an almost-complex structure $\jdgs$ whose Chern class evaluates on the 2-handles corresponding to the Legendrians in Figure \ref{Figure I} as the corresponding rotation number---see \cite[Proposition 3.1]{DGS}, also \cite{EliashbergStein,Gompf}. (Here and to follow, we suppose an oriented Seifert surface for $K$ has been fixed.) We get a corresponding \spinc structure $\sdgs$ on $\Xhat_{p,q}$, and write $\sdgs\in \Spinc(W_{p/q})$ also for its restriction to the rational surgery cobordism. Note that both $\srobs$ and $\sdgs$ restrict to $Y_{p/q}(K)$ as the \spinc structure induced by the contact structure. 

The next result characterizes the \spinc structure arising from $\jdgs$; we will see that $\srobs$ is essentially the ``same'' as this one (in quotes, since $\srobs$ is on the oppositely-oriented manifold).

\begin{lemma}\label{Jevaluationlemma} Fix a generator $[S]\in H_2(W_{p/q};\zee)$ satisfying \eqref{conv1}, i.e., so that $[S]$ corresponds to $q$ times the relative class given by the core $F$ of the 2-handle, where the latter is oriented so that $\partial F = -K$. The \spinc structure $\sdgs\in \Spinc(W_{p/q})$ satisfies
\[
\langle c_1(\sdgs), [S]\rangle = \rot(K)\,y + x - 1 = p + (\rot(K) - \tb(K))\, q - 1
\]
\end{lemma}

\begin{proof} For simplicity, we suppose that the initial contact manifold is $(Y,\xi) = (S^3,\xi_{std})$. The general case (involving keeping track of a contact surgery diagram for $(Y,\xi)$) is similar.

The Legendrians $K = K_0,\, K_1,\ldots, K_n$ in Figure \ref{Figure I} correspond to 2-handles in $\Xhat_{p,q}$, and since all the $K_i$ are nullhomologous we get corresponding homology classes $k_0,\,k_1,\ldots, k_n\in H_2(\Xhat_{p,q};\zee)$. The Chern class $c = c_1(\jdgs)\in H^2(\Xhat_{p,q},\zee)$ of the  almost-complex structure satisfies
\begin{eqnarray*}
\langle c, k_0\rangle &=& \rot(K)\\
\langle c, k_j\rangle &=& \rot(K) + a_1 + \cdots + a_j + 2j -1 \quad \mbox{for $j  =1,\ldots n$}.
\end{eqnarray*}
The handleslides relating Figure \ref{Figure I} to Figure \ref{Figure K}(d) show that 
\begin{eqnarray}
\langle c, \ell_0\rangle &=& \rot(K)\nonumber\\
\langle c, \ell_1\rangle &=& a_1 + 1\label{c1values}\\
\langle c, \ell_j \rangle &=& a_j + 2 \quad \mbox{for $j = 2,\ldots, n$},\nonumber
\end{eqnarray}
where $\ell_0 = k_0$ and $\ell_j = k_j - k_{j-1}$ for $j = 1,\ldots, n$ are the homology classes corresponding to the handles in Figure \ref{Figure K}(d) (with brackets removed).

With respect to the basis $\{\ell_0,\ell_1,\ldots, \ell_n\}$, the intersection form of $X_{p,q}$ is given by the matrix
\[
Q_{X_{p,q}}\cong\left[\begin{array}{ccccccc}  t+1 & -1 & & & & \\
  -1 & a_1+1 & 1 &&&\\
 & 1 & a_2 &&&\\
 &&&\cdots &\\
 &&&& & 1 \\
 &&&&1 & a_n \end{array}\right]
\]
The cobordism $W_{p/q}$ is given just by the handle addition corresponding to $\ell_0$, after adding $\ell_1,\ldots,\ell_n$, and therefore the class  $[S]\in H_2(W_{p/q},\zee)$ can be represented in our basis as a generator for the kernel of the matrix obtained by deleting the row corresponding to $\ell_0$ in the intersection matrix for $X_{p,q}$ (see, for example, \cite[Section 9]{ManOzs}). To understand this class, recall that for a rational number $x_0/x_1>0$, its continued fraction expansion is obtained inductively by using the division algorithm to write 
\begin{eqnarray}
x_0 &=& b_1 x_1 - x_2 \quad \mbox{with $0\leq x_2<x_1$}\nonumber\\
x_1 &=& b_2 x_2 - x_3 \quad \mbox{with $0\leq x_3< x_2$}\nonumber\\
&\vdots& \label{recursion}\\
x_{n-1} &=& b_n x_n,\nonumber
\end{eqnarray}
where the procedure stops when the remainder $x_n$ divides the preceding remainder $x_{n-1}$. If $x_0/x_1$ is in lowest terms, this happens just when $x_n = 1$. In particular, this means that the vector $(x_0,x_1,\ldots, x_n)$ is in the kernel of
\begin{equation}\label{Mdef}
M(-b_1,\ldots, -b_n): = \left[\begin{array}{ccccccc} 1 & -b_1 & 1 & & & & \\
 & 1 & -b_2 & 1 &&&\\
 && 1 & -b_3 &&&\\
 &&&&\cdots &\\
 &&&&& & 1 \\
 &&&&&1 & -b_n \end{array}\right]
 \end{equation}
 and conversely the kernel of this matrix is spanned by a vector uniquely specified by requiring its entries to be decreasing positive integers with final entry 1. Moreover, this canonical kernel vector $(x_0,\ldots, x_n)$ has the property that $\gcd(x_j, x_{j+1}) = 1$, and $x_0/x_1 = [b_1,\ldots, b_n]$ as a continued fraction. 

Deleting the row corresponding to $\ell_0$ in the intersection matrix for $X_{p,q}$ gives the matrix $\MM = M^-(a_1+1,a_2, \ldots, a_n)$, which is obtained from $M(a_1+1,a_2, \ldots, a_n)$ by reversing the sign of the $(1,1)$ entry. The kernel of $\MM$ is spanned by the vector $\y = (y_0,y_1,\ldots, y_n)$, where the $y_j$ satisfy equations analogous to \eqref{recursion} above, modulo the first sign. It follows, arranging $y_n = 1$, that $y_0/y_1 =1 + [a_1,\ldots, a_n] = y/(y-x)$. Since $y_j>0$ for $j\geq 1$, we conclude $y_0 = -y$ and $y_1 = x-y$.

Thinking of the generator $\y$ as a linear combination of the $\ell_j$ and using \eqref{c1values}, the evaluation of $c$ on $\y$ is given by $\rot(K)\, y_0 + (a_1+1)y_1 + (a_2+2)y_2  + \cdots + (a_n+2)y_n$. It follows from the equations $\MM\y = 0$ that 
\begin{equation*}
-y_0 +  (a_1+1)y_1 + (a_2+2)y_2 + \cdots + (a_{n-1}+2)y_{n-1} + (a_n + 1)y_n = 0,
\end{equation*}
and this, along with $y_0 = -y$, $y_1 = x-y$, and $y_n = 1$, shows quickly that $\langle c,\y\rangle = -\rot(K)\,y - x + 1$. 

Finally, observe that $\y$ corresponds to the class $-[S]$, since the coefficient of the relative class $[F]$ is $y_0 = -y = -q$ rather than $q$.

\end{proof}


By construction, $\jdgs$ induces the plane field $\xi^-_{x/y}$ on $Y_{p/q}$. It also induces a plane field $\xidgs$ on the lens space $L(q,q-r)\subset \widehat{X}_{p,q}$, but in general this is not homotopic to a ``standard'' contact structure. To understand this, write $P$ for the plumbed 4-manifold obtained by attaching 2-handles to a chain of unknots in $S^3$ with framings $a_1+1,a_2,\ldots,a_n$ (that is, $P$ is given by the diagram of Figure \ref{Figure K}(d), with the knot $K$ omitted and brackets removed). Writing $\ell_1,\ldots,\ell_n$ for the corresponding 2-dimensional homology classes as above, the Chern class of $\jdgs$ evaluates on the $\ell_j$ according to \eqref{c1values}. On the other hand, there is a standard contact structure on $L(q,q-r)$ (the universally tight one), induced by a particular Stein structure on a plumbed manifold bounding $L(q,q-r)$. If $a_1\leq -3$ then this manifold is just $P$ itself, but if $a_1 = -2$ we may need to blow down. The result we need is the following.

\begin{lemma} There exists a complex structure $\jhat$ on $P$ inducing the standard Stein fillable contact structure $\xi_{Stein}$ on $L(q,q-r)$. Moreover, we have
\begin{eqnarray*}
\langle c_1(\jhat), \ell_1\rangle &=& a_1 + 3\\
\langle c_1(\jhat), \ell_j\rangle &=& a_j + 2 \quad \mbox{for $j = 2,\ldots, n$}.
\end{eqnarray*}
\end{lemma}

\begin{proof} First suppose $a_1 \leq -3$, so that all surgery coefficients appearing in the diagram for $P$ are at most $-2$. Then $P$ admits a Stein structure obtained by drawing each unknot in the plumbing diagram as a Legendrian, with all-negative stabilizations chosen so that the surgery coefficients are each one less than the corresponding Thurston-Bennequin numbers. As before, the Chern class of the complex structure evaluates on the classes $\ell_j$ as the rotation number, and then it is easy to check the desired statement.

Now suppose that for some $k\in\{1,\ldots, n-1\}$ we have $a_1 = \cdots = a_k = -2$, and $a_{k+1}\leq -3$. Blowing down sequentially, we see that $P \cong P_0\# k\cptwobar$ where $P_0$ is a plumbing of spheres having self-intersection $a_{k+1}+1, a_{k+2},\ldots, a_n$. If $e_1,\ldots, e_k$ are the exceptional spheres corresponding to the blowups, we can write 
\[
\mbox{$\ell_1 = e_1,\,\, \ell_2 = e_2 - e_1, \,\ldots,\, \ell_{k} = e_k- e_{k-1},\,\,\ell_{k+1} = {\ell}'_{k+1} - e_k$,}
\]
where $\ell'_{k+1}$ is the homology class of the sphere of square $a_{k+1} +1$ in $P_0$. There is then a Stein structure $J_{Stein}$ on $P_0$ obtained as above, with $\langle c_1(J_{Stein}), \ell'_{k+1}\rangle = a_{k+1} +3$ and $\langle c_1(J_{Stein}), \ell_j\rangle = a_j + 2$ for $j = k+2,\ldots, n$.
We can then blow up $(P_0, J_{Stein})$ in the complex category to obtain a complex structure $\jhat$ on $P = P_0\# k\cptwobar$. This complex structure has $c_1(\jhat) = c_1(J_{Stein}) - \tilde{e}_1-\cdots - \tilde{e}_k$, where $\tilde{e}_j$ is Poincar\'e dual to the exceptional sphere $e_j$, and moreover $\jhat$ induces the same plane field on $L(q,q-r) = \partial P =\partial P_0$ as $J_{Stein}$ does. It is now easy to check that $\langle c_1(\jhat), \ell_{k+1}\rangle = a_{k+1} +2$, that $\langle c_1(\jhat), \ell_{j}\rangle = 0 = a_j +2$ for $2\leq j \leq k$, and that $\langle c_1(\jhat), \ell_{1}\rangle = 1 = a_1 + 3$ as desired.

The only remaining case, that all $a_j = -2$, is proved easily along the same lines.
\end{proof}

Comparing the Lemma above with \eqref{c1values}, we see that $c_1(\jhat) - c_1(\sdgs) = 2 \ell_1^*$, where $\ell_1^*\in H^2(P;\zee)$ is Kronecker dual to $\ell_1$. Strictly, $\sdgs$ is defined on $P\# \cee P^2$ and here we refer to the restriction of its Chern class on $P$. We know, however, that $c_1(\sdgs)$ evaluates on the generator of $H_2(\cee P^2)$ as 3, and so $c_1^2(\sdgs)_{P\#\cee P^2} = c_1^2(\sdgs)_P + 9$ (using a subscript to indicate the manifold on which we consider a cohomology class).

Recall that for an oriented 2-plane field $\eta$ having torsion Chern class on an oriented 3-manifold $M$, the rational number $d_3(\eta)$ is the ``3-dimensional invariant'' of Gompf \cite{Gompf}, defined by 
\[
4d_3(\eta) = c_1^2(Z,J) - 3\sigma(Z) - 2\chi(Z)
\]
for any almost-complex 4-manifold $(Z,J)$ having $\partial Z = M$ and such that $TM\cap J(TM) = \eta$. From this definition, we find
\begin{eqnarray*}
d_3(\xidgs) &=& \frac{1}{4}(c_1^2(\sdgs)_{P\#\cee P^2} - 3\sigma(P\#\cee P^2) - 2\chi(P\#\cee P^2))\\
&=& \frac{1}{4}(c_1^2(\sdgs)_P - 3\sigma(P) -2\chi(P)) + 1
\end{eqnarray*}
and therefore using $\jhat$ on $P$ to compute $d_3(\xi_{Stein})$ gives
\begin{eqnarray}
d_3(\xi_{Stein}) - d_3(\xidgs) &=& \frac{1}{4}(c_1^2(\jhat)_P - c_1^2(\sdgs)_P) - 1\nonumber\\
&=& (c_1(\sdgs)\cup \ell_1^*) + (\ell_1^*)^2 - 1.\label{d3difference}
\end{eqnarray}
To evaluate this, identify the classes $\ell_j$ with their Poincar\'e duals in $H^2(P;\cue)$ (since we must pass to rational coefficients to evaluate the expressions above), and write $\ell_1^* = \sum m_j\ell_j$ for some coefficients $m_j$. Clearly, $(m_1,\ldots, m_n)$ is the first column of the inverse  to the intersection matrix of $P$. The intersection matrix is $\MM_1$, which is obtained from $M(a_1+1,a_2,\ldots a_n)$ by deleting its first column (c.f. \eqref{Mdef}). From the discussion after \eqref{Mdef}, we have that $\MM_1\y_1 = (y_0,0,\ldots,0)$, where the $y_j$ are as in that discussion and $\y_1$ is the vector $(y_1,\ldots, y_n)$. Dividing through by $y_0$ gives that the first column of $\MM_1^{-1}$ is $(\frac{y_1}{y_0},\ldots, \frac{y_n}{y_0})$, and now it is straightforward to see that 
\[
c_1(\sdgs)\cup \ell_1^* = (a_1+1, a_2+2, \ldots, a_n+2)\cdot (\frac{y_1}{y_0},\ldots, \frac{y_n}{y_0})^T = \frac{1}{y_0}(y_0 - y_1 + 1).
\]
Now note that $(\ell_1^*)^2$ is just the $(1,1)$ entry of $\MM_1^{-1}$, which we have seen to be $\frac{y_1}{y_0} = [a_1+1,a_2,\ldots,a_n]^{-1}$. Returning with this to \eqref{d3difference} proves
\[
d_3(\xi_{Stein}) - d_3(\xidgs) = \frac{1}{y_0}(y_0 - y_1 + 1) + \frac{y_1}{y_0} - 1 = \frac{1}{y_0} = -\frac{1}{q},
\]
where the last equality is the fact that $y_0 = -y = -q$ observed in the proof of Lemma \ref{Jevaluationlemma}.

To relate this to the surgered contact structure $\xi^-_{x/y}$, observe that by additivity of the 3-dimensional invariant,
\[
d_3(\xi^-_{x/y}) - d_3(\xi\#\xidgs) = \frac{1}{4}(c_1^2(\sdgs)_{W_{p/q}} - 3\sigma(W_{p/q}) - 2\chi(W_{p/q})).
\]
Here we assume that $\xi$ and $\xi^-_{x/y}$ have torsion first Chern class. Note that $d_3(\xi\#\xidgs) = d_3(\xi) + d_3(\xidgs) + \frac{1}{2}$, so the previous two equations give
\begin{equation}\label{Jc1}
 \frac{1}{4}(c_1^2(\sdgs)_{W_{p/q}} - 3\sigma(W_{p/q}) - 2\chi(W_{p/q})) = d_3(\xi^-_{x/y}) - d_3(\xi) - d_3(\xi_{Stein}) - \frac{1}{q} - \frac{1}{2}.
 \end{equation}
This last equation essentially characterizes $\sdgs$ in terms of contact geometric data that we can carry to our consideration of $\srobs$.

Turning now to  $\srobs\in\Spinc(-W_{p/q})$, observe that the map induced in $\hfhat$ by $\srobs$ carries the class $c(\xi)\otimes \tilde{c}$ to $c(\xi^-_{x/y})$. Therefore we must have
\begin{eqnarray*}
\deg(\widehat{F}_{-W_{p/q},\srobs}) &=& \deg(c(\xi^-_{x/y})) - \deg(c(\xi)\otimes \tilde{c})\\
&=& -d_3(\xi^-_{x/y}) +d_3(\xi) - \deg(\tilde{c}).
\end{eqnarray*}
In the above we are using the fact that for a contact structure $\xi$ with torsion Chern class the contact invariant lies in degree $-d_3(\xi) - \frac{1}{2}$ of Heegaard Floer homology (of the oppositely-oriented 3-manifold).

On the other hand, the degree shift formula in Heegaard Floer homology tells us
\[
\deg(\widehat{F}_{-W_{p/q},\srobs}) = \frac{1}{4}\left(c_1^2(\srobs)_{-W_{p/q}} - 3\sigma(-W_{p/q}) - 2\chi(-W_{p/q})\right).
\]
Hence 
\begin{equation}\label{robsc1}
\frac{1}{4}(c_1^2(\srobs)-3\sigma(-W_{p/q}) -2\chi(-W_{p/q})) = -d_3(\xi^-_{x/y}) +d_3(\xi) - \deg(\tilde{c}).
\end{equation}
Our goal now is to determine the quantity $\deg(\tilde{c})$.

Recall that $\tilde{c}$ is a class in the Floer homology of the 3-manifold represented by the open book $(S'',\phi'')$ on the right side of Figure \ref{Figure H}. The corresponding Heegaard diagram is drawn explicitly in Figure \ref{embHdiag}, where the class $\tilde{c}$ is given by the canonical intersection generator $\x''$, with basepoint $z''$.

\begin{figure}
\includegraphics[width=4.5in]{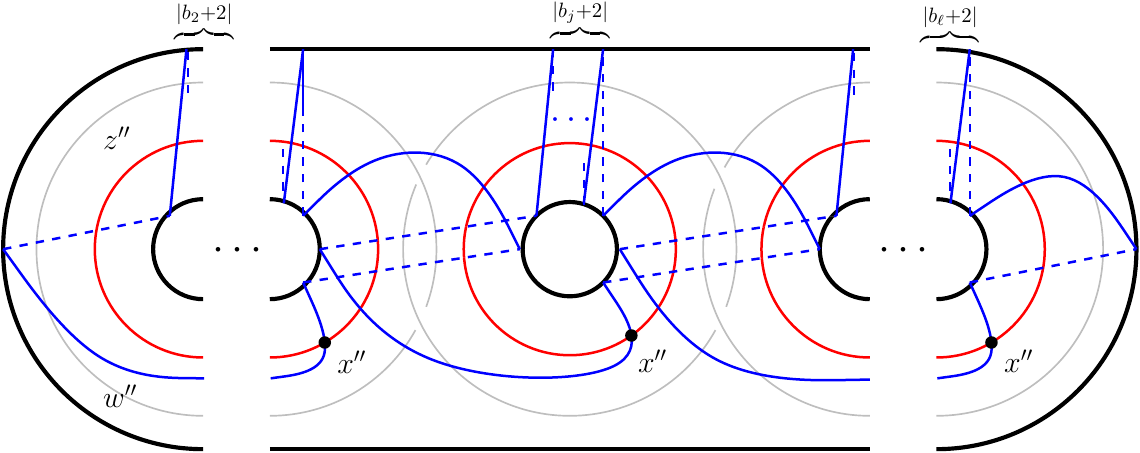}
\caption{\label{embHdiag}A genus $(\ell-1)$ Heegaard diagram arising from the HKM procedure applied to the open book $(S'',\phi'')$ of Figure \ref{Figure H}. It can also be seen as an embedded Heegaard diagram for $L(q,r)$ if the surface is oriented ``inward.'' Here $-q/r = [b_2,\ldots, b_{\ell}]$.}
\end{figure}

It is worth pausing for a moment with this diagram, which we write as $(\Sigma'', \aalpha'', \bbeta'')$. By construction, if equipped with basepoint $w''$, it is an HKM diagram associated to the open book $(S'', \phi'')$ and describing the 3-manifold $L(q,q-r)$. We can also see the diagram as embedded in a surgery picture for $L(q,r)$. Indeed, if $L(q,r)$ is described as surgery on a link of unknots with coefficients $b_j$, where $-q/r = [b_2,\ldots, b_{\ell}]$ for $b_j\leq -2$, we can surround the link by the Heegaard surface $\Sigma''$ as suggested in Figure \ref{embHdiag} (where the link is drawn in light grey). The $\alpha$-circles in the diagram bound disks in the complement, while the $\beta$ circles are exactly $b_j$-framed longitudes of the link components and therefore bound disks after the surgery. We therefore obtain a Heegaard diagram for $L(q,r)$, but notice that in this description the Heegaard surface must be oriented by an inward normal (as the boundary of the $\alpha$-handlebody; this is consistent with $L(q,r) = -L(q,q-r)$). Observe that the surgery coefficients $b_j$ correspond to the number of twists in each $\beta$ circle at the ``top'' of the diagram, which  correspond to the number of parallel copies of each stabilization of $K$ appearing in the Ding-Geiges-Stipsicz algorithm. Explicitly, if $K^{-r}$ appears $n_r$ times in the algorithm for contact $x/y$ surgery on $K$, then for $j = 2,\ldots, \ell-1$ there are $n_r$ extra twists in the $j$-th $\beta$ circle, while there are $n_\ell-1$ extra twists in the $\ell$-th. Correspondingly, the surgery coefficients $b_j$ are given by $b_{j} = -n_j-2$ for $j = 2,\ldots, \ell-1$, while $b_{\ell} = -n_\ell -1$. The special case of $j = \ell$ arises because the last boundary component in $S''$ does not get a Dehn twist from the stabilization procedure, c.f. Figure \ref{Figure H}.

It is straightforward to check that the open book $(S'',\phi'')$ for $L(q,q-r)$ corresponds to the standard (universally tight) contact structure; in particular it is the boundary of the complex structure $\jhat$ on the plumbing manifold $P$ considered before. Hence the intersection point $(\x'', w'')$ represents the contact invariant $c(\xi_{Stein})$ of the universally tight contact structure, which is a generator of the group $\hfhat(-L(q,q-r),\s_{Stein})$. Changing the basepoint gives the element $\tilde{c}$, which is represented by $(\x'',z'')$ and now lies in a different \spinc structure. In fact, by \cite[Lemma 2.19]{OS1}, we see that $\s_{z''}(\x'') = \s_{Stein} + PD[\mu]$, where $\mu$ is a closed circle on the Heegaard diagram dual to the leftmost $\alpha$ circle in Figure \ref{embHdiag} and oriented so as to cross $\alpha$ once when traveling from $w''$ to $z''$. Thinking of the diagram as embedded in the surgery picture for $L(q,r) = -L(q,q-r)$, we have that $\mu$ is just a positively-oriented meridian to the first component of the chain of unknots. 

Recall that the Floer homology group $\hfhat(L(q,r),\s)$ in any \spinc structure $\s$ is 1-dimensional over $\F$, and lies in a grading denoted $d(L(q,r),\s)\in\cue$. These grading levels were computed recursively by Ozsv\'ath and Szab\'o in \cite[Proposition 4.8]{OSabsgrad}; in our orientation convention their formula reads
\begin{equation}\label{dinvts}
d(L(q,r), i) = \frac{1}{4qr}\left(qr - (2i+1-q-r)^2\right) - d(L(r,q), i).
\end{equation}
Here the index $i\in\{0,\ldots, q-1\}$ refers to a particular labeling of the \spinc structures on $L(q,r)$, and in the $d$-invariant on the right we reduce $q$ and $i$ modulo $r$. The labeling is obtained by drawing a genus-1 Heegaard diagram for $L(q,r)$ where the $\alpha$ circle has slope 0 and the $\beta$ circle has slope $-q/r$, then marking the $q$ intersection points sequentially around the $\alpha$ circle, starting with the intersection point adjacent to the basepoint (c.f. Figure 2 of \cite{OSabsgrad}). This diagram can be seen as an ``embedded'' diagram as above, where there is now just a single unknot with coefficient $-q/r$ describing the surgery and the Heegaard surface is oriented inward. The usual sequence of Kirby moves from the integer surgery on a chain to this picture clearly takes the meridian $\mu$ to a meridian of this unknot, and moreover it is easy to see that changing a \spinc structure by adding the dual to $[\mu]$ corresponds to increasing the label by $i\mapsto i+r$. 

 The universally tight contact structure we are denoting by $\xi_{Stein}$ on $L(q,q-r)$ has the property that with the \spinc structure labeling above we have $c(\xi_{Stein})\in \hfhat(L(q,r), i = 0)$. Thus from the discussion above we have
 \[
 \deg(\tilde{c}) - \deg(c(\xi_{Stein})) = d(L(q,r), r) - d(L(q,r), 0) = 1- \frac{1}{q},
 \]
 where the second equality is a simple exercise with \eqref{dinvts}. Now recall that $\deg(c(\xi_{Stein}))=-d_3(\xi_{Stein}) - \frac{1}{2}$. Making this replacement above, and using the result to eliminate $\deg(\tilde{c})$ from \eqref{robsc1} gives
 \begin{eqnarray*}
\ts\frac{1}{4}(c_1^2(\srobs)_{-W_{p/q}} -3\sigma(-W_{p/q}) - 2\chi(-W_{p/q})) &&\\
&&\hspace*{-2in}\ts= -d_3(\xi^-_{x/y}) + d_3(\xi) + d_3(\xi_{Stein}) + \frac{1}{q} - \frac{1}{2}\\
&&\hspace*{-2in}\ts= -\frac{1}{4}(c_1^2(\sdgs)_{W_{p/q}} - 3\sigma(W_{p/q}) - 2\chi(W_{p/q})) -1,
 \end{eqnarray*}
 where we have used \eqref{Jc1}. Observe that $\sigma(-W_{p/q}) = -\sigma(W_{p/q})$ and $\chi(-W_{p/q}) = \chi(W_{p/q}) = 1$ to conclude:
 \[
 c_1^2(\srobs)_{-W_{p/q}} = -c_1^2(\sdgs)_{W_{p/q}}.
 \]
 
 \begin{corollary}\label{c1cor} Suppose that $\xi$ and $\xi^-_{x/y}$ have torsion first Chern class. Then the \spinc structure $\srobs$ carrying $c(\xi)\otimes \tilde{c}$ to $c(\xi^-_{x/y})$ satisfies
 \[
 \langle c_1(\srobs), [S]\rangle = \pm \langle c_1(\sdgs),[S]\rangle.
 \]
for $[S]\in H_2(W_{p/q}; \zee)$ a generator.
\end{corollary}
This follows since, unless $p/q = 0$, the nondegenerate part of the intersection form of $W_{p/q}$ is 1-dimensional, corresponding to the class $[S]$. If $p/q=0$, then much of the above discussion does not apply (since typically $\xi^-_{x/y}$ will have non-torsion Chern class), but the conclusion of the corollary still holds simply because in this case $\srobs$ and $\sdgs$ are determined by their restriction to $Y_{p/q}(K)$, where they agree---in particular, in this case we get that the sign is $+1$.

Note that the corollary implies that $\srobs$ and $\sdgs$ are equal up to conjugation (and possibly elements of order 2 in $H_2(W_{p/q};\zee)$). Furthermore, since $\srobs$ and $\sdgs$ both restrict to $\s_{\xi^-_{x/y}}$on $Y_{p/q}(K)$, it follows that unless $\s_{\xi^-_{x/y}}$ is self-conjugate the sign appearing in the corollary must be $+1$. In practice we are usually interested in algebraic properties of the map induced by $\srobs$ in Floer homology (e.g., injectivity), and since these are insensitive to conjugation the sign is immaterial.

\subsection{HKM Strong}\label{strongsec}

In order to apply the naturality property of reducible open book surgeries (Theorem \ref{naturalitythm}), and complete the proof of Lemma \ref{surgcoblemma} we must verify that Heegaard diagrams of the sort in Figure \ref{embHdiag}, arising from planar open book decompositions $(S'',\phi'')$ as on the right of Figure \ref{Figure H}, are HKM strong: that is, the canonical generator for Heegaard Floer homology is the only one in its \spinc structure. This is essentially independent of the rest of our arguments, and here we reproduce the relevant Heegaard diagram with more natural notation. Figure \ref{embHdiag2} shows an ``embedded'' Heegaard diagram describing the lens space $L(x_0, x_1)$, where $\frac{x_0}{x_1} = [c_1,\ldots, c_n]$ and each $c_j \geq 2$ (the Heegaard surface should be considered as oriented by an inward-pointing normal). As before, the diagram can be seen either as arising from the  description of $L(x_0,x_1)$ as the result of surgery along a chain of unknots---the ``embedded'' picture---or from the HKM procedure applied to a planar open book decomposition. The intersection point $\V = (v_1,\ldots, v_n)$ is the canonical generator in the latter description.

\begin{figure}
\includegraphics[width=4.5in]{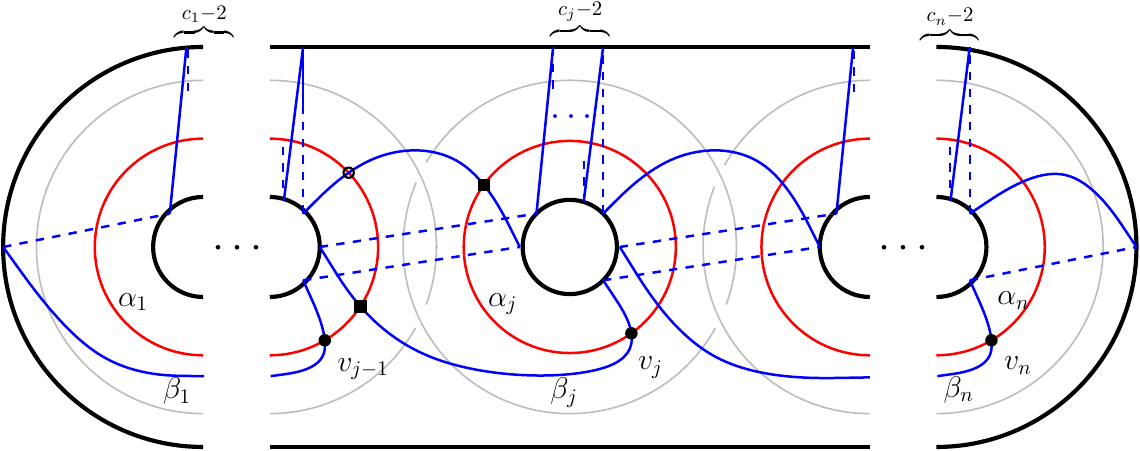}
\caption{\label{embHdiag2}}
\end{figure}

We number the $\alpha$ and $\beta$ curves in the diagram from left to right as shown. Observe that each $\beta_j$ intersects only $\alpha_{j-1}$, $\alpha_j$, and $\alpha_{j+1}$, and has just one intersection with each of $\alpha_{j-1}$ and $\alpha_{j+1}$ (with obvious modifications if $j = 1$ or $n$). If $\U\in \Ta\cap\Tb$ is a Heegaard Floer generator, write $\U = (u_1,\ldots, u_n)$ where $u_j\in \alpha_j\cap \beta_{\sigma(j)}$ for a permutation $\sigma$. We will say that $\U$ is a ``paired generator'' if $\sigma(j) = j$ for each $j$, and a ``non-paired'' generator otherwise; similarly an individual intersection point between $\alpha$ and $\beta$ curves is ``paired'' if it lies on $\alpha_j\cap \beta_j$ for some $j$. Our first observation is that a non-paired generator is always in the same \spinc equivalence class as a paired one. Indeed, by the way in which the $\beta$ curves intersect the $\alpha$'s, it is easy to see that a non-paired generator must have components $u_{j-1}$, $u_j$ appearing as the intersections marked with squares in Figure \ref{embHdiag2}, for possibly several values of $j$.  For such $j$ there is an obvious quadrilateral in the figure connecting  $u_{j-1}, u_j$ with two paired intersection points, namely $v_j$ and the intersection point marked with an open circle. Hence to determine \spinc equivalence classes of generators, it suffices to consider only paired generators. 

Recall that the difference between the \spinc structures induced by intersection points $\U,\WW\in \Ta\cap\Tb$ is measured by a class $\epsilon(\U,\WW)\in H_1(L(x_0,x_1);\zee)$ that is the union of 1-chains on the Heegaard surface: one traces paths on $\alpha$-circles from the components of $\U$ to the components of $\WW$, then returns to $\U$ along the $\beta$ circles. We wish to see that for any (paired) generator $\U$ distinct from the canonical generator $\V$, the class $\epsilon(\V, \U)$ is nonzero.

Let $\U$ be a paired generator. If $u_j$ is a component of $\U$ that is distinct from $v_j$, we can construct a 1-cycle $\epsilon(u_j)$ by following $\alpha_j$ counterclockwise in the diagram from $v_j$ to $u_j$, then turning right to follow $\beta_j$ back to $v_j$. Orient the components of the chain of unknots (light grey in Figure \ref{embHdiag2}) counterclockwise and write $\mu_j$ for the oriented meridian of the $j$-th component: then by inspecting Figure \ref{embHdiag2} we see $\epsilon(u_j)$ is homologous to $r_j(\U)\mu_j - \mu_{j+1}$, for some integer $r_j(\U)$ with $1\leq r_j(\U)\leq c_j - 1$. If the coordinate $u_j$ coincides with $v_j$, we set $\epsilon(u_j) = 0$. Then we have
\[
\epsilon(\V, \U) = \sum_j \epsilon(u_j) = \sum_j n_j \mu_j,
\]
for some coefficients $n_j$ satisfying $-1\leq n_j\leq c_j - 1$ for each $j$. 

To understand the class in $H_1(L(x_0,x_1);\zee)$ corresponding to this element, observe that the linking matrix coming from the surgery diagram is a presentation matrix for $H_1(L(x_0,x_1);\zee)$ in which the meridians $\mu_1,\ldots,\mu_n$ provide a generating set. This linking matrix is
\[
\left[\begin{array}{rrrrr} -c_1 & 1 & & & \\ 1 & -c_2 &  & & \\ & & \cdots & & \\ & & & & 1 \\ & & & 1 & -c_n \end{array}\right] \sim \left[\begin{array}{ccccc} 0 & 1 & & & \\ & 0& 1 & & \\ & & & \cdots & \\ & & & & 1 \\ -x_0 & -x_1 & -x_2 & & -x_{n-1} \end{array}\right]
\]
where we have applied a sequence of column operations, and the $x_j$ satisfy the recursive relations $x_{n+1} = 0$, $x_n = 1$ and $x_j = c_{j+1}x_{j+1} - x_{j+2}$ for $0\leq j \leq n-2$ (c.f. \eqref{recursion}). In particular, the order of the homology is $x_0$, and the lens space is $L(x_0,x_1)$. From the second version of the presentation matrix we see $\mu_n$ is a generator for $H_1(L(x_0,x_1);\zee) \cong \zee/x_0\zee$, and for $1\leq j \leq n-1$ we have $\mu_j = x_j \mu_n$. Thus in $H_1$, 
\[
\epsilon(\V,\U) = \left(\sum_j n_j x_j\right) \mu_n
\]
and to see this is nontrivial it suffices to show that $0< \sum n_j x_j < x_0$ for the coefficients $n_j$ arising from a paired generator $\U\neq \V$. 

We have observed already that $-1\leq n_j\leq c_j - 1$ for each $j$. By considering the classes $\epsilon(u_j)$ described above, it is easy to say a little more:
\begin{itemize}
\item[(a)] For $j\geq 2$, we have $n_j \leq c_j - 2$ unless $\epsilon(u_{j-1}) = 0$.
\item[(b)] If $\epsilon(u_{j}) =0$ then either $n_j = 0$ or $n_j = -1$, where the latter occurs if and only if $\epsilon(u_{j-1}) \neq 0$.
\item[(c)] If $n_j  = -1$ for some $j$, then there exists $j' < j$ with $n_{j'} >0$, and moreover $n_{j''} = 0$ for all $j' < j''< j$.
\end{itemize}

{\it Claim 1:} If $\epsilon(\V,\U) = \left(\sum n_j x_j\right) \mu_n$ as above, then $\sum n_j x_j > 0$. 

Since not all $n_j$ can vanish if $\U\neq \V$, this is obvious if all $n_j \geq 0$. Suppose $j_0$ is the largest index with $n_{j_0} = -1$; then we have $\ds\sum_{j\geq j_0} n_j x_j \geq -x_{j_0}$. If $j_1< j_0$ is the next smaller index with $n_{j_1}>0$ as in point (c) above, then $\sum_{j\geq j_1} n_j x_j \geq x_{j_1} - x_{j_0} >0$ since the $x_j$ form a strictly decreasing sequence. Repeat this argument inductively.

{\it Claim 2:} If $\sum n_j x_j$ is a linear combination of the integers $x_j$ with $n_1 \leq c_1 - 1$ and $n_j \leq c_j-2$ for all $j>1$, then $\sum n_j x_j < x_0$. 

To see this use the recursive relations among the $x_j$ to write 
\[
\sum_j n_j x_j \leq (c_1-1)x_1 + (c_2-2)x_2 + \cdots (c_n - 2)x_n= x_0 - x_n < x_0.
\]

Now suppose $\epsilon(\V,\U)= (\sum n_j x_j) \mu_n$. We have just seen that if $n_j \leq c_j -2$ for $j\geq 2$ then $\epsilon(\V,\U)\neq 0$ in the first homology of $L(x_0,x_1)$, while we know $n_j \leq c_j -1$ for all $j$. It is easy to see that if $n_{k} = c_{k} - 1$ for some $k>1$ then necessarily $\epsilon(u_{k-1}) = 0$, and hence from point (b) above we know $n_{k-1}$ is either 0 or $-1$. In the second case we have
\begin{eqnarray*}
\sum_j n_j x_j &=& \cdots + n_{k-2} x_{k-2}  - x_{k-1} + (c_{k}-1)x_{k} + n_{k+1} x_{k+1} + \cdots \\
&< & \cdots + n_{k-2} x_{k-2} + n_{k+1} x_{k+1} + \cdots ,
\end{eqnarray*}
where we have used $x_{k-1} = c_{k} x_{k} - x_{k + 1} > (c_{k}-1)x_{k} $, following from the recursion for the $x_j$ and the fact that the $x_j$ are decreasing. In the last expression the coefficients of $x_{k - 1}$ and $x_{k}$, both being zero, are no more than $c_{k - 1} -2$ and $c_{k} -2$, respectively.

Suppose $j_0$ is the largest index such that $n_{j_0} = c_{j_0} -1$. As we have just seen, if the preceding coefficient $n_{j_0 -1} = -1$, we can bound the sum $\sum n_j x_j$ by a linear combination of $x_j$ in which two more of the coefficients satisfy the hypothesis of Claim 2 above, then repeat the argument (a little thought shows that the new sum does in fact arise from a difference class $\epsilon(\V,\U)$ for some $\U$, but regardless it is clear that the coefficients $n_j$ for $j< j_0-1$ are unchanged and therefore still satisfy (a), (b), and (c) above). In the other case, that $n_{j_0-1} = 0$, we have $\epsilon(u_{j_0-2}) = 0$ and therefore $n_{j_0-2}$ is also either 0 or $-1$. Continuing inductively, we can therefore bound $\sum n_j x_j$ either by a combination of $x_j$ that satisfies the hypotheses of Claim 2, or a sum satisfying those hypotheses except that for some $j_0$ we have $n_1 = n_2 = \cdots = n_{j_0-1} = 0$ and $n_{j_0} = c_{j_0} -1$. In the latter case, simply observe
\begin{eqnarray*}
\sum n_j x_j &=& (c_{j_0} -1)x_{j_0} + n_{j_0 + 1}x_{j_0 +1} + \cdots\\
& <&  x_{j_0-1} + n_{j_0 + 1}x_{j_0 +1} + \cdots\\
& \leq& x_1  + n_{j_0 + 1}x_{j_0 +1} + \cdots,
\end{eqnarray*}
and since the coefficient of $x_1$ is $1 \leq c_1 -1$, the last expression satisfies the hypotheses of Claim 2 and is therefore less than $x_0$. This completes the proof that the diagram of Figure \ref{embHdiag2} is HKM strong.

\subsection{General surgery coefficient}\label{gencoeffsubsec}

The results of the preceding subsections suffice to prove Theorem \ref{naturalitythmintro} in the case that the contact surgery coefficient $\frac{x}{y}$ is at least 1. To allow general positive coefficients we use the following result, which generalizes \cite[Proposition 2.4]{LS2011} (see also \cite{conway14}).

\begin{lemma} Let $K\subset (Y,\xi)$ be an oriented nullhomologous Legendrian knot and $0<\frac{x}{y}\in \cue$ a positive contact framing. Then the contact structure $\xi^-_{x/y}(K)$ obtained by contact $\frac{x}{y}$ surgery along $K$ is isomorphic to the contact structure $\xi^-_{1+x/y}(K^-)$ obtained by contact $1+\frac{x}{y}$ surgery along the negatively stabilized knot $K^-$.
\end{lemma}

\begin{proof} Let $r\in\zee$ be the minimal positive integer with $\frac{x}{y-rx} <0$, and write the continued fraction expansion $\frac{x}{y-rx} = [a_1,\ldots, a_n]$ as in Theorem \ref{DGSalgorithm}. Observe that $\frac{x}{y}+1 = \frac{x+y}{y} >1$, and the relevant continued fraction expansion for this surgery coefficient is $\frac{x+y}{y-(x+y)} = -\frac{x+y}{x} = -1-\frac{y}{x} = [b_1,\ldots, b_m]$ for some $b_j$. It is easy to see that in fact
\[
[b_1,\ldots, b_m] = [-(r+1), a_1,\ldots, a_n]
\]
(just note $-(r+1) - \frac{1}{\ts\frac{x}{y-rx}} = -1-\frac{y}{x}$). Therefore, having placed $K$ on the page of an open book supporting $\xi$ as before, we compare the results of the DGS algorithm applied to $\frac{x}{y}$ surgery on $K$ and to $1+\frac{x}{y}$ surgery on $K^-$:
\begin{itemize}
\item For $\frac{x}{y}$ surgery on $K$, we apply $r$ left-handed Dehn twists along curves parallel to $K$ in the page, then right handed Dehn twists along stabilized pushoffs according to the coefficients $[a_1,\ldots, a_n]$.
\item For $1+\frac{x}{y}$ surgery on $K^-$, we apply a single left Dehn twist along $K^-$ in the page, then right handed Dehn twists along stabilized pushoffs according to $[b_1,\ldots, b_m]$, where as just observed we have $b_1 = -(r+1)$, and $(b_2,\ldots, b_m ) = (a_1,\ldots,a_n)$ (in particular $m = n+1$).
\end{itemize}
\begin{figure}[t]
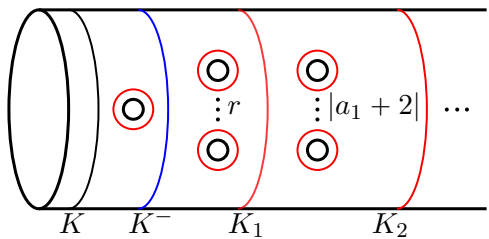
\caption{\label{relationfig1}Open book describing $1+\frac{x}{y}$ contact surgery along the negatively stabilized knot $K^-$. The diagram is to be interpreted as a subsurface of the page of the open book, compare to Figure \ref{Figure G}. The knot $K$ is indicated by the black circle at the left of the diagram (and does not get a Dehn twist). Circles in red indicate that a right Dehn twist is performed while the blue circle corresponding to $K^-$ indicates a left Dehn twist.}
\end{figure}

The situation for $1+\frac{x}{y}$ surgery on $K^-$ is illustrated in Figure \ref{relationfig1}. To relate this to the diagram describing $\frac{x}{y}$ surgery on $K$, we recall the {\it daisy relation} from \cite{EMVHM, olgaVHM2010}. In our context, the relation described in \cite[Figure 2]{EMVHM} or \cite[Figure 11]{olgaVHM2010} can be drawn as in Figure \ref{daisyrelation}. 
\begin{figure}[b]
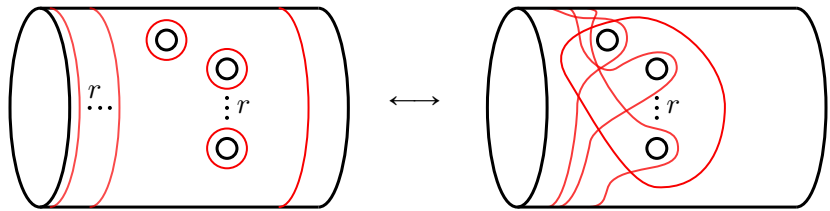
\caption{\label{daisyrelation}The daisy relation between compositions of right Dehn twists, drawn on a cylinder with $r+1$ holes. The order of twists on the curves on the right is from ``bottom to top,'' but that will not be important for our purpose.}
\end{figure}
Moving the parallel twists around the left boundary component in that figure to the other side of the relation, it is easy to see that Figure \ref{relationfig1} is equivalent to Figure \ref{relationfig2}(a). There is clearly an arc connecting different boundary components of the page and intersecting $c_r$ exactly once, which we can use to destabilize the open book. Then a similar situation holds sequentially for the curves $c_{r-1}, c_{r-2}, \ldots$ , so we can make a sequence of destabilizations to obtain Figure \ref{relationfig2}(b). The latter open book is just the description of contact $\frac{x}{y}$ surgery on $K$.
\end{proof}

\begin{figure}
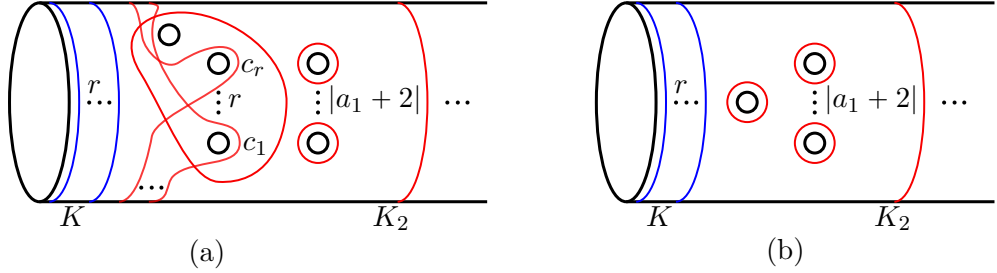
\caption{\label{relationfig2}In (a), the result of applying the daisy relation to the open book described in Figure \ref{relationfig1}. Diagram (b) results from sequential destabilizations along the circles $c_1,\ldots c_r$; observe in (b) we see $r$ left Dehn twists along $K$ and a right Dehn twist along $K_2$, which is an $|a_1+1|$-stabilized copy of $K$.}
\end{figure}

\begin{proof}[Proof of Theorem \ref{naturalitythmintro}] First assume $\frac{x}{y}\geq 1$. Then the first part of Theorem \ref{naturalitythmintro} is Theorem \ref{naturalitythmspec}, whose proof is now complete. In particular, the \spinc structure on $-W$ that has $F_{-W, \s}(c(\xi)\otimes \tilde{c}) = c(\xi^-_{x/y})$ is $\srobs$, coming from the reducible open book surgery result Theorem \ref{naturalitythm}. The second part of Theorem \ref{naturalitythmintro} follows from Lemma \ref{Jevaluationlemma} determining the evaluation of $c_1(\sdgs)$ on the homology generator $[S]$ of the surgery cobordism, together with Corollary \ref{c1cor} that shows $\srobs$ and $\sdgs$ have the same pairing with $[S]$ (up to sign, and in the torsion case).

If $0<\frac{x}{y}< 1$, then by the preceding lemma $c(\xi_{x/y}^-(K)) = c(\xi_{1+x/y}^-(K^-))$, in obvious notation. Applying the part of the theorem just proved to the contact $1+\frac{x}{y}$ surgery, note that  the smooth surgery coefficient $\frac{p}{q}$ corresponding to contact $\frac{x}{y}$ surgery on $K$ is the same as that corresponding to contact $1+\frac{x}{y}$ surgery on $K^-$. Hence the corresponding surgery cobordisms $W$ are smoothly identical, and the same is true for the homomorphisms $F_{-W}$. Furthermore, since $\rot(K^-) - \tb(K^-) = \rot(K) -\tb(K)$, we see that the \spinc structure on $W$ specified in the case of $1+\frac{x}{y}$ surgery is the same as the one claimed in part 2 of Theorem \ref{naturalitythmintro} for $\frac{x}{y}$ surgery.
\end{proof}

Finally, we spell out the proof of Corollary \ref{cor1}. We recall the statement: suppose $\K\subset (S^3,\xi_{std})$ is a Legendrian knot in the standard contact structure on $S^3$ and write $K$ for the underlying smooth knot type. Fix a contact surgery coefficient $\frac{x}{y}>0$ and let $\frac{p}{q}$ be the corresponding smooth surgery coefficient. We claim that under the identification of $\hfhat(-S^3_{p/q}(K))$ with $H_*(\XX_{-p/q}(-K))$, where $-K$ is the mirror of $K$, the contact invariant $c(\xi^-_{x/y})$ (or possibly its conjugate) is given by the image in homology of the inclusion
\[
(k, {B})\to \XX_{-p/q}(-K)
\]
where 
\[
2k = (\rot(\K) - \tb(\K) + 1)q -2.
\]
Put another way, the contact invariant is equal to the image in homology of the copy of ${B}$ in $\XX_{-p/q}(-K)$ that is the target of the map $v_k :(k,{A}_s(-K))\to {B}$, where $s = \lfloor \frac{k}{q}\rfloor = \frac{1}{2}(\rot(\K) - \tb(\K) - 1)$.

\begin{proof}[Proof of Corollary \ref{cor1}] By the argument in the proof of Theorem \ref{naturalitythmintro} above, it suffices to assume $\frac{x}{y}\geq 1$. Let $W_{p/q}(K): -L(q,r)\to S^3_{p/q}(K)$ be the rational surgery cobordism of Figure \ref{Figure P}. According to Theorem \ref{naturalitythmspec}, the manifold $-W_{p/q}(K) = W_{-p/q}(-K)$ equipped with the \spinc structure $\srobs$ carries the generator $\tilde{c}\in \hfhat(L(q,r))$ to the contact invariant $c(\xi^-_{x/y})\in \hfhat(-(S^3_{p/q}(K)))$. Moreover, by Lemma \ref{Jevaluationlemma} and Corollary \ref{c1cor}, $\srobs$ has the property that (possibly replacing $\srobs$ by its conjugate)
\[
\langle c_1(\srobs), [\widetilde{S}]\rangle = p + (\rot(\K) -\tb(\K))q - 1.
\]
By Theorem \ref{ratsurgformula} (c.f. Corollary \ref{inclusioncor}), the map induced by $\srobs$ corresponds to the inclusion of $(k,{B})$ where $k$ is characterized by $\langle c_1(\srobs), [S]\rangle -p + q - 1 = 2k$ since we are working on $W_{-p/q}(-K)$. The conclusion follows.

\end{proof}

\section{Proof of Theorem \ref{cinvariantthm}}\label{lastsection}

Recall that Theorem \ref{cinvariantthm} asserts necessary and sufficient conditions for the nonvanishing of the contact invariant $c(\xi_{x/y}^-)$ obtained by contact $\frac{x}{y}$ surgery on a Legendrian in $S^3$. Using Corollary \ref{cor1}, our approach to the proof will be to characterize the conditions under which the inclusion of $(k, B)$ in the mapping cone $\XX_{-p/q}(-K)$ induces a nontrivial map in homology, where 
\[
k = -\frac{1}{2}(\tb(\K) -\rot(\K) + 1)q + q -1
\]
(c.f. \eqref{kchar}). Recall that $(k,B)$ is the target of the maps $v_{\lfloor k/q\rfloor}: A_{\lfloor\frac{k}{q}\rfloor}\to B$ and $h_{\lfloor(k+p)/q\rfloor}: A_{\lfloor\frac{k+p}{q}\rfloor} \to B$, and note that $\lfloor \frac{k}{q}\rfloor = -\frac{1}{2}(\tb(\K) - \rot(\K) + 1)$, while
$\lfloor\frac{k+p}{q}\rfloor = -\frac{1}{2}(\tb(\K) - \rot(\K) +1) + 1 + \lfloor\frac{p-1}{q}\rfloor$.

\begin{lemma}\label{claim} For $k$ as above,  
\begin{itemize}

\item $v_{\lfloor k/q\rfloor}$ is trivial in homology if and only if both 
\[
\nu(-K) = -\tau(K) + 1 \quad\mbox{and}\quad \frac{1}{2}(\tb(\K) -\rot(\K) + 1) =\tau(K).
\]
\item $h_{\lfloor(k+p)/q\rfloor}$ is trivial in homology if and only if 
\[
\frac{p}{q} > -\nu(-K) + \frac{1}{2}(\tb(\K) -\rot(\K) + 1).
\]
\end{itemize}
\end{lemma}

\begin{proof} Recall that the knot invariant $\nu(K)$ is defined to be the smallest value of $s$ such that $v_s: A_s\to B$ is surjective in homology. It is known that in fact $v_{s*}$ is surjective for all $s\geq \nu(K)$, and vanishes for $s <\nu(K)$. By symmetries of the knot complex (more specifically, conjugation invariance of maps induced by cobordisms together with \cite[Theorem 2.3]{OSintsurg}), the map $h_{s*}$ is surjective if and only if $v_{-s*}$ is surjective, and therefore $h_{s*}$ is trivial if and only if $s > -\nu(K)$. Since we are considering the mapping cone for the mirror image $-K$, we find that  $v_{\lfloor k/q\rfloor}$ is trivial in homology if and only if $-\frac{1}{2}(\tb(\K) - \rot(\K) + 1) <\nu(-K)$. Now, it is also known that $\nu(-K)$ equals either $-\tau(K)$ or $-\tau(K) +1$. On the other hand, by Plamenevskaya's result \eqref{plamresult} we always have the inequality $-\frac{1}{2}(\tb(\K) - \rot(\K) + 1) \geq -\tau(K)$, from which the first part of the lemma follows.

Turning to the second part, we have that $h_{\lfloor (k+p)/q\rfloor}$ vanishes in homology if and only if the $v$-map with ``opposite'' domain does. This is equivalent to
\begin{eqnarray*}
\ts-\lfloor \frac{k+p}{q}\rfloor < \nu(-K) &\iff &  \ts\frac{1}{2}(\tb(\K) - \rot(\K) + 1) - 1 - \lfloor\frac{p-1}{q}\rfloor < \nu(-K)\\
&\iff & \ts \lfloor \frac{p-1}{q}\rfloor > -\nu(-K) + \frac{1}{2}(\tb(\K) - \rot(\K) + 1) - 1\\
&\iff & \ts\frac{p}{q} > -\nu(-K) + \frac{1}{2}(\tb(\K) - \rot(\K) + 1),
\end{eqnarray*}
which is the second part of the lemma.
\end{proof}

We now prove the first claim of Theorem \ref{cinvariantthm}, which is that $c(\xi_{x/y}^-)$ vanishes if $\frac{1}{2}(\tb(\K) - \rot(\K) +1 ) < \tau(K)$. From the Lemma, this assumption means that $v_{\lfloor k/q\rfloor}$ is surjective in homology. We claim that there is a cycle  $a\in \Akq$ such that $v_{\lfloor k/q\rfloor*}([a])$ is a generator of $H_*(B)$, while $h_{\lfloor k/q\rfloor*}([a]) = 0$, which clearly shows that the generator of the homology of $(k,B)$ vanishes in the homology of $\XX_{-p/q}(-K)$. To see the claim, recall that $\Akq$ is the subquotient of the knot Floer complex typically described as $C\{\max(i,j-\lfloor\frac{k}{q}\rfloor) = 0\}$. Here $i$ and $j$ are the two filtrations on the knot Floer complex $CFK^{\infty}(-K)$; for details see \cite{OSknot}. By definition of $\tau$, there is a cycle $a$ in the vertical complex $B = C\{i = 0\}$ that is supported in $C\{i = 0, j\leq \tau(-K)\}$ and generates the homology of $B$.  Our assumption says that $\lfloor\frac{k}{q}\rfloor = -\frac{1}{2}(\tb(\K) -\rot(\K) + 1) > \tau(-K)$, hence $a$ can be considered as a cycle in $\Akq$, and since it lies in the subcomplex with $j < \lfloor\frac{k}{q}\rfloor$, it vanishes under $h_{\lfloor k/q\rfloor}$. This proves part (1) of Theorem \ref{cinvariantthm}.

Now we turn to part 2 of the theorem. The relevant portion of the mapping cone appears as:
\[
\begin{diagram}
(k,A_{\lfloor\frac{k}{q}\rfloor}) && (k+p, A_{\lfloor\frac{k+p}{q}\rfloor}) \\
\dTo^{v} & \ldTo_{h} \\
(k, B)
\end{diagram}
\]
We have drawn this figure under the assumption that $\frac{p}{q} > 0$, but the arguments are insensitive to this condition.

The condition $\nu(-K) = -\tau(K) + 1$ is equivalent to $\epsilon(K) = 1$, and therefore the assumptions in 2(a) of the theorem---that $\epsilon(K) = 1$ and $\tb(K) - \rot(K) = 2\tau(K) -1$---are equivalent to the condition that $v_{\lfloor k/q\rfloor}$ (which is to say $v_{\tau(-K)}$, given the assumptions) vanishes in homology. Hence in this case we get 
\[
\begin{diagram}
(k, A_{\tau(-K)}) && (k+p,A_{\tau(-K) + 1 + \lfloor\frac{p-1}{q}\rfloor} )\\
\dDashto^{v_* = 0} & \ldTo_{h} \\
(k,B)
\end{diagram}
\]
Turning to $h_{\lfloor (k+p)/q\rfloor}$, observe that by Lemma \ref{claim} the condition $\frac{p}{q} > 2\tau(K) - 1$ in 2(a) is equivalent in this situation to the condition that $h_{\lfloor (k+p)/q\rfloor}$ vanishes in homology. Hence for such $\frac{p}{q}$, the mapping cone picture becomes
\[
\begin{diagram}
(k,A_{\tau(-K)}) && (k+p, A_{s})  & \mbox{\quad where $s\geq -\tau(-K)$}\\
\dDashto^{v_* = 0} & \ldDashto_{h_*=0} \\
(k,B)
\end{diagram}
\]
and it is now clear that the generator of the homology of $(k,B)$ survives as a nonzero class in the homology of the mapping cone. This proves that in 2(a), the condition that the smooth surgery coefficient $\frac{p}{q} = \frac{x}{y} + \tb(\K)$ be greater than $2\tau(K) - 1$ suffices to ensure that $c(\xi_{x/y}^-) \neq 0$.

For the converse, we must see that $c(\xi_{x/y}^-) = 0$ when $\frac{x}{y} + \tb(\K)= \frac{p}{q} \leq 2\tau(K) - 1$. Observe that if $c(\xi_r^-) = 0$ for some $r>0$, then the contact invariant of $\xi_s^-$ vanishes for all $0<s<r$: indeed, for such $s$ the result of contact $r$ surgery is obtained from the contact $s$ surgery by a sequence of Legendrian surgeries, which preserve nonvanishing of the contact invariant. Hence it suffices to assume $\frac{p}{q} = 2\tau(K) - 1$ (this simplifies the arguments to follow somewhat, though they go through in general). In particular, we have $q = 1$. The portion of the mapping cone near $(k,B)$ now appears as follows, where we write $\bar{\tau}$ for $\tau(-K)$, and observe $\tau(-K) +p = -\tau(-K) -1$: 
\[
\begin{diagram}
A_{\btau} && A_{-\btau -1} \\
\dTo^{v_{\btau}} & \ldTo_{h_{-\btau -1}} & \dTo_{v_{-\btau -1}} \\
B & & \cdots
\end{diagram}
\]
Since $h_s$ is onto in homology if and only if $v_{-s}$ is onto, which is always true for $-s \geq \btau +1$ (recall we are working in the mapping cone for $-K$), we have that $h_{-\btau - 1}$ above is onto.  Hence, as long as there is a class $a$ in $H_*(A_{-\btau-1})$ with ${h_{-\btau -1}}_*(a) \neq 0$ and ${v_{-\btau -1}}_*(a) = 0$, we see that $c(\xi_r^-)$ is a boundary in $\XX_{-p/q}(-K)$. We claim this is always the case. Observe there is always a cycle $y\in C\{j = -\btau -1, i <0\}$  generating the homology of $C\{j = -\btau - 1\}$ by definition of $\tau$: indeed, this is the ``horizontal version'' of the statement that the homology of the vertical complex $C\{i = 0\}$ is generated by a cycle supported in $C\{i = 0, j < \btau + 1\}$. But such a chain $y$ then clearly determines a cycle in $A_{-\btau -1}$ that is carried to a generator of homology by $h_{-\btau - 1}$ and is in the kernel of $v_{-\btau -1}$. (Note that this argument does not actually require $\epsilon(K) = 1$, but we will need to prove stronger vanishing statements in other cases.) This completes the proof of 2(a) of Theorem \ref{cinvariantthm}.

For the remaining cases, the following will be useful.

\begin{lemma}\label{epsilonlemma} A knot $K\subset S^3$ has $\epsilon(K) = 0$ if and only if the chain maps
\[
v_\tau: A_{\tau(K)} \to B \quad\mbox{and}\quad h_\tau : A_{\tau(K)}\to B
\]
induce the same nontrivial map in homology (with coefficients in $\F$).
\end{lemma}

\begin{proof} First recall that $h_s$ is defined as the composition of the quotient map $A_s\to C\{j = s\}$, followed by a certain chain homotopy equivalence between the latter complex and $B = C\{i = 0\}$. At the level of homology, however, there is a unique identification $H_*(j = s) = H_*(B) = \F$. Since we are interested in the map on homology, for the purposes of this proof we consider $h_s$ to simply be the quotient map to $C\{j = s\}$, and leave implicit the homotopy equivalence.

For an integer $s$, we have the two complexes
\[
A_s = C\{\max(i, j-s) = 0\} \quad\mbox{and}\quad A'_s = C\{\min(i,j-s) = 0\},
\]
which are subquotient complexes of the $\zee\oplus\zee$-filtered knot complex $CFK^\infty(K)$. Between these, in a sense, is the vertical complex $B = C\{i = 0\}$, and we have natural chain maps
\[
v_s : A_s \to B \quad\mbox{and}\quad v'_s : B\to A'_s
\]
given by a quotient followed by inclusion in each case. We have already seen the first of these; the other map $v_s'$ is given by the quotient $C\{i = 0\}\to C\{i \geq s\}$ followed by the inclusion of the latter as a subcomplex of $A'_s$. Both of these maps have interpretations as maps induced by surgery cobordisms: $v_s$ is induced by a cobordism $S^3_N(K)\to S^3$ for $N\gg 0$, while $v'_s$ arises from a cobordism $S^3\to S^3_{N}(K)$ for $N\ll 0$.  The invariant $\epsilon(K)$ can be defined by considering the maps induced in homology by $v_s$ and $v_s'$ for the value $s = \tau(K)$:
\begin{itemize}
\item $\epsilon(K) = 1$ if and only if $v_\tau$ is surjective and $v_\tau'$ is trivial in homology.
\item $\epsilon(K) = 0$ if and only if both $v_\tau$ and $v'_\tau$ are surjective in homology.
\item $\epsilon(K) = -1$ if and only if $v_\tau$ is trivial and $v'_\tau$ is surjective in homology.
\end{itemize}
It is shown by Hom in \cite{Hom1}, where the invariant $\epsilon$ is introduced, that these are the only possibilities for the behavior of $v_\tau$ and $v'_\tau$. Moreover, it is always true that $v_s$ is surjective in homology for $s\geq \tau + 1$ and trivial for $s< \tau(K)$, and $v_s'$ is surjective for $s'\leq\tau-1$ and trivial for $s >\tau(K)$.

We have a commutative diagram
\[\begin{diagram}
&H_*(A_\tau) & \rTo^{h_\tau} & H_*(j = \tau)&\cong\F \\
&\dTo^{v_\tau} &\hspace*{2cm}& \dTo_{h'_\tau}\\
\F \cong& H_*(i = 0) & \rTo_{v'_\tau} & H_*(A'_\tau),
\end{diagram}
\]
where $h'_\tau$ is the map in homology induced by the quotient $C\{j = \tau\}\to C\{j = \tau, i\geq 0\}$ followed by inclusion as a subcomplex in $A'_\tau$. 

Now assume $\epsilon(K) = 0$, so in particular $v_\tau$ is surjective in homology. By Proposition 3.6 (2) of \cite{Hom1} the vanishing of $\epsilon$ also implies $\tau(K) = 0$, and hence by symmetry $h_\tau$ is surjective as well. If there is a class $a \in H_*(A_\tau)$ such that $v_\tau(a) \neq 0$ while $h_\tau(a) = 0$, the diagram above shows that $v'_\tau$ is necessarily trivial, contrary to the assumption $\epsilon(K) = 0$. Hence $\ker (h_\tau)\subset \ker(v_\tau)$, and since both kernels are codimension 1 in $H_*(A_\tau)$ they are identical. Therefore the maps in homology induced by $v_\tau$ and $h_\tau$ are surjective maps to $\F$ having the same kernel, which proves the forward implication of the lemma.

For the reverse implication, observe that if $\epsilon(K)  = -1$ then by definition $v_\tau$ is trivial in homology. Hence it suffices to assume $\epsilon(K) = 1$, so that $v_\tau$ is surjective in homology, and prove that there is a class $a\in H_*(A_\tau)$ with $v_\tau(a) \neq 0$ but $h_\tau(a) = 0$. For this, we recall from \cite{homsummand} another characterization of $\epsilon$ in terms of ``simplified bases'' for the knot Floer complexes.

The knot Floer complex $(CFK^\infty(K),\partial^\infty)$ is $\zee\oplus \zee$-filtered, which means that the chain groups are bigraded and the differential is nonincreasing in both gradings. In particular the boundary map $\partial^\infty$ decomposes as a sum of homogeneous maps, and we denote by $\dvert$ and $\dhorz$ the sum of components that preserve the first grading or the second, respectively. A filtered basis $\{x_i\}$ for $CFK^\infty$ is {\it vertically simplified} if for each $i$ exactly one of the following is true:
\begin{itemize}
\item When expressed as a combination of basis elements, $\dvert x_i$ has exactly one nonzero term.
\item $x_i$ appears as a nonzero term of $\dvert x_j$ for exactly one $j$, and $\dvert x_j = x_i$.
\item $\dvert x_i = 0$, and $x_i$ does not appear in the basis expression of $\dvert x_j$ for any $j$.
\end{itemize}
There is a similar definition for a {\it horizontally simplified} basis. Such bases (vertically or horizontally simplified) give rise to bases for $C\{i = 0\}$ and $C\{j = 0\}$ respectively, with the property that there is a unique basis element satisfying the third condition above: this is because both complexes have homology $\F$. Such an element is called the ``distinguished element'' of the basis. Hom shows that there is always a horizontally simplified basis for $CFK^\infty$ with a particular element $x_0$, which is  the distinguished element of some {\it vertically} simplified basis (in general, the latter basis must be different from the former). Then the assumption $\epsilon(K) = 1$ is equivalent to the condition that this $x_0$ is equal to $\dhorz x_j$ for some $x_j$ in the horizontally simplified basis. 

Let $\{x_j\}$ be such a horizontally simplified basis for $CFK^\infty$ and consider the associated basis for $A_\tau$. We assume $\epsilon(K) = 1$, so that $v_\tau: H_*(A_\tau)\to H_*(i=0)$ is surjective. By definition of $\tau$, a generator of $H_*(i=0)$ lies in the subcomplex $C\{i=0, j\leq \tau\}$, but in a vertically simplified basis such a generator is the distinguished element $x_0$. Hence we can consider $x_0$ as an element of $A_\tau$ as well, and we note $\partial_{A_\tau} x_0 = \dvert x_0 = \dhorz x_0 = 0$, the last equality following from the fact that $x_0 = \dhorz x_j$ when $\epsilon(K) = 1$. Thus $x_0$ determines a cycle in $A_\tau$ such that ${v_\tau}_*[x_0]$ is a generator of $H_*(B)$. On the other hand ${h_\tau}_*[x_0]$ is nothing but the homology class of $x_0$ thought of in the horizontal complex $C\{j = 0\}$, which clearly vanishes. The element $x_0\in A_\tau$ represents the desired homology class $a$.
\end{proof}

We now return to the proof of part (2) of Theorem \ref{cinvariantthm}, and so assume $\tb(\K) - \rot(\K) = 2\tau(K) - 1$. This means that we are interested in the image in homology of the inclusion of $(k, B)$ in $\XX_{-p/q}(-K)$, where $k = -\frac{1}{2}(\tb(\K) -\rot(\K) + 1)q + q - 1 = -q\tau(K) +q -1$. As we saw before, since $k/q = -\tau(K) + 1-1/q$, we have $\lfloor\frac{k}{q}\rfloor = -\tau(K) = \tau(-K)$ while $\lfloor\frac{k+1}{q}\rfloor = \tau(-K) + 1$.

First consider the case $\epsilon(K) = -1$ as in 2(c) of Theorem \ref{cinvariantthm}; we must see that $c(\xi_{x/y}^-) = 0$.
The mapping cone near $(k,B)$ looks like
\[
\begin{diagram}
 && A_{\btau}\\ 
 & \ldTo_{h_{\btau}} & \dTo_{v_{\btau}}\\
\cdots & & B
\end{diagram}
\]
We have that $v_{\btau}$ is onto homology, since $\epsilon(-K) = -\epsilon(K) = 1$. By Lemma \ref{epsilonlemma}, there is in fact a class $a\in H_*(A_{\btau})$ with ${v_{\btau}}_*(a) \neq 0$ and ${h_{\btau}}_*(a) = 0$, proving that $c(\xi_{x/y}^-)$ is a boundary in $\XX_{-p/q}(-K)$.

Finally we turn to the case $\epsilon(K) = 0$. Here $\tau(K) = \nu(K) = \nu(-K) = 0$, and part 2(b) of Theorem \ref{cinvariantthm} is equivalent to the assertion that the contact invariant $c(\xi_{x/y}^-)$ is nontrivial if and only if the smooth surgery coefficient $\frac{p}{q} = \frac{x}{y} + \tb(\K)$ is nonnegative. For $\frac{p}{q} > 0$, the mapping cone reads:
\[
\begin{diagram}
 \cdots && A_{\btau}&& A_{s} \\
 \dTo& \ldTo_{h_{\btau}} & \dTo_{v_{\btau}} & \ldTo_{h_{s}} \\
(k-p,B) & & (k, B)
\end{diagram}
\]
where $s = \lfloor\frac{k + p}{q}\rfloor$. Since necessarily $s > -\nu(-K)$, the map $h_{s}$ is trivial in homology. On the other hand, $v_\btau$ is surjective in homology; in fact by Lemma \ref{epsilonlemma} $v_\btau$ and $h_\btau$ give the same surjection in homology. Hence any cycle in $A_\btau$ that is mapped to the generator $c(\xi_r^-)$ of the homology of $(k, B)$ is also mapped onto the generator of homology of $(k-p,B)$, and we conclude only that these two generators are homologous in the mapping cone. If $p>q$, then the vertical map to $(k-p,B)$ has domain $A_{s'}$ with $s' < \btau$, and hence is trivial in homology. This shows that $c(\xi_r^-)$ determines a nonzero class in the homology of $\XX_{-p/q}(-K)$. For general $p>0$, a similar argument holds with a longer ``sawtooth'' picture demonstrating the nontriviality of $c(\xi_r^-)$. 

If $p<0$, the diagram above becomes
\[
\begin{diagram}
A_{\lfloor\frac{k-j|p|}{q}\rfloor} & & \phantom{A_{\lfloor\frac{a}{b}\rfloor}}& & A_{\lfloor\frac{k}{q}\rfloor} && A_{\lfloor\frac{k+|p|}{q}\rfloor}& \\
\dTo & \rdTo^{h_j} & \cdots&\rdTo(2,2) & \dTo^{v_{\btau}}&\rdTo^{h} & \dTo^{v} & \rdDashto^{h = 0}\\
(k -j|p|, B) &&&& (k, B) && (k+|p|, B)&&\cdots
\end{diagram}
\]
where we choose $j$ to be the smallest integer such that $\lfloor\frac{k-j|p|}{q}\rfloor = \lfloor\frac{k}{q}\rfloor-1$. In particular, all intermediate complexes $A_{\lfloor \frac{k - j'|p|}{q}\rfloor}$ are copies of $A_{\btau}$. Since $\lfloor\frac{k}{q}\rfloor =\btau = \nu(-K) = 0$ the vertical map to $(k-j|p|, B)$ is trivial in homology, while the map labeled $h_j$ is a homology surjection. The remaining solid arrows in the diagram are surjections in homology having the same kernel at $A_\btau$, given our assumptions and Lemma \ref{epsilonlemma}. It follows easily that the homology generator of $(k,B)$ is trivial in the homology of $\XX_{-p/q}(-K)$.

Finally for $\frac{p}{q} = 0$, the only maps in $\XX_{-p/q}(-K)$ interacting with $(k,B)$ are $v_\btau, h_\btau: A_\btau \to B$. These are nontrivial, but give the same map in homology by Lemma \ref{epsilonlemma}---hence their sum vanishes in homology with coefficients in $\F$, and the class determined by the homology generator of $(k,B)$ is nontrivial in the homology of $\XX_{-p/q}(-K)$.

This completes the proof of Theorem \ref{cinvariantthm}.

\section{Rational Surgery Mapping Cone}\label{section5}

In this section we describe the proof of the second part of Theorem \ref{ratsurgformula}, and along the way prove some facts that were used at isolated points in the preceding. This section is nearly independent of the rest of the paper, and follows lines that will be familiar to experts; we assume a reasonable familiarity with Heegaard Floer theory as in \cite{OSknot,OSintsurg, OSratsurg}. 

\subsection{Rationally nullhomologous knots}\label{cobordismssec}
Let $M$ be a compact oriented 3-manifold with connected boundary diffeomorphic to a torus. It is a standard exercise that the inclusion induces a homomorphism $H_1(\partial M;\zee)\to H_1(M;\zee)$ with kernel isomorphic to $\zee$, and hence up to orientation there is a unique isotopy class of simple closed curve $\nu\subset\partial M$ and integer $c>0$ such that $c[\nu]$ generates this kernel. We can find a curve $\eta\subset \partial M$ dual to $\nu$, and we orient $\nu$ and $\eta$ such that, with the natural orientation on $\partial M$, the intersection number $\nu.\eta$ is $+1$. It follows that $\eta$ represents a class in $H_1(M;\zee)$ of infinite order. 

Let $Y$ be the closed 3-manifold obtained by Dehn filling of $M$ along a curve $\mu$ homologous to $s\nu + t\eta$, for relatively prime $s$, $t$ with $t>0$ and $s\neq 0$. Observe that $H_1(Y) = H_1(M)/[\mu]$. The core of the filling torus gives rise to a knot $K\subset Y$, whose homology class we can describe as follows. Choose integers $s',t'$ with $s't-t's = 1$; then a longitude of $K$ is given by the curve $\lambda=s'\nu + t'\eta$. It is now easy to see that $K$ is rationally nullhomologous in $Y$, of order $q:= ct$.

Turning this around, if $K\subset Y$ is a knot that is rationally nullhomologous of order $q$, then there is a well-defined isotopy class of curve $\nu$ on the boundary of $\textrm{nbd}(K)$ and an integer $c>0$ dividing $q$ such that $c\nu$ represents a generator of the kernel of the map on homology induced by the inclusion $\partial (\textrm{nbd}(K))\to M = Y-\textrm{nbd}(K)$. Likewise, we have a canonically-determined meridian $\mu$ of $K$ with $\nu.\mu = t:= q/c$ on $\partial M$. 

We have some choice in selecting a longitude $\lambda$ for $K$, in that any curve of the form $\lambda + k\mu$ is also a longitude. On the other hand, for given $\lambda$ we can write $\nu = t\lambda + r\mu$ for some uniquely determined $r$. Hence, a canonical longitude $\lambda_{can}$ for $K$ is specified by requiring
\[
\nu = t\lambda_{can} + r\mu \quad \mbox{where $0\leq r<t$.}
\]
Observe that while there is some ambiguity in the choice of $\eta$ in the discussion above, if the pair $(Y,K)$ is given (with $K$ rationally nullhomologous of order $q$), then the meridian $\mu$ and the  curve $\nu$ on $\partial (\textrm{nbd}(K))$ are canonically determined, and the canonical longitude $\lambda_{can}$ is then uniquely specified by the equation above.

Using $\mu,\lambda_{can}$ as coordinates for $\partial(\textrm{nbd}(K))$ we can consider an integer (``Morse'') surgery on $K$ with surgery coefficient $m$. Equivalently, the surgery $Y_m(K)$ is given by Dehn filling of $M$ along the curve $\lambda_{can} + m\mu$.

\begin{lemma}\label{selfintlemma} The natural 2-handle cobordism $W_{\lambda_{can}}: Y\to Y_{\lambda_{can}}(K)$ has $H_2(W_{\lambda_{can}};\zee)/H_2(Y;\zee) \cong \zee$, generated by the homology class of a surface $S_{\lambda_{can}}$ having self-intersection $-qcr$. 

More generally, the second homology of the surgery cobordism corresponding to $\lambda_m = \lambda_{can} + m\mu$ is generated by the class of a surface $S_{\lambda_m}$ with self-intersection $q(mq-cr)$.
\end{lemma}

This can be proven by examining the exact sequences for the pair $(W_{\lambda_m}, Y)$ and the triple $(W_{\lambda_m}, \partial W_{\lambda_m}, Y)$, or by a direct geometric construction of the surface $S_{\lambda_m}$. In either case the key observation is that in $Y_{\lambda_m}$ the induced knot (the core of the surgery) has order $|c(mt-r)|$ in homology.

\begin{remark} The integer $c$ is called the ``multiplicity'' of $K$ by Baker-Etnyre \cite{BErational} and is equal to the number of boundary components of a rational Seifert surface for $K$. It is not hard to check that the quantity $cr$ appearing in the Lemma is equal modulo $q$ to the intersection number between a pushoff of $K$ and its rational Seifert surface, i.e., it is essentially the numerator of the $\cue/\zee$ valued self-linking of $[K]$. More precisely, if $\ell$ is the representative in $[0,1)$ of $\lk_{\cue/\zee}([K],[K])$, then $cr = q\ell$.
\end{remark}

Now fix a framing $\lambda = \lambda_m$ for the rationally null-homologous knot $K\subset Y$. It is easy to see that the surgery cobordism $W_\lambda$ has $H_2(W_\lambda, Y;\zee)\cong \zee$, generated by a relative cycle $[F_\lambda]$ (represented by the core of the 2-handle), and moreover we can arrange that under the natural map $\iota: H_2(W_\lambda)\to H_2(W_\lambda, Y)$ we have
\begin{equation}\label{conv1}
\iota([S_\lambda]) = q[F_\lambda].
\end{equation}
Hence for a class $\alpha\in H^2(W_\lambda)$, we can define an evaluation of $\alpha$ on $[F_\lambda]$ by
\[
\langle\alpha, [F_\lambda]\rangle = {1\over q}\langle\alpha, [S_\lambda]\rangle \in \cue.
\]
Similarly, we have a rational number 
\[
[F_\lambda].[F_\lambda] = {1\over q^2}[S_\lambda].[S_\lambda].
\]
In particular if $\lambda = \lambda_m$ as above, then $[F_\lambda].[F_\lambda] = (mq-cr)/q$. 
\begin{remark} Strictly, we should work with an oriented knot $K$. Then longitudes and meridians are taken to be oriented by standard conventions; likewise the surface $S_\lambda$ inherits an orientation from a rational Seifert surface for $K$. While we are mostly interested in the case that $Y$ is a rational homology sphere, for the general case we suppose a choice of rational Seifert surface has been fixed once and for all. We stick with these conventions henceforth, but without further mention.
\end{remark}
Now let $(\Sigma, \aalpha,\bbeta,w,z)$ be a Heegaard diagram adapted to the knot $K\subset Y$. Recall that this means $\beta_g$ corresponds to a meridian of $K$, and the basepoints $w$ and $z$ lie to either side of $\beta_g$. 
Choose any framing $\lambda$ for $K$, and consider the corresponding set of attaching curves $\ggamma$, all obtained by small Hamiltonian translation of the $\beta$ curves, except that $\gamma_g$ represents the framing $\lambda$ (this conforms to the usual conventions of the theory, but reverses the role of $\gamma$ and $\beta$ curves as compared to Section \ref{robssection}). For any Heegaard Floer generator $\x\in\Ta\cap \Tb$, we define a rational number
\[
f(\x) = \langle c_1(\s_w(\psi)), [F_\lambda]\rangle + [F_\lambda].[F_\lambda] - 2(n_w(\psi) - n_z(\psi)).
\]
Here $\psi\in\pi_2(\x,\Theta_{\beta\gamma},\x')$ is any triangle connecting the generator $\x$ to some intersection point $\x'\in\Ta\cap\Tg$.

\begin{lemma} \label{flemma}The function $f(\x)$ is independent of the choice of $\psi$, $\x'$, and $\lambda$, i.e., it depends only on $\x$.
\end{lemma}

This was essentially proven by Ozsv\'ath and Szab\'o \cite[proof of Lemma 4.6]{OSratsurg}. We reprise and expand their argument.

\begin{proof} Note that by introducing a trivial winding of $\gamma_g$ around $\beta_g$ we can always arrange that a given intersection point $\x\in\Ta\cap\Tb$ is connected by a (small) triangle to some point $\x'\in\Ta\cap\Tg$.

Now fix $\lambda = \lambda_m$. Since $H_2(W_\lambda)/H_2(Y)\cong \zee$ is generated by $[S_\lambda]$, there is a triply-periodic domain $\P_S$ in $(\Sigma, \aalpha,\bbeta,\ggamma,w)$ representing this generator. We can determine the coefficients of $\P_S$ in regions of $\Sigma$ near $\beta_g$ as follows. Observe that if $\beta_g$ is replaced by the curve $\nu$, we obtain a Heegaard diagram  describing the result of Dehn filling $Y-\textrm{nbd}(K)$ along the torsion curve $\nu$. In particular, the first Betti number of this filling is one more than the Betti number of $Y$. Hence, there is a (doubly) periodic domain in the diagram $(\Sigma, \aalpha,\mbox{\boldmath $\nu$}, w)$ containing $\nu$ with multiplicity $c$ in its boundary, corresponding to a rational Seifert surface for $K$. On the other hand, since $\nu = t\lambda_{can} + r\mu$ in homology, by replacing $\nu$ by a concatenation of copies of $\lambda = \gamma_g$ and $\mu= \beta_g$, we can construct a triply-periodic domain in $(\Sigma, \aalpha, \bbeta,\ggamma, w)$ representing $S_\lambda$ and containing $\gamma_g$ with multiplicity $q$ and $\beta_g$ with multiplicity $mq-cr$ in its boundary. See Figure \ref{figure}. 
\begin{figure}
\hspace*{-1cm}\includegraphics{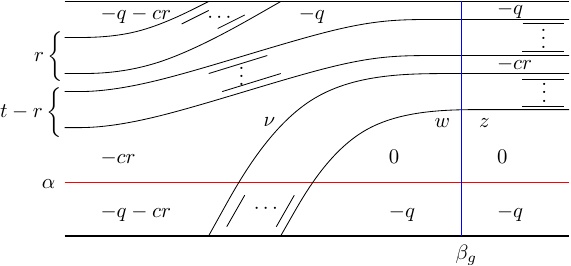}

\vspace*{1cm}
\hspace*{-.5cm}\includegraphics{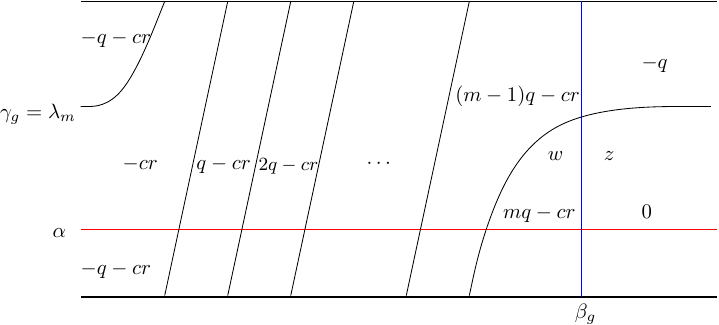}
\caption{\label{figure}Heegaard diagrams near the meridian curve $\beta_g$, where the top and bottom of each diagram are to be identified. Top, the torsion curve $\nu$ is shown together with coefficients of a periodic domain containing $\nu$ $c$ times in its boundary. Away from the pictured region, $\nu$ is taken to be $t$ copies of the longitude $\lambda_{can}$ (not shown). Bottom, the longitude $\lambda_m = \lambda_{can} + m\mu$, with coefficients of the corresponding triply-periodic domain. Away from the picture, the coefficients agree with those in the upper diagram, after collapsing the parallel copies of $\lambda_{can}$ that comprise $\nu$. Note that there may be additional $\alpha$ curves appearing, parallel to the one shown.}
\end{figure}

To see that $f(\x)$ is independent of the choice of triangle $\psi\in \pi_2(\x,\Theta_{\beta\gamma},\x')$ (for $\x'$ fixed), observe that two such triangles differ by a triply-periodic domain (up to a multiple of the Heegaard surface, which clearly does not affect $f$). Thus it suffices to assume $\psi' = \psi + \P_S$. The values of $f$ corresponding to $\psi$ and $\psi'$ then differ by the quantity
\[
\langle 2PD[\P_S],[F_\lambda]\rangle - 2(n_w(\P_S)-n_z(\P_S)).
\]
Using Lemma \ref{selfintlemma} and considering Figure \ref{figure}, we have that both terms above are equal to $2(mq-cr)$.

Now consider the effect on $f(\x)$ of replacing $\x'$ by another intersection point $\x''$ such that $\s_w(\x'') = \s_w(\x')$. Then there is a Whitney disk $\phi\in\pi_2(\x',\x'')$, and given $\psi'\in \pi_2(\x,\Theta_{\beta\gamma},\x')$, we can construct a triangle $\psi''\in\pi_2(\x,\Theta_{\beta\gamma},\x'')$ as $\psi'' = \psi' + \phi$. But this adjustment does not affect $\s_w(\psi')$ and hence the first term in $f(\x)$ is preserved. The second term is not affected by choice of $\psi$, while the third is invariant since $\psi'$ and $\psi''$ have the same boundary in $\Tb$. Therefore, given $\lambda$ the function $f(\x)$ depends at most on the \spinc structure $\s_w(\x')$.

It is now clear that we can adjust $\gamma_g = \lambda$ by an isotopy without affecting $f$. Thus, we introduce sufficient trivial winding of $\gamma_g$ around $\beta_g$ (c.f. \cite[Figure 2]{OS1} or Figure \ref{woundfigure} below) such that the following holds: any \spinc structure represented by an intersection point in $\Ta\cap\Tg$ that is connected by a triangle in $(\Sigma,\aalpha,\bbeta,\ggamma,w)$ to a generator in $\Ta\cap\Tb$, is also represented by an intersection point that is supported in the winding region. (This is possible since any two \spinc structures in $Y_\lambda$ cobordant to a given $\s\in\Spinc(Y)$ differ by a multiple of the Poincar\'e dual of the meridian of $K$.)

Hence, to examine the dependence of $f(\x)$ on the \spinc structure $\s_w(\x')$, it suffices to consider two generators $\x',\x''\in \Ta\cap\Tg$, both supported in the winding region and differing only in their component on $\gamma_g$. Moreover we can suppose these components are as pictured in Figure \ref{woundfigure}. 
Finally, we can assume that $\x'$ and $\x''$ are connected to $\x$ by a ``small triangle,'' i.e., for $i\neq g$, the component of $\x'$ (and $\x''$) on $\gamma_i$ is the canonical ``closest point'' to the corresponding component of $\x$ under the Hamiltonian isotopy between $\beta_i$ and $\gamma_i$.

\begin{figure}
\includegraphics{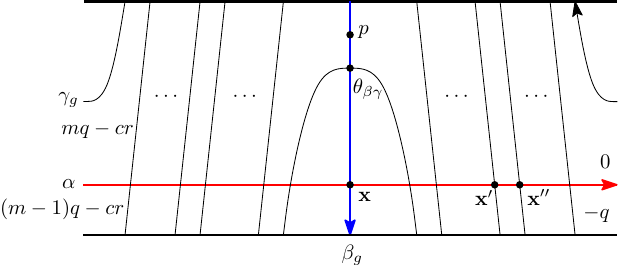}
\caption{\label{woundfigure}The region near $\beta_g$ in Figure \ref{figure} (bottom), after winding. Chosen orientations on $\alpha$, $\beta$, and $\gamma$ curves are indicated, with some coefficients of $\P_S$.}
\end{figure}

Letting $\psi'$ and $\psi''$ be the small triangles corresponding to $\x'$ and $\x''$ respectively, we consider the difference $f(\x,\psi'') - f(\x,\psi')$. Clearly $[F_\lambda].[F_\lambda]$ is unchanged, while 
\[
(n_w(\psi'')-n_z(\psi'')) - (n_w(\psi') - n_z(\psi')) = -1.
\]
Recall that there is a method to calculate the evaluation of $c_1(\s_w(\psi))$ on $[S_\lambda]$ from the Heegaard diagram, summarized by the formula
\begin{equation}\label{chern}
\langle c_1(\s_w(\psi)), [S_\lambda]\rangle = \hat{\chi}(\P_S) - 2n_w(\P_S) + \#\partial\P_S + 2\sigma(\psi,\P_S)
\end{equation}
(see \cite[Proposition 6.3]{OSabsgrad}). We review the definitions of the terms in this formula as we go along, but for the moment observe that the only term on the right hand side that depends on $\psi$ is the ``dual spider number'' $\sigma(\psi,\P_S)$. This quantity is obtained by considering small left-hand pushoffs $\alpha'$, $\beta'$, and $\gamma'$ of each $\alpha$, $\beta$, and $\gamma$ curve, according to a chosen orientation on these curves (the dual spider number is independent of this choice). Take arcs $a$, $b$ and $c$ in the 2-simplex $\Delta$ that is the domain of $\psi: \Delta\to \Sym^g(\Sigma)$, where $a$, $b$, $c$ connect a basepoint $u\in int(\Delta)$ to the $\alpha$, $\beta$, and $\gamma$ boundary segments of $\Delta$, respectively. Identifying these arcs with their image 1-chains in $\Sigma$, we have
\[
\sigma(\psi,\P_S) = n_{\psi(u)}(\P_S) + (\partial_{\alpha'}\P_S).a + (\partial_{\beta'}\P_S).b + (\partial_{\gamma'}\P_S).c
\]
With conventions indicated in Figure \ref{woundfigure}, and taking $u$ to be near the $\beta\gamma$ corner of $\Delta$, we find that in the difference  $\sigma(\psi'',\P_S) - \sigma(\psi',\P_S)$ only the terms involving $(\partial_{\alpha'}\P_S).a$ remain. Each intersection of the arc $a$ with the $\alpha$ curve in the diagram contributes $-q$ to this quantity, and we get one more such contribution from $\psi''$ than from $\psi'$. Hence
\[
\langle c_1(\s_w(\psi'')), [S_\lambda]\rangle - \langle c_1(\s_w(\psi')), [S_\lambda]\rangle = -2q.
\]
Therefore 
\[
\langle c_1(\s_w(\psi'')), [F_\lambda]\rangle - \langle c_1(\s_w(\psi')), [F_\lambda]\rangle = -2,
\]
which cancels the difference in the term $-2(n_w(\psi) - n_z(\psi))$ in $f(\x)$. Thus, $f$ is independent of the choice of $\x'$.

Finally, we must see $f$ is independent of the framing $\lambda$. In light of the preceding it suffices for this to consider the diagrams of Figure \ref{framingchange}, 
in which a given framing $\lambda$ is replaced by $\lambda - \mu$. We can consider the smallest triangles in each diagram and assume the points of $\x'$ and $\x''$ agree away from the portion of the diagram indicated in the figure. From Lemma \ref{selfintlemma}, the term $[F_\lambda].[F_\lambda]$ decreases by 1 from the left to the right side of the figure. Clearly the small triangles have $n_w - n_z = 0$, so we turn to the Chern class term.

\begin{figure}
\includegraphics[width=4.5in]{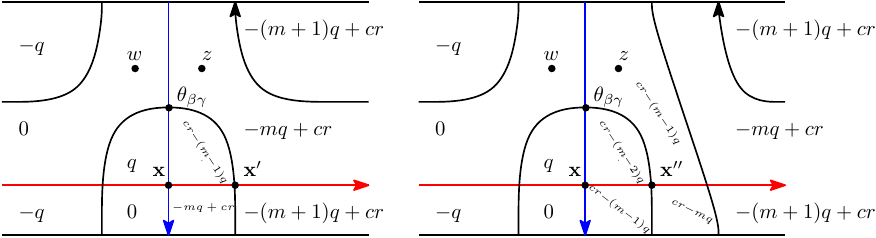}
\caption{\label{framingchange}Invariance of $f(\x)$ under change of framing. Shown is the region of Figure \ref{figure} (bottom) near $\beta_g$ after one winding, before and after a change of framing. The coefficients of the triply periodic domain are shown, after subtracting $mq-cr$ copies of the Heegaard surface for convenience. }
\end{figure}

Referring to \eqref{chern}, the Euler measure $\hat{\chi}(\P_S)$ is defined by 
\[
\hat\chi(\P_S) = \sum_i n_i(\chi(D_i) - \ts\frac{1}{4}\#\mbox{(corners in $D_i$)})
\]
if the domain $\P_S$ is expressed as a linear combination $\sum_i n_i D_i$ of domains (closures of components) of $\Sigma - (\aalpha,\bbeta,\ggamma)$. It is easy to see, using the fact that the domains $\P_{S_\lambda}$ and $\P_{S_{\lambda - \mu}}$ on the two sides of Figure \ref{framingchange} agree in the portion of $\Sigma$ not pictured, that this term is unchanged on replacing $\lambda$ by $\lambda -\mu$. Likewise, we've arranged that $n_w(\P_{S_\lambda}) = n_w(\P_{S_{\lambda - \mu}}) = 0.$

The term $\#\partial \P_S$ in \eqref{chern} denotes the coefficient sum of all terms in $\partial\P_S$, expressed as an integer linear combination of $\alpha$, $\beta$, and $\gamma$ curves. Easily,
\[
\#\partial\P_{S_\lambda} - \#\partial\P_{S_{\lambda-\mu}} = -q,
\]
coming just from the contribution of the term $\beta_g$. For the dual spider number, we find
\[
\sigma(\psi',\P_{S_\lambda}) - \sigma(\psi'', \P_{S_{\lambda - \mu}}) = -q \beta' . b + n_{\psi(x)}(\P_{S_\lambda} - \P_{S_{\lambda - \mu}}) = 0.
\]
Adding these contributions to the Chern class term and dividing by $q$ as before, it follows that $f(\x, \lambda) = f(\x,\lambda -\mu)$.

\end{proof}

\subsection{Surgery exact triangle}
The proof of Theorem \ref{ratsurgformula} is based on a long exact sequence relating the Floer homology groups of manifolds obtained by different surgeries along a knot $K$, and in particular the deduction of part 2 of the theorem is based on an analysis of the maps in this sequence. We recall the construction of this surgery exact sequence, following \cite{OSintsurg} and \cite{OSratsurg}.

As before, consider a rationally null-homologous knot $K\subset Y$, and an adapted Heegaard diagram $(\Sigma, \aalpha, \bbeta, w)$ where $\beta_g = \mu$ is a meridian of $K$.  Fixing a framing $\lambda$ for $K$ and an integer $m>0$, we can produce Heegaard diagrams
\begin{itemize}
\item $(\Sigma, \aalpha, \ggamma, w)$ for the surgery $Y_\lambda(K)$ where $\gamma_j$ is a small translate of $\beta_j$ except that $\gamma_g = \lambda$.
\item $(\Sigma, \aalpha,\ddelta, w)$ for $Y_{\lambda + m\mu}(K)$, where $\delta_j$ is a translate of $\beta_j$ except that $\delta_g = \lambda + m\mu$. 
\end{itemize}

To condense notation, fix a basepoint $p\in \beta_g$ such that $w$ and $z$ are joined by an arc only intersecting $\beta_g$ transversely at $p$, and not intersecting any other curve in the diagram. For a domain $\P$ in a Heegaard diagram as above we write $m_p(\partial \P)$ for the multiplicity with which the portion of $\partial\P$ on $\beta_g$ crosses $p$: equivalently, we set $m_p(\partial\P) = n_w(\P) - n_z(\P)$. Next define a twisted Floer chain complex for $Y$ generated by intersections between the tori $\Ta$ and $\Tb$ in $\Sym^g(\Sigma)$, with coefficients in the ring $\F[C_m]$ where $C_m$ is the cyclic group of order $m$. We think of this as $\F[C_m] = \F[T, T^{-1}]/(1- T^m)$. The boundary in this chain complex is
\[
\partial(U^{-i}\x) = \sum_{\phi\in\pi_2(\x,\y)}\#\M(\phi) T^{m_p(\partial \phi)} U^{-i+n_w(\phi)} \y.
\]
Since $K$ is rationally null-homologous the homology is actually untwisted: there is an $\F[C_m]$ chain isomorphism
\[
\theta: CF^+(Y, \F[C_m]) \to CF^+(Y) \otimes \F[C_m],
\]
where the differential on the codomain is the tensor product of the differential on $CF^+(Y)$ with the identity on $\F[C_m]$, given by
\[
\theta(\x ) = \x \otimes T^{m_p(\phi_{\x \x_0})}.
\]
Here $\x_0$ denotes a chosen intersection point (in each \spinc equivalence class) and $\phi_{\x\x_0}$ is a fixed choice, for each $\x\in\Ta\cap\Tb$, of a disk connecting $\x$ to $\x_0$ (c.f. \cite[proof of Theorem 3.1]{OSintsurg}). Note that we are free also to multiply $\theta$ by a fixed power of $T$ if we choose.

Now, in this situation there is an exact triangle
\begin{diagram}
HF^+(Y_\lambda) && \rTo^{f_1} && HF^+(Y_{\lambda + m\mu}) \\ 
&\luTo^{f_3} && \ldTo^{f_2} \\
&& \ul{HF}^+(Y, \F[C_m]) &&
\end{diagram}
The maps in this triangle are induced by chain maps defined as follows: first note that $(\Sigma, \ggamma, \ddelta)$ describes the connected sum of $L(m,1)$ with $\#^{g-1} S^1\times S^2$. We fix a ``canonical'' intersection point $\Theta_{\gamma\delta}\in \Tg \cap\Td$ (adjacent to the basepoint). Then
\begin{equation}\label{f1def}
f_1(\x) = \sum_{\psi\in\pi_2(\x,\Theta_{\gamma\delta},\y)} \#\M(\psi) U^{n_w(\psi)}\y.
\end{equation}
For the maps involving the twisted Floer group, we have
\begin{equation}\label{f2def}
f_2(\y) = \sum_{\psi\in \pi_2(\y,\Theta_{\delta\beta}, \V)}\#\M(\psi) T^{m_p(\partial\psi)}U^{n_w(\psi)}\V,
\end{equation}
and
\begin{equation}\label{f3def}
f_3(T^s\V) = \sum_{\stackrel{\psi\in\pi_2(\V,\Theta_{\beta\gamma},\x)}{ m_p(\partial\psi) = -s \mod m}} \#\M(\psi)U^{n_w(\psi)}\x
\end{equation}
where $\Theta_{\delta\beta}$ and $\Theta_{\beta\gamma}$ are the usual canonical intersection points.

Ozsv\'ath and Szab\'o's proof of the exactness of the triangle \cite[Theorem 3.1]{OSintsurg} implies that $CF^+(Y_\lambda)$ is quasi-isomorphic to the mapping cone of $f_2$. In particular the map 
\[
CF^+(Y_{\lambda+m\mu})\oplus \ul{CF}^+(Y, \F[C_m])\to CF^+(Y_\lambda)
\]
given by $(a,b) \mapsto f_3(b) + h_2(a)$ defines a quasi-isomorphism $Cone(f_2)\to CF^+(Y_\lambda)$, where $h_2$ is a null-homotopy of $f_3\circ f_2$. It follows that the composition of this quasi-isomorphism with the map in homology induced by the natural inclusion $\ul{CF}^+(Y, \F[C_m]) \to Cone(f_2)$ is just the map $f_3$ in the surgery triangle.

On the other hand, the map $f_3$ (and all the maps in the triangle) can be identified with homomorphisms induced by cobordisms, as we now explain. Consider the natural 2-handle cobordism $W_\lambda: Y\to Y_\lambda$, equipped with a \spinc structure $\s$. The corresponding homomorphism on Floer complexes $CF^+(Y)\to CF^+(Y_\lambda)$ is defined by a count of holomorphic triangles in $(\Sigma, \aalpha,\bbeta,\ggamma,w)$ analogously to \eqref{f3def}, without reference to $T$ or $s$, and where the sum is over homotopy classes of triangle whose associated \spinc structure is exactly $\s$. 

\begin{lemma}\label{YtoYlambdalemma}
For a given integer $s$, let $\eta_s$ denote the restriction of $\theta^{-1}$ to the summand $CF^+(Y)\otimes T^s\subset CF^+(Y)\otimes\F[C_m]$. Then the composition 
\[
f_3\circ \eta_s: CF^+(Y) \cong CF^+(Y)\otimes T^s \subset CF^+(Y)\otimes \F[C_m]\to CF^+(Y_\lambda)
\]
is equal to the sum of the homomorphisms induced by those \spinc structures on $W_\lambda$ represented by homotopy classes of triangle $\psi$ having \[m_p(\partial\psi) - m_p(\partial \phi_{\x\x_0}) + s = 0 \mod m.\] In particular, the set of such triangles constitutes a union of \spinc equivalence classes.
\end{lemma}

\begin{proof} 
From the definition of $\theta$, it follows that $\eta_s : CF^+(Y)\to CF^+(Y, \F[C_m])$ is given by
\[
\eta_s(\x) = T^{s - m_p(\partial\phi_{\x\x_0})}\x.
\]
Hence the composition is given explicitly by
\[
f_3\circ \eta_s(\x) = \sum_{\stackrel{\psi\in\pi_2(\x,\Theta_{\beta\gamma},\y)}{m_p(\partial\psi) = m_p(\partial\phi_{\x\x_0}) - s\mod m}} \#\M(\psi)U^{n_w(\psi)}\x.
\]
Thus we need only verify the last statement.

For this recall that triangles $\psi\in\pi_2(\x,\Theta_{\beta\gamma}, \y)$ and $\psi'\in\pi_2(\x',\Theta_{\beta\gamma},\y')$ induce the same \spinc structure on $W_\lambda$ (they are ``\spinc equivalent'') if and only if $\psi' = \psi + \phi_{\alpha\beta} + \phi_{\alpha\gamma}$ for some disks $\phi_{\alpha\beta}\in\pi_2(\x',\x),\, \phi_{\alpha\gamma}\in\pi_2(\y,\y')$ in the indicated diagrams. (In principle a disk $\phi_{\beta\gamma}$ could also appear, but here $\psi$ and $\psi'$ are assumed to have corners on $\Theta_{\beta\gamma}$. Hence such a disk is a $(\beta,\gamma)$-periodic domain, up to multiples of the Heegaard surface. Both of these have no boundary on $\beta_g$ hence their boundary has $m_p = 0$.) Since the basepoint $p$ is on $\beta_g$, we have $m_p(\partial \phi_{\alpha\gamma}) = 0$. Additivity shows that if $m_p(\partial\psi') -m_p(\partial\phi_{\x'\x_0}) +s = 0$ mod $m$ then modulo $m$,
\begin{eqnarray*}
0 &=& m_p(\partial \psi) + m_p(\partial \phi_{\alpha\beta}) -m_p(\partial\phi_{\x'\x_0}) + s\\ &=& m_p(\partial\psi) - m_p(\partial(\phi_{\x'\x_0} - \phi_{\alpha\beta})) + s
\end{eqnarray*}
But $\phi_{\x'\x_0} - \phi_{\alpha\beta}$ is a disk connecting $\x$ to $\x_0$ and while it need not be the same as $\phi_{\x\x_0}$, we observe that since the knot is rationally null-homologous the value of $m_p$ vanishes for periodic domains in $(\Sigma,\aalpha,\bbeta,w)$. Hence we can replace $\phi_{\x'\x_0} - \phi_{\alpha\beta}$ by $\phi_{\x\x_0}$ in the above and the conclusion follows.

\end{proof}

In a similar vein, we have:

\begin{lemma}\label{YmtoYlemma} Let $\theta_s : \ul{CF}^+(Y, \F[C_m])\to CF^+(Y)$ denote the composition of $\theta$ with the projection to the coefficient of $T^s$. Then the composition
\[
\theta_s\circ f_2: CF^+(Y_{\lambda+m\mu})\to CF^+(Y)
\]
is equal to a sum of maps induced by the 2-handle cobordism $Y_{\lambda+m\mu}\to Y$ equipped with the \spinc structures represented by triangles $\psi$ in the diagram $(\Sigma, \aalpha,\ddelta,\bbeta,w)$ such that \[m_p(\partial\psi) + m_p(\partial\phi_{\x\x_0}) = s \mod m.\]
\end{lemma}

\begin{proof} Again, the point is that the assignment $\psi \mapsto m_p(\partial\psi) + m_p(\partial\phi_{\x\x_0}) \in \zee/m\zee$ descends to \spinc equivalence classes. Writing $\psi' = \psi + \phi_{\alpha\delta} + \phi_{\alpha\beta}$ as before, note that this time $\phi_{\alpha\beta} \in \pi_2(\x,\x')$ rather than $\pi_2(\x',\x)$. Hence 
\begin{eqnarray*}
m_p(\partial\psi') + m_p(\partial \phi_{\x'\x_0}) &=& m_p(\partial\psi) + m_p(\partial\phi_{\alpha\beta}) + m_p(\partial\phi_{\x'\x_0})\\
&=& m_p(\partial\psi) + m_p(\partial \phi_{\x\x_0})
\end{eqnarray*}
by analogous reasoning.
\end{proof}
Explicitly, the last lemma says that
\[
\theta\circ f_2 = \sum_{\s\in\Spinc(W_{\lambda+m\mu})} F_{W_{\lambda+m\mu},s} \cdot T^{\ul{m}(\s)}
\]
where $\ul{m}: \Spinc(W_{\lambda+m\mu})\to \zee/m\zee$ is induced by $\psi\mapsto m_p(\partial\psi) + m_p(\partial\phi_{\x\x_0})$ as above.


We now use the results from the previous subsection to relate the \spinc structures on the various surgery cobordisms relevant to the surgery triangle to one another. According to Lemma \ref{YtoYlambdalemma}, for given $s$ the composition $f_3\circ \eta_s: CF^+(Y)\to CF^+(Y_\lambda)$ is given by the sum of homomorphisms induced by \spinc structures on $W_\lambda$ represented by triangles $\psi$ such that $m_p(\partial\psi) - m_p(\partial \phi_{\x\x_0}) + s =0$ mod $m$. For such $\psi$ we have 
\begin{eqnarray}
\langle c_1(\s_w(\psi)), [F_\lambda]\rangle + [F_\lambda].[F_\lambda] &=& f(\x) + 2m_p(\partial\psi)\nonumber\\
&=& f(\x) + 2(m_p(\partial\phi_{\x\x_0})-s) \mod m.\label{labeleqn}
\end{eqnarray}
A similar statement is true for the triangles counted in the composition $\theta_s\circ f_2$, except that in this case we consider the cobordism $Y_{\lambda+m\mu}\to Y$ given by $-W_{\lambda+m\mu}$, and triangles are in the diagram $(\Sigma, \aalpha, \ddelta,\bbeta,w)$. If $\bar\psi_m$ is such a triangle, the triangle $\psi_m = - \bar\psi_m$ obtained by negating the coefficients of $\bar\psi_m$ is a triangle we can use to compute $f(\x)$, and in particular we have
\begin{eqnarray*}
f(\x)&=& \langle c_1(\s_w(\psi_m)), [ F_{\lambda_m}]\rangle + [ F_{\lambda_m}].[ F_{\lambda_m}] - 2m_p(\partial \psi_m)\\
&=& \langle c_1(\s_w(\psi_m)), [ F_{\lambda_m}]\rangle + [ F_{\lambda_m}].[ F_{\lambda_m}] + 2m_p(\partial \bar\psi_m)
\end{eqnarray*}
According to Lemma \ref{YmtoYlemma}, $\theta_s\circ f_2$ is given by the sum of maps counting triangles $\bar\psi_m$ satisfying $m_p(\partial\bar\psi_m) + m_p(\partial\phi_{\x\x_0}) = s$ mod $m$. Hence \eqref{labeleqn} becomes, modulo $m$,
\[
\langle c_1(\s_w(\psi)), [F_\lambda]\rangle + [F_\lambda].[F_\lambda] = \langle c_1(\s_w(\psi_m)), [ F_{\lambda_m}]\rangle + [ F_{\lambda_m}].[ F_{\lambda_m}].
\]
 One way to view this result is to note that for any framing on $K$, the \spinc structures on the corresponding surgery cobordism $W$, which extend a fixed \spinc structure on $Y$, can be labeled uniquely by the rational numbers $\langle c_1(\s),[F]\rangle + [F].[F]$ where $[F]$ is the generator of $H_2(W,Y)$ as usual. We have seen that the inclusion of $CF^+(Y)$ as the coefficient of $T^s$ in the mapping cone of $f_2$ corresponds to a map induced by \spinc structures on $W_\lambda$. The above says that the coefficient of $T^s$ is also the target of maps induced by \spinc structures on $W_{\lambda+m\mu}$ equipped with \spinc structures having the {\em same labels} (modulo $m$) as the ones on $W_\lambda$ (but thought of as on $-W_{\lambda+m\mu}$).

\subsection{Rational Surgeries}\label{ratsurgsec}

We now consider the situation of a knot $K_0\subset S^3$, and the formula deduced by Ozsv\'ath and Szab\'o for the Heegaard Floer homology of a 3-manifold obtained by rational surgery on $K_0$. We write $p/q$ for the surgery coefficient, where henceforth $p/q$ is in lowest terms with $q>0$. In \cite{OSratsurg} it is pointed out that the surgered manifold $S^3_{p/q}(K_0)$ can be obtained by an integral (``Morse'') surgery on a rationally null-homologous knot $K$ in a lens space. Here we adopt slightly different conventions from those in \cite{OSratsurg}: write $p = mq-r$ for $0\leq r< q$, and consider the knot $O_{q/r}$ described as the meridian of the surgery curve in $S^3_{q/r}(U)$ (here $U$ is the unknot in $S^3$). Sticking with our conventions, $S^3_{q/r}(U) = - L(q,r)$. If we let $K = K_0\# O_{q/r}$, then an integral surgery on $K$ gives rise to $S^3_{p/q}(K_0)$, in particular the relevant surgery has framing $m$ in the obvious surgery diagram. Note that in this situation $K$ is rationally null-homologous of order $q$ and admits a rational Seifert surface with connected boundary, and therefore (in the notation of Section \ref{cobordismssec}) has $c = 1$ and $t=q$. Moreover, it can be seen that our use of the symbols $m$ and $r$ here is consistent with that previously, in the sense that the framing on $K$ that yields $S^3_{p/q}(K_0)$ is $\lambda_{can} + m\mu$.

%

Now, for sufficiently large integral framing $\lambda$ on a rationally nullhomologous knot $K$, the groups $\hfhat(Y_{\lambda}(K),\tt)$ become standard, i.e., independent of $\lambda$ in an appropriate sense (c.f. \cite[Section 4]{OSratsurg}). In particular this holds for the knot $K_0\subset S^3$, where it is known \cite{OSknot} that for sufficiently large framings $N$ there is an isomorphism
\begin{equation}\label{largesurgiso}
\Psi: \cfhat(S^3_N(K_0), s) \to A_s(K_0).
\end{equation}
Here $A_s(K_0)$ is a certain subquotient complex of the knot Floer chain complex described as $A_s = C_*\{\max(i, j-s) = 0\}$ (c.f. \cite[Theorem 4.4]{OSknot}). It would be appropriate to write $\ahat_s$ for this complex, but since we will have no need for other variants (e.g., $A^+_s$), we omit the extra notation. On the other side, $\cfhat(S^3_N(K_0), s)$ indicates the Floer chain complex in a \spinc structure---indicated by $s\in\zee$---characterized by the property that it is the restriction of a \spinc structure $\s_s$ on the corresponding surgery cobordism satisfying $\langle c_1(\s_s), [S_N]\rangle + N = 2s$. Also relevant for us, the isomorphism \eqref{largesurgiso} is realized by a count of holomorphic triangles in a Heegaard triple-diagram $(\Sigma,\aalpha,\ggamma,\bbeta, w,z)$ describing the surgery cobordism $-W_N$ connecting $S^3_N(K)$ to $S^3$. The relevant set of triangles comprises those inducing the given \spinc structure $s$ on the surgery, and could in principle induce many \spinc structures on $-W_N$ all differing by multiples of $N[\widehat{\Sigma}]$, but for sufficiently large $N$ only one of these can contribute to the stated isomorphism. 

Turning to the rationally null-homologous knot $K = K_0\# O_{q/r}$, recall that combining the K\"unneth principle for knot Floer homology with the large-surgery result just mentioned, and observing that the knot Floer theory for $O_{q/r}$ is essentially trivial, we find that for sufficiently large framings $\lambda$ on $K$ there is an isomorphism $\cfhat(Y_\lambda(K), \tt)\cong A_s(K_0)$ for some integer $s$ depending on the \spinc structure $\tt\in \Spinc(Y_\lambda(K))$ (see \cite{OSratsurg}, Corollary 5.3 and the proof of Theorem 1.1). Our aim is to determine the relationship between $\tt$, as specified in terms of the surgery cobordism $-W_\lambda: Y_\lambda(K)\to -L(q,r)$, the integer $s$, and the mapping cone formula for rational surgery deduced in \cite{OSratsurg} as a consequence of the surgery triangle described above.

We begin by comparing the Heegaard triples for integer surgery on $K_0\subset S^3$ and for integer surgery on $K\subset -L(q,r)$. To avoid confusion let us now write $W_\lambda$ for the surgery cobordism $S^3\to S^3_\lambda(K_0)$, and decorate the ``rational version'' with tildes, as $\tW_{\tlambda}: -L(q,r)\to Y_\tlambda(K)$ with $\tlambda$ a framing on $K$. Starting from the Heegaard diagram for $O_{q/r}$ described in \cite[Proof of Lemma 7.1]{OSratsurg}, we can form a triple diagram describing $-\tW_{\tlambda}$ by connected sum with a corresponding diagram for $W_\lambda$: c.f.\ Figure \ref{ratsurgfig}. 

\begin{figure}
\includegraphics{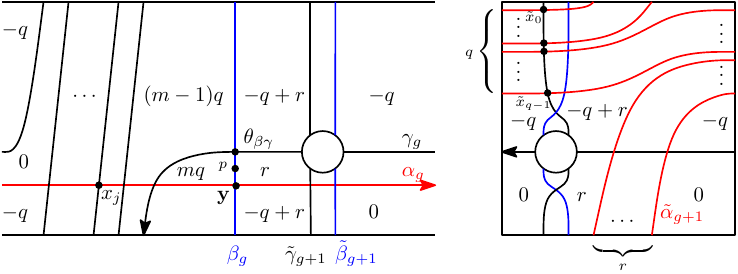}
\caption{\label{ratsurgfig}A diagram for integer surgery on $K = K_0\#O_{q/r}$. The lens space summand appears on the right, where $\tilde\alpha_{g+1}$ is shown as a $(q,r)$ curve intersecting $\tilde\beta_{g+1}$ in points $\tilde x_{0},\ldots \tilde x_{q-1}$. Coefficients for $\P_{\widetilde S}$ are shown; note that $\tilde{\gamma}_{g+1}$ does not appear in $\partial \P_{\widetilde{S}}$.}
\end{figure}

Thus $(\Sigma,\aalpha,\ggamma,\bbeta,w,z)$ describes the cobordism $W_\lambda$, while $(\tSigma, \taalpha,\tggamma,\tbbeta,w,z)$ corresponds to the diagram after connected sum with $-L(q,r)$ and represents $\tW_\tlambda$. (Here $\tSigma$ denotes the connected sum of $\Sigma$ with a torus.) We let $\beta_g$, $\gamma_g$ be the curves appearing in the triple for $W_\lambda$, depicted in Figure \ref{ratsurgfig}, while $\tilde{\beta}_{g+1}$, $\tilde{\gamma}_{g+1}$ are the indicated ``extra'' curves in the diagram for $\tW_\tlambda$. Fix an intersection point $\x\in\Ta\cap \Tg$, and suppose $x_j\in \gamma_g$ is the component appearing in the figure. There are $q$ distinct lifts of $\x$ to a generator in $\mathbb{T}_{\tilde\alpha}\cap \mathbb{T}_{\tilde\beta}$, with components $\tilde x_{i}$ $(0\leq i\leq q-1)$ on $\tilde{\beta}_{g+1}$. Correspondingly, for a ``small triangle'' $\psi_j$ with corner on $\x_j$ in $(\Sigma,\aalpha,\ggamma,\bbeta)$, there are $q$ small triangles $\psi_{j,i}$ in $(\tSigma,\taalpha,\tggamma,\tbbeta)$. In each diagram we can find a triply-periodic domain, $\P_S$ and $\P_{\widetilde{S}}$, representing the generator in second homology of $W_\lambda$ and $\tW_{\tlambda}$ respectively.

\begin{lemma} We have the following relation between the Chern numbers of the triangles $\psi_j$ and $\psi_{j,i}$:
\[
\langle c_1(\s_w(\psi_{j,i})),[\P_{\widetilde{S}}]\rangle - q\langle c_1(\s_w(\psi_j)),[\P_S]\rangle = q+r-1- 2i.
\]
\end{lemma}

\begin{proof} The coefficients of the triply-periodic domains are indicated in Figure \ref{ratsurgfig}. From there the calculation is an exercise with the Chern class formula along the lines of those in  Section \ref{cobordismssec}.
\end{proof}

\begin{corollary}\label{largesurgerycor} Let $\tt_k\in \Spinc(Y_\tlambda(K))$ (for $\tlambda = \lambda_N = \lambda_{can} + N\mu$ sufficiently large) be the \spinc structure obtained as the restriction of $\s_k\in\Spinc(\tW_{\tlambda})$ characterized by 
\[
\langle c_1(\s_k),[S_{\tlambda}]\rangle +P+q -1 = 2k,
\]
where $[S_\tlambda]\in H_2(\tW_\tlambda;\zee)$ is a generator as previously, and $P = Nq-r$. Then there is an isomorphism
\[
\cfhat(Y_\tlambda(K),\tt_k) \cong A_s(K_0),
\]
where $s = \lfloor\frac{k}{q}\rfloor$.
\end{corollary}

\begin{proof} Let $\tx$ be a generator in $(\Sigma,\taalpha,\tggamma,w)$. We can write $\tx = \x_j\times \tilde{x}_i$ for some $i$, where $\x_j\in \Ta\cap\Tg$ is a generator for large integral surgery on $K_0\subset S^3$. The identification \eqref{largesurgiso} is realized by counting triangles in $(\Sigma,\aalpha,\ggamma,\bbeta, w,z)$ that are \spinc equivalent to the evident small triangle $\psi_j$ connecting $\x_j$ to its corresponding generator $\y\in\Ta\cap\Tb$. 

Likewise, if $\tt = \s_w(\tx)$ then the corresponding large-surgery identification $\cfhat(Y_{\tlambda}, \tt) \cong A_{{\xi}}(K)$ (for some relative \spinc structure ${\xi}\in \ul{\Spinc}(-L(q,r), K)$) involves a count of triangles \spinc equivalent to the small triangle $\psi_{j,i}$. By the K\"unneth theorem for the knot Floer chain complex, we have an identification $A_{{\xi}}(K) \cong A_s(K_0)$. Now $s$ is characterized by  the equation $\langle c_1(\s_w(\psi_j)), [S_N]\rangle + N = 2s$. From the lemma,
\begin{eqnarray*}
\langle c_1(\s_w(\psi_{j,i})), [S_{\tilde\lambda}]\rangle &=& q\langle c_1(\s_w(\psi_j)), [S_N]\rangle + q + r - 1-2i\\
&=& 2sq - Nq +q+ r -1 -2i\\
&=& 2((s+1)q -(i+1)) -P-q+1
\end{eqnarray*}
showing $\langle c_1(\s_w(\psi_{j,i})), [S_{\tilde\lambda}]\rangle +P+q-1 = 2k$, where $k = (s+1)q - (i+1)$. Since $0\leq i\leq q-1$, it follows that $s = \lfloor\frac{k}{q}\rfloor$, while $\tt$ is the restriction of $\s_w(\psi_{j,i})$.

\end{proof}

\begin{corollary}\label{inclusioncor} If $\XX_{p/q}(K_0)$ is the rational surgery mapping cone for $K_0\subset S^3$, then the map in homology induced by the inclusion 
\[
(k,{B})\to \XX_{p/q}(K_0)\]
 corresponds, under the identification $H_*(\XX_{p/q}(K_0))\cong \hfhat(S^3_{p/q}(K_0))$, to the map in Floer homology induced by the 2-handle cobordism $W : -L(q,r)\to S^3_{p/q}(K_0)$ equipped with the \spinc structure $\s_k$ characterized by 
\[
\langle c_1(\s_k), [\widetilde{S}]\rangle + p + q - 1 = 2k.
\]
\end{corollary}

\begin{proof} It suffices, by the truncation arguments in \cite{OSratsurg}, to prove a corresponding statement in the context of the surgery exact triangle. Indeed, the cone complex $\XX_{p/q}(K_0)$ is the limit of the cone of $f_2: \cfhat(Y_{\lambda + N\mu})\to \cfhat(Y,\F[C_N])$ as $N$ increases. We saw previously that for a rationally null-homologous knot $K\subset Y = -L(q,r)$, the inclusion of ${B} = \cfhat(Y)$ as the coefficient of $T^k$ corresponds to a map induced by $W$ with a certain \spinc structure. The mapping cone is set up in such a way that the coefficient of $T^k$ becomes (after the truncation argument) the target of the map ${v}_s: (k,A_s(K_0))\to (k, {B})$ where $s = \lfloor\frac{k}{q}\rfloor$. This map, in turn, is induced by the surgery cobordism: we have a commutative diagram (combining the proofs of \cite[Theorem 4.1]{OSratsurg} and Corollary \ref{largesurgerycor}):
\[
\begin{diagram} 
\cfhat(S^3_{p/q}(K_0), \tt_k) & \rTo^{F_{-W_{\lambda_N, \s_k}} }& \cfhat(-L(q,r), \s_k) \\
\dTo^\Psi &&\dTo^=\\
A_s(K_0) & \rTo^{{v_s}}& {B}.
\end{diagram}
\]
Therefore, from the remarks at the end of the previous section, we know the \spinc structure inducing the inclusion map and the one inducing ${v}_s$ have the same ``label'' (the value of the Chern number plus the square of the generator in relative homology). Let us write $\tW_{\lambda_m}: -L(q,r)\to S^3_{p/q}(K)$ for the surgery cobordism, where $p = mq-r$, while $-\tW_{\lambda_N}:S^3_{P/q}\to -L(q,r)$ denotes the large-surgery cobordism turned around (with $P =Nq-r$). If $\s_a\in\Spinc(\tW_{\lambda_m})$ has
\[
\langle c_1(\s_a), [F_{\lambda_m}]\rangle + [F_{\lambda_m}]^2 = a,
\]
we can write $[F_{\lambda_m}]^2 = m-\frac{r}{q}$ by Lemma \ref{selfintlemma}. Then the map induced by $\s_a$ corresponds to the inclusion of ${B}$ in the mapping cone as the target of the map induced by the \spinc structure $\tilde{\s}_a$ having the same label, i.e., having
\[
\langle c_1(\tilde{\s}_a), [F_{\lambda_N}]\rangle + [F_{\lambda_N}]^2 = a.
\]
Since $[F_{\lambda_N}]^2 = N- \frac{r}{q}$, we can say
\[
qa = \langle c_1(\tilde{\s}_a), [S_{\lambda_N}]\rangle + Nq-r = \langle c_1(\s_a), [S_{\lambda_m}]\rangle + mq -r.
\]
Hence, 
\[
\langle c_1(\tilde{\s}_a), [S_{\lambda_N}]\rangle + P + q -1 = \langle c_1(\s_a), [S_{\lambda_m}]\rangle + p+q-1,
\]
so that the convention determining the index $k$ in the mapping cone $\XX_{p/q}(K_0)$ agrees with the identification between the large-surgery complexes $A_{\xi}(K)$ and $A_s(K_0)$ obtained in the previous corollary.

\end{proof}

\bibliography{naturality}

\end{document}